\documentclass[reqno]{amsart}

\usepackage{url}
\usepackage{xcolor}
\definecolor{darkgreen}{rgb}{0,0.5,0}

\usepackage[T1]{fontenc}
\usepackage{amsmath,amsfonts,amssymb,amsthm}
\usepackage{mathtools}
\usepackage{comment}
\usepackage{pgf,tikz}
\usetikzlibrary{matrix,arrows,calc}
\usepackage{tikz-cd}
\usepackage{enumitem}
\usepackage[colorinlistoftodos]{todonotes}
\usepackage{cancel}
\usepackage{lmodern}
\usepackage{aligned-overset}

\usepackage[
	backend=bibtex,  %
	style=alphabetic,
	minalphanames=3,
	maxnames=99,
	maxalphanames=4,
	giveninits=true,
	isbn=false,
	doi=true,
]{biblatex}
\usepackage[
	colorlinks, 
	citecolor=darkgreen,
]{hyperref}
\usepackage{cleveref}

\newcommand{\sm}{\mathrm{sm}}

\numberwithin{equation}{section}

\newtheorem{thmABC}{Theorem}

\newtheorem{thm}{Theorem}
\newtheorem{prop}[thm]{Proposition}
\newtheorem{lemma}[thm]{Lemma}

\newtheorem{cor}[thm]{Corollary}

\theoremstyle{remark}
\newtheorem{rem}[thm]{Remark}

\newtheorem{example}[thm]{Example}

\theoremstyle{definition}
\newtheorem{defn}[thm]{Definition}

\numberwithin{thm}{section}

\AddToHook{env/defn/begin}{\crefalias{thm}{defn}}
\AddToHook{env/prop/begin}{\crefalias{thm}{prop}}
\AddToHook{env/lemma/begin}{\crefalias{thm}{lemma}}
\AddToHook{env/conjecture/begin}{\crefalias{thm}{conjecture}}
\AddToHook{env/cor/begin}{\crefalias{thm}{cor}}
\AddToHook{env/rem/begin}{\crefalias{thm}{rem}}
\AddToHook{env/remarks/begin}{\crefalias{thm}{remarks}}
\AddToHook{env/example/begin}{\crefalias{thm}{example}}
\AddToHook{env/constr/begin}{\crefalias{thm}{constr}}

\newcommand{\bA}{\mathbb A}

\newcommand{\bC}{\mathbb C}

\newcommand{\cC}{\mathcal C}
\newcommand{\cl}{\mathrm{cl}}
\newcommand{\Cl}{\mathrm{Cl}}
\newcommand{\coh}{\mathrm{coh}}
\newcommand{\cpt}{\mathrm{cpt}}
\newcommand{\rd}{\mathrm d}
\newcommand{\cE}{\mathcal{E}}

\newcommand{\cF}{\mathcal F}

\newcommand{\bG}{\mathbb G}

\newcommand{\length}{\mathrm{length}}

\newcommand{\rH}{\mathrm H}

\newcommand{\fm}{\mathfrak m}

\newcommand{\fp}{\mathfrak p}
\newcommand{\fq}{\mathfrak q}

\newcommand{\bP}{\mathbb P}

\newcommand{\cP}{\mathcal P}
\newcommand{\bQ}{\mathbb Q}
\newcommand{\cQ}{\mathcal Q}
\newcommand{\Qbar}{{\overline{\mathbb Q}}}
\newcommand{\bZ}{\mathbb Z}
\newcommand{\cO}{\mathcal{O}}
\newcommand{\cX}{\mathcal{X}}

\newcommand{\cY}{\mathcal{Y}}
\newcommand{\cZ}{\mathcal{Z}}

\newcommand{\cJ}{\mathcal{J}}
\newcommand{\bR}{\mathbb{R}}

\newcommand{\fS}{\mathfrak{S}}

\newcommand{\Tr}{\mathrm{Tr}}

\newcommand{\cK}{\mathcal K}
\newcommand{\cH}{\mathcal H}
\newcommand{\cV}{\mathcal V}

\newcommand{\red}{{\mathrm{red}}}
\newcommand{\csp}{{\mathrm{csp}}}

\newcommand{\lto}{\longrightarrow}
\renewcommand{\div}{\mathrm{div}}

\newcommand{\Kbar}{\overline{K}}

\DeclareMathOperator{\AJ}{AJ}

\DeclareMathOperator{\coker}{coker}

\DeclareMathOperator{\cD}{\mathcal{D}}

\DeclareMathOperator{\Div}{Div}

\DeclareMathOperator{\bF}{\mathbb{F}}
\DeclareMathOperator{\Gal}{Gal}

\DeclareMathOperator{\dR}{dR}

\DeclareMathOperator{\Lie}{Lie}

\DeclareMathOperator{\NS}{NS}
\DeclareMathOperator{\Res}{Res}

\DeclareMathOperator{\rank}{\mathrm{rk}}

\DeclareMathOperator{\Sel}{Sel}
\DeclareMathOperator{\Spec}{Spec}

\DeclareMathOperator{\modulo}{mod}
\DeclareMathOperator{\ord}{ord}
\DeclareMathOperator{\mult}{mult}
\DeclareMathOperator{\tf}{tf}
\DeclareMathOperator{\tors}{tors}
\DeclareMathOperator{\hol}{hol}

\title{Affine Chabauty I}
\author{Marius Leonhardt}
\address{Marius Leonhardt,
	Department of Mathematics and Statistics,
	Boston University,
	665 Commonwealth Ave,
	Boston, MA 02215,
	USA}
\email{mleonhar@bu.edu}

\author{Martin Lüdtke}
\address{Martin Lüdtke,
	Institut für Mathematik,
	Carl von Ossietzky Universität Oldenburg,
	26111 Oldenburg,
	Germany
}
\email{martin.luedtke@uol.de}

\begin{document}

\thispagestyle{empty}

\begin{abstract}
	We prove finiteness and give an explicit upper bound on the number of $S$-integral points on affine curves satisfying a certain rank-genus inequality. 
	We achieve this by developing an analogue of the Chabauty method, embedding the curve into its generalised Jacobian and bounding the Abel--Jacobi image of the $S$-integral points using arithmetic intersection theory. Our results also provide the foundations for an algorithm to determine the set of $S$-integral points on affine curves presented in a follow-up article.
\end{abstract}

\maketitle
\tableofcontents

\section{Introduction}
\label{sec: introduction}

Let $Y/\bQ$ be a smooth affine curve and let $\cY/\bZ_S$ be a
regular model of $Y$ over the ring of $S$-integers for some finite set of primes~$S$.
If $Y$ is hyperbolic, the set of $S$-integral points $\cY(\bZ_S)$ is finite by the theorems of Siegel, Mahler, and Faltings.
However, the natural questions of finding the points in $\cY(\bZ_S)$ or bounding the size of this set remain difficult open problems in general.
We present a new approach addressing these questions, which results in an explicit upper bound on $\#\cY(\bZ_S)$ for curves satisfying the hypothesis
\begin{equation}
\label{eq:uniform-chabauty-condition-intro}
r + \#S < g + \#|D| + n_2(D) - 1,
\end{equation}
an inequality involving the genus $g$, the Mordell--Weil rank $r$, the cardinality of~$S$, the number of closed points and the number of conjugate pairs of complex points in the boundary $D$ of $Y$.
Our method is an $S$-integral analogue of the method of Chabauty--Coleman \cite{Cha41, Col85}.
Here, we embed $Y$ into its \emph{generalised Jacobian} $J_Y$, choose an auxiliary prime $p \not\in S$ (possibly of bad reduction), and cut out $\cY(\bZ_S)$ inside $\cY(\bZ_p)$ by $p$-adic integrals of \emph{logarithmic differentials}, i.e.\ meromorphic differentials on a compactification of~$Y$ with at worst simple poles at $D$.
The existence of suitable log differentials is guaranteed by assumption~\eqref{eq:uniform-chabauty-condition-intro}, which ensures that the image of $\cY(\bZ_S)$ inside $J_Y(\bQ)$ is contained in a finite union of translates of subgroups of controllable rank.
We then bound the size of $\cY(\bZ_S)$ by estimating the number of zeros of $p$-adic integrals of log differentials.
In the follow-up article \cite{affchab2}, we present an algorithm to determine the log differentials and the zeros of their integrals in practice, thus computing a finite subset of $\cY(\bZ_p)$ containing $\cY(\bZ_S)$.

\subsection{Main results}

In this introduction, we work with curves over $\bQ$, but in the body of the paper we formulate our results over number fields.
Let $Y/\bQ$ be given as $X\smallsetminus D$ where $X/\bQ$ is a smooth projective curve and $D \neq \emptyset$ is a finite set of closed points called \emph{cusps}. Assume $Y(\bQ) \neq \emptyset$ and fix a base point $P_0\in Y(\bQ)$.
Let $S$ be a finite set of primes.
Let~$\cX$ be a regular model of $X$ over the ring~$\bZ_S$ of $S$-integers, i.e., a regular, flat, projective $\bZ_S$-scheme with an isomorphism $\cX_{\bQ} \cong X$ \cite[Definition~10.1.1]{liu2006algebraic}.
Let $\cD$ be the closure of~$D$ in~$\cX$ and set $\cY \coloneqq \cX \smallsetminus \cD$.
We use the following notation, which will be kept throughout this paper:
\begin{itemize}
	\item $r \coloneqq \rank J(\bQ)$ the Mordell--Weil rank of the Jacobian $J$ of~$X$;
	\item $g$ the genus of~$X$;
	\item $\#|D| > 0$ the number of cusps;
	\item $n \coloneqq \#D(\Qbar)>0$ the number of geometric cusps;
	\item write $n = n_1(D) + 2n_2(D)$ with $n_1(D) \coloneqq \#D(\bR)$ the number of real cusps and $n_2(D)$ the number of conjugate pairs of complex cusps.
\end{itemize}

We partition the $S$-integral points
\[
\cY(\bZ_S) = \coprod_{\Sigma} \cY(\bZ_S)_{\Sigma},
\]
where $\Sigma$ runs through the finitely many $S$-integral \emph{reduction types} of $\cY$, see \Cref{def:S-integral-reduction-type}.
Roughly, an $S$-integral reduction type $\Sigma = (\Sigma_{\ell})_{\ell}$ specifies for each prime $\ell \not\in S$ a component $\Sigma_{\ell}$ of the mod-$\ell$ fibre~$\cX_{\ell}$, while for $\ell \in S$ it specifies either a component of the mod-$\ell$ fibre or the mod-$\ell$ reduction of a cusp on a suitable extension of~$\cX$ to a regular model over $\bZ$.
The subset $\cY(\bZ_S)_{\Sigma} \subseteq \cY(\bZ_S)$ consists of those $S$-integral points that reduce onto $\Sigma_{\ell}$ modulo $\ell$ for all~$\ell$.
Our main result is:

\begin{thmABC}
	\label{thm:existence-of-chabauty-function-intro}
	Let $\Sigma$ be an $S$-integral reduction type and let $C(\Sigma)\subseteq S$ denote the set of those $\ell \in S$ where $\Sigma_{\ell}$ is cuspidal, see \Cref{def:C-Sigma}.
	Let $p \not\in S$ be a prime.
	Assume
	\begin{equation}
	\label{eq:chabauty-condition-intro}
	r + \#C(\Sigma) < g + \#|D| + n_2(D) - 1.
	\end{equation}
	Then there exist a non-zero log differential $0 \neq \omega \in \rH^0(X_{\bQ_p}, \Omega^1(D))$ and a constant $c \in \bQ_p$ such that the function $\rho\colon \cY(\bZ_p) \to \bQ_p$ given by
	\begin{equation}
	\label{eq:rho-function-intro}
	\rho(P) \coloneqq \int_{P_0}^P \omega - c
	\end{equation}
	vanishes on $\cY(\bZ_S)_{\Sigma}$.
\end{thmABC}

\Cref{thm:existence-of-chabauty-function-intro} is a special case of \Cref{thm:existence-of-chabauty-function-nf}, which holds over general number fields. As \eqref{eq:uniform-chabauty-condition-intro} implies \eqref{eq:chabauty-condition-intro} for every $S$-integral reduction type, we deduce:

\begin{thmABC}
	\label{thm:fin-many-Chab-functions}
	If the inequality \eqref{eq:uniform-chabauty-condition-intro} holds, 
	then $\cY(\bZ_S)$ is contained in a finite union of zero loci of $p$-adic analytic functions on $\cY(\bZ_p)$. %
\end{thmABC}

The $p$-adic analytic functions~\eqref{eq:rho-function-intro} are integrals of log differentials plus a constant, and by analysing their Newton polygons we arrive at bounds for the number of $S$-integral points.

\begin{thmABC}
	\label{thm:bound}
	If \eqref{eq:uniform-chabauty-condition-intro} holds and $p>2g+n$, then
	\[ \# \cY(\bZ_S) \leq (\#\cY^{\sm}_{p}(\bF_p) + 2g - 2 + n) \prod_{\ell\in S} (n_\ell+\# (\cX_{{\ell}}^{\sm} \cap \cD)(\bF_\ell)) \prod_{\ell \not\in S \cup \{p\}} n_{\ell}, \]
	where $n_\ell$ denotes the number of components of $\cX_{{\ell}}\smallsetminus \cD$ that contain a smooth $\bF_\ell$-point and $\mathcal{A}^{\sm}_{{\ell}}$ denotes the smooth locus of the special fibre of $\mathcal{A}\in\{\cX,\cY\}$ at $\ell$.
\end{thmABC}

In fact, different reduction types might give rise to the same Chabauty function $\rho$ in \eqref{eq:rho-function-intro}, so the factor in \Cref{thm:bound} coming from summing over all reduction types can be improved, see \Cref{rem:improved-bound}.
We demonstrate this in \Cref{thm:even-deg-hyperell-bound} by exhibiting a large class of even degree hyperelliptic curves $Y\colon y^2 = f(x)$ for which the rank condition $r=g$ implies the simple bound
\[
	\#\cY(\bZ)\leq \#\cY(\bF_p)+2g
\]
for any prime $p > 2g + 2$ of good reduction.
We also include examples where the bound is sharp.

\begin{rem}\label{rem:circle}
	The results in this article are inspired by the following circle of ideas:
	
	\begin{enumerate}
		\item Chabauty's method for projective curves \cite{Cha41, Col85, mcpoonen};
		\item the ``linear Chabauty'' (or \emph{depth~1}) part of Kim's nonabelian Chabauty method for affine curves \cite{Kim05, Kim09, BDCKW};
		\item the Chabauty--Skolem method \cite{skolem, poonen:skolem, triantafillou:ROS-Chabauty} for affine curves.
	\end{enumerate}
	
	Our approach is closest to (1), but replaces the Jacobian by the generalised Jacobian, holomorphic differentials by logarithmic differentials, and uses arithmetic intersection theory to constrain the Abel--Jacobi image, see \Cref{sec:strategy}.
	
	Using approach (2), we previously showed in joint work with Steffen Müller that hypothesis \eqref{eq:uniform-chabauty-condition-intro} implies finiteness of the depth-1 Chabauty--Kim locus \cite[Theorem~A(1)+Remark~1.4]{LLM:LQChabAffine}, and obtained bounds under a more restrictive hypothesis \cite[Remark 1.2]{LLM:LQChabAffine}.
	The present manuscript improves upon this in several aspects: 
	the bound in \Cref{thm:bound} is smaller, we only need hypothesis \eqref{eq:uniform-chabauty-condition-intro} to prove both finiteness and a bound, we allow $p$ to be a prime of bad reduction for~$\cX$, our method works over general number fields, and we determine the type of Coleman functions on $\cY(\bZ_p)$ vanishing on $S$-integral points.
	Crucial for these improvements is the geometry and arithmetic of the generalised Jacobian.
	Moreover, the functions $\rho$ in \Cref{thm:existence-of-chabauty-function-intro} can be algorithmically determined in practice, and by computing their zeros one finds a finite subset of $\cY(\bZ_p)$ containing $\cY(\bZ_S)$. In the case of an even-degree hyperelliptic curve over $\bQ$ and $S = \emptyset$, such a method was recently developed in \cite{GM:LinearQuadraticChabauty} using $p$-adic Coleman--Gross heights. This ``Linear Quadratic Chabauty'' method arises as a special case of our method, see Section~\ref{sec:recovering-linear-quadratic-chabauty}.
	
	Our method can also be seen as a refined version of (3) in the sense of \cite{BD:refined}.
	Namely, the idea of (3) is to embed $\cY$ into 
	a suitable model $\cJ/\bZ_S$ of the generalised Jacobian $J_Y$. Finiteness of $\cY(\bZ_S)$ is ensured by 
	\begin{equation*}
	\rank \cJ(\bZ_S) < \dim J_Y,
	\end{equation*}
	which is usually stricter than our Affine Chabauty Condition~\eqref{eq:uniform-chabauty-condition-intro}, as can be seen by looking at the complement in~$\bP^1$ of $n \geq 2$ rational points, for example. The reason is ultimately that we get a more refined control on the Abel--Jacobi image in the generalised Jacobian by also incorporating local conditions at primes inside~$S$ via the notion of $S$-integral reduction type.
\end{rem}

\begin{rem}
	This article does not touch upon the ``quadratic Chabauty'' results in \cite{LLM:LQChabAffine}.
	There, the hypothesis for finiteness of the depth-2 locus $\cY(\bZ_p)_{S,2}^{\mathrm{BD}}$ has an additional term of $\rho_f = \rank \NS(J) + \rank \NS(J_{\Qbar})^{\sigma=-1}$ on the right hand side of~\eqref{eq:uniform-chabauty-condition-intro}. %
	We believe the methods of this article to be a first step in the direction of making affine quadratic Chabauty explicit as well, and aim to pursue this line of investigation in the future.
\end{rem}

\subsection{Strategy}
\label{sec:strategy}

Let $J_Y$ denote the generalised Jacobian of $(X,D)$, see \cite[Chapter V]{serrealggp}.
We have an Abel--Jacobi map $\AJ_{P_0}\colon Y\to J_Y$ associated to the base point $P_0$, which yields the Affine Chabauty diagram
\begin{equation}
\label{eq:chabauty-diagram0}
\begin{tikzcd}
\cY(\bZ_S) \dar["\AJ_{P_0}"] \rar[hook] & \cY(\bZ_p) \dar["\AJ_{P_0}"] \drar["\int_{P_0}"]& \\
J_Y(\bQ) \rar[hook] & J_Y(\bQ_p) \rar["\log_{J_Y}"] & \rH^0(X_{\bQ_p}, \Omega^1(D))^\vee,
\end{tikzcd}
\end{equation}
where $\Omega^1(D)$ is the module of logarithmic differentials on $(X,D)$.
Unlike the situation for projective curves, where we have the Jacobian $J$ in the bottom row whose Mordell--Weil group $J(\bQ)$ is finitely generated, the generalised Jacobian $J_Y$ is only a semi-abelian variety and the group $J_Y(\bQ)$ is almost never finitely generated.
The crucial observation is that if we only look at $S$-integral points of a given reduction type $\Sigma$, then the Abel--Jacobi image of $\cY(\bZ_S)_{\Sigma}$ inside $J_Y(\bQ)$ does in fact lie in a translate of a certain finitely generated subgroup.
The hypothesis \eqref{eq:chabauty-condition-intro} ensures that the rank of this subgroup is strictly less than $\dim_{\bQ_p} \rH^0(X_{\bQ_p}, \Omega^1(D)) = g+n-1$, and from this point onwards we deduce \Cref{thm:existence-of-chabauty-function-intro} by the ``usual Chabauty logic''.
Namely, the image of this subgroup under $\log_{J_Y}$ is contained in a proper $\bQ_p$-subspace of $\rH^0(X_{\bQ_p}, \Omega^1(D))^\vee$, so there is a non-trivial log differential $\omega$ vanishing on this image. The map $\cY(\bZ_p) \to \bQ_p$, $P \mapsto \int_{P_0}^P \omega$ is then constant on $\cY(\bZ_S)_{\Sigma}$.

It remains to explain how we restrict the Abel--Jacobi image of $\cY(\bZ_S)_{\Sigma}$ inside $J_Y(\bQ)$. We achieve this by constructing a homomorphism $\sigma\colon J_Y(\bQ) \to V_D$ to a certain $\bQ$-vector space~$V_D$ with the property that $\ker(\sigma)$ has finite rank and the image of $\cY(\bZ_S)_{\Sigma}$ in~$V_D$ is contained in a translate of a subgroup of small rank. The map~$\sigma$, whose construction is inspired by local Néron--Tate and Coleman--Gross heights, is defined using arithmetic intersection theory on a regular model. For simplicity, assume that~$\cX$ is a regular model of~$X$ over~$\bZ$ such that the closure~$\cD$ of~$D$ in~$\cX$ is normal.
Consider the group $Z_0(\cD)$ of zero-cycles on~$\cD$ and let $\pi$ denote the map $\cD \to\Spec(\bZ)$. We define the \emph{$D$-intersection map}
\[
\sigma \colon \Div^0(Y) \to V_D \coloneqq Z_0(\cD)/\pi^* Z_0(\Spec(\bZ)) \otimes_{\bZ} \bQ, \quad F \mapsto \Psi(F).\cD.
\]
Here we have $\Psi(F)=\cF+\Phi(F)$ with $\cF$ the extension of $F$ to a horizontal divisor on $\cX$ obtained by taking the closure of points, $\Phi(F)$ is a sum of vertical $\bQ$-divisors on $\cX$ such that $\Psi(F)$ has $\ell$-intersection number~$0$ with all vertical divisors above~$\ell$ for all primes~$\ell$, and 
\[\Psi(F).\cD = \sum_x i_x(\Psi(F),\cD) [x]\]
denotes the intersection cycle of the divisors $\Psi(F)$ and~$\cD$ on~$\cX$. 
We show that $\sigma$ descends to $J_Y(\bQ)$ and is independent of the choice of regular model $\cX$. The map $\sigma$ naturally decomposes as a direct sum of local $D$-intersection maps $\sigma_{\ell}\colon J_Y(\bQ_{\ell}) \to V_{D,\ell}$. The construction is inspired by the Arakelov-theoretic interpretation of local Néron symbols due to Faltings and Hrijlac \cite{faltings:calculus, hriljac} which appear as the local components of the Néron--Tate height and of the $p$-adic Coleman--Gross height at places away from~$p$ \cite{CG89}.

In the local setting, we show that the $D$-intersection map provides a $\bQ$-form of the $p$-adic étale Kummer map of the generalised Jacobian (\Cref{kummer-comparison}). This can be seen as an ``independence of $p$'' result and suggests an interpretation of $\sigma_{\ell}$ as a \emph{motivic Kummer map} for $J_Y$, which may be of independent interest.

In a sense, the intersection cycle $\sigma(P-P_0)$ records where and how badly the points $P$ and $P_0$ reduce to the cusps.
If we restrict to $S$-integral points $P$ of a fixed reduction type, we can show that the composition $\sigma_{\ell}\circ\AJ_{P_0}\colon \cY(\bZ_S)_\Sigma \to V_{D,\ell}$ lands in a certain subset $\fS_{\ell}(P_0,\Sigma_{\ell}) \subseteq V_{D,\ell}$, which is either a point or a translate of a rank $1$ subgroup (\Cref{thm:abel-jacobi-rational-reduction-type}).
Putting things together globally, we show in \Cref{thm:map-to-selmer-set} that the Abel--Jacobi image of $\cY(\bZ_S)_\Sigma$ lies in the \emph{global Selmer set}
\[
\Sel(P_0,\Sigma) \coloneqq \sigma^{-1}\left(\prod_{\ell} \fS_{\ell}(P_0,\Sigma_{\ell})\right) \subseteq J_Y(\bQ).
\]
By determining the rank of the kernel of $\sigma$ (\Cref{thm:ker-sigma}), we show that $\Sel(P_0,\Sigma)$ is a translate of a subgroup of a certain rank (\Cref{thm:rank-of-selmer-set}).
Hypothesis \eqref{eq:chabauty-condition-intro} then precisely ensures that this rank is strictly less than $\dim_{\bQ_p} \rH^0(X_{\bQ_p}, \Omega^1(D)) = g+n-1$.

\subsection{Structure of the paper}

In \Cref{sec:local-theory}, we define the $D$-intersection map $\sigma_{\fm}$ in a local setting, namely for affine curves over discretely valued fields. 
We introduce $D$-transversal models, integral and rational reduction types $\Sigma_{\fm}$ and show that the composition $\sigma_{\fm}\circ\AJ_{P_0}$ restricted to points of a given reduction type lands in the subset $\fS_{\fm}(P_0,\Sigma_{\fm})$.
We also compare the $D$-intersection map with the local Kummer map.
In \Cref{sec:global-theory}, we define the global $D$-intersection map $\sigma$ for an affine curve over a number field, and we determine the rank of $\ker(\sigma)$. We define $S$-integral reduction types $\Sigma$ and the global Selmer set $\Sel(P_0,\Sigma)$.
We then show that the Abel--Jacobi images of $S$-integral points of reduction type $\Sigma$ are contained in $\Sel(P_0,\Sigma)$ and bound the rank of $\Sel(P_0,\Sigma)$.
In \Cref{sec:integration}, we summarise the theory of $p$-adic abelian integrals of log differentials. %
In \Cref{sec:chabauty-functions} we prove \Cref{thm:existence-of-chabauty-function-intro} and describe an $S$-integral version of Siksek's Restriction of Scalars Chabauty \cite{siksek:number-field-chabauty}.
We then proceed to prove \Cref{thm:bound} in \Cref{sec:bounds}, %
and end this paper in \Cref{sec:examples} with three applications of our methods to even degree hyperelliptic curves: recovering the Linear Quadratic Chabauty method of~\cite{GM:LinearQuadraticChabauty}, an improved bound on the number of integral points of such curves, and the complete determination of the $\bZ[\zeta_3]$-points of a curve of genus~2 and Mordell--Weil rank~2.

\subsection*{Acknowledgements}

We thank Jennifer Balakrishnan, Alex Betts, Stevan Gajovi\'c, Sachi Hashi\-moto, Enis Kaya, Bjorn Poonen, Pim Spelier, and Elie Studnia for discussions on this paper. We thank Steffen Müller for helpful dicussions and comments on a previous version of this paper. 
The first author is supported by a Walter Benjamin Scholarship from the Deutsche Forschungsgemeinschaft (DFG, German Research Foundation), project number LE 5634/1, and also acknowledges support from the DFG through TRR 326 Geometry and Arithmetic of Uniformized Structures, project number 444845124. The second author is supported by a Minerva Fellowship of the Minerva Stiftung Gesellschaft für die Forschung mbH and also acknowledges support through a guest postdoc fellowship at the Max Planck Institute for Mathematics in Bonn and through an NWO Grant, project number VI.Vidi.192.106.

\section{Local theory: The $D$-intersection map}
\label{sec:local-theory}

We recall the definition of the generalised Jacobian and analyse the image of the Abel--Jacobi map for an affine curve over a discretely valued field. The key innovation is the construction of the so-called \emph{$D$-intersection map} from $J_Y(K)$ to a finite-dimensional $\bQ$-vector space $V_{D,\fm}$ in \Cref{sec:intersection map}. This map is used to constrain the image of integral and rational points of fixed reduction type under the Abel--Jacobi map. Subsections \ref{sec:normal-models} and \ref{sec:rational-reduction-types} are only relevant for computing $S$-integral points with $S \neq \emptyset$. On a first reading we recommend skipping §\ref{sec:phi}, §\ref{sec:comparison-with-kummer-map}, and §\ref{sec:local-neron-symbols}, referring back to them as needed.

\subsection{The generalised Jacobian}
\label{sec:generalised-jacobian}

Let $K$ be a field, let $X/K$ be a smooth, projective curve, let $D \neq \emptyset$ be a finite set of closed points of~$X$, and set $Y = X \smallsetminus D$. The Jacobian of~$X$ is denoted by~$J$. The \emph{generalised Jacobian}~$J_Y$ of~$Y$ \cite[Ch.~V]{serrealggp} is a semi-abelian variety over~$K$ sitting in a short exact sequence
\begin{equation}
	\label{eq:generalised-jacobian-exact-sequence}
	0 \to T_D \to J_Y \to J \to 0
\end{equation}
with the Jacobian of~$X$ on the right and a torus $T_D$ on the left. The toric part is given by the Weil restriction $T_D = (\Res_{D/K} \bG_{m,D}) / \bG_m$, whose $K$-points are given by
\[ T_D(K) = \cO(D)^\times / K^\times = \bigl(\prod_{Q \in |D|} k(Q)^\times \bigr) \bigm/ K^\times. \]
Here $k(Q)$ denotes the residue field of the cusp~$Q$ and the group $K^\times$ is embedded diagonally. %
The $K$-points of the generalised Jacobian $J_Y$ have the following interpretation: two divisors $E, E'$ on~$Y$ are called \emph{$D$-equivalent} if there is a rational function $f \in k(X)^\times$ with $f\vert_D = 1$ such that $E - E' = \div(f)$. Now $J_Y(K)$ is the group of $D$-equivalence classes of divisors of degree zero on~$Y$:
\[ J_Y(K) = \frac{\Div^0(Y)}{\left\{\div(f) \bigm \vert f \in k(X)^\times, f\vert_D = 1 \right\}}. \]
The map $J_Y \to J$ in the exact sequence~\eqref{eq:generalised-jacobian-exact-sequence} is given by viewing degree zero divisors on~$Y$ as divisors on~$X$ and replacing $D$-equivalence with the finer notion of linear equivalence. 
The inclusion $T_D \to J_Y$ is given on $K$-points as follows: given $a = (a_Q)_Q \in \prod_{Q \in |D|} k(Q)^\times$, let~$f$ be any rational function on~$X$ realising the values $f(Q) = a_Q$ for $Q \in |D|$. Such a function exists by approximation. Then $a \mapsto \div(f)$ in $J_Y(K)$, which is well-defined and descends to $T_D(K)$. 
Given a $K$-rational point $P_0 \in Y(K)$, we have the \emph{Abel--Jacobi map}
\[ \AJ_{P_0}\colon Y \to J_Y. \]
On $K$-points, it is given by mapping $P \in Y(K)$ to the $D$-equivalence class of $[P] - [P_0]$.

\subsection{Intersection numbers}
\label{sec:intersection-numbers}

For the remainder of this section, let $(R,\fm,k)$ be a discrete valuation ring with field of fractions~$K$, let $X$ be a smooth projective curve over~$K$, let $D \neq \emptyset$ be a finite set of closed points of $X$ and set $Y = X \smallsetminus D$. 
In order to construct the $D$-intersection map $\sigma_{\fm}\colon J_Y(K) \to V_{D,\fm}$, we choose a regular model $\cX$ of~$X$ over~$R$, i.e.\ an integral, regular, projective and flat $R$-scheme with an isomorphism $\cX_K \cong X$ \cite[Definition~10.1.1]{liu2006algebraic}. We will show later that the map $\sigma_{\fm}$ does not depend on the chosen model.
Recall from \cite[§III.2]{lang:arakelov} the definition of intersection numbers on the arithmetic surface~$\cX$. If $E,F$ are effective divisors on~$\cX$ without common component and $x$ is a closed point of~$\cX_{\fm}$, the \emph{intersection number} of $E$ and $F$ at~$x$ is defined as
\[ i_x(E,F) \coloneqq \length_{\cO_{\cX,x}}(\cO_{\cX,x}/(f,g)), \]
where~$f$ and~$g$ are defining equations for $E$ and $F$ in a neighbourhood of~$x$. The symbol $i_x$ is bilinear, hence can be extended to not necessarily effective divisors and to $\bQ$-divisors as long as they have no common component in their support. 
The \emph{intersection cycle} of such divisors is defined as the zero-cycle
\[ E . F \coloneqq \sum_x i_x(E,F) [x], \]
on~$\cX_{\fm}$. We also have the $\fm$-intersection number
\begin{equation}
	\label{eq:m-intersection-number}
	i_{\fm}(E,F) \coloneqq E\cdot F \coloneqq \deg_k(E.F) = \sum_x i_x(E,F) [k(x) : k],
\end{equation}
where the sum is over all closed points of~$\cX_{\fm}$. By replacing a divisor with a linearly equivalent one, one can also define the self-intersection number $i_{\fm}(E,E)$ of a vertical prime divisor~$E$. Hence, the symbol $i_{\fm}$ extends to all pairs of divisors on~$\cX$ with no \emph{horizontal} common component \cite[§III.3]{lang:arakelov}.

\begin{defn}
	\label{def:normalisation}
	Let $\tilde \cD \to \Spec(R)$ denote the \emph{normalisation} of $\Spec(R)$ in~$\cO(D)$. Concretely, if $D = \{Q_1,\ldots,Q_r\}$ and if $R_i$ is the integral closure of~$R$ in the residue field~$k(Q_i)$, then 
	\[ \tilde \cD = \coprod_{i=1}^r \Spec(R_i). \]
	For $Q \in |D|$, we write $\tilde \cQ$ for the normalisation of $\Spec(R)$ in~$k(Q)$.
\end{defn}

We need a generalisation of intersection numbers that allows us to ``intersect'' divisors on $\cX$ with $\tilde{\cD}$.
By the valuative criterion of properness, the inclusion $D \hookrightarrow \cX$ extends to a morphism $\nu\colon \tilde \cD \to \cX$. 
Its image is the closure $\cD$ of $D$ in~$\cX$.
If $\cD$ is normal, i.e.\ the Zariski closures of the points $Q \in |D|$ in~$\cX$ are normal and pairwise disjoint, then $\nu$ is just the inclusion of the horizontal divisor $\cD \hookrightarrow \cX$. In general, however, $\nu$ is not a closed immersion. Nevertheless, by a slight abuse of notation, for any closed point $x \in \tilde \cD$ and effective divisor~$E$ on~$\cX$ without common component with~$\cD$, we write
\begin{equation}
	\label{eq:intersection-Dtilde}
	i_x(E,\tilde \cD) \coloneqq \ord_x(\nu^* f)
\end{equation}
with $f$ a defining equation for~$E$ in a neighbourhood of $\nu(x)$. From \cite[Lemma~III.2.2]{lang:arakelov} it follows that this agrees with the ordinary intersection number on~$\cX$ when $\tilde \cD = \cD$. Again, the definition extends to not necessarily effective divisors and to~$\bQ$-divisors~$E$. We also have the intersection cycle
\[ E.\tilde \cD \coloneqq \sum_x i_x(E,\tilde \cD) [x] \in Z_0(\tilde \cD), \]
with $x$ running over all closed points of~$\tilde \cD$.

\subsection{The $D$-intersection map}
\label{sec:intersection map}

We now define the vector space $V_{D,\fm}$ and construct the $D$-intersection map $\sigma_{\fm}\colon J_Y(K) \to V_{D,\fm}$.
Each of the rings $\cO(\tilde \cQ)$ for $Q \in |D|$ is a Dedekind domain by the Krull--Akizuki theorem \cite[Theorem~11.7]{matsumura}. Moreover, the special fibre $\tilde \cQ_{\fm} = \cQ \otimes_R k$ is finite over~$k$. Its points are in bijection with discrete valuations on $k(Q)$ extending the valuation on~$K$. For a closed point $x$ of $\tilde \cQ$, we write $\ord_x\colon k(Q)^\times \twoheadrightarrow \bZ$ for the associated normalised discrete valuation. Denote by $Z_0(\tilde \cD)$ the group of zero-cycles on~$\tilde \cD$, i.e.\ formal linear combinations $\sum n_x [x]$ with $n_x \in \bZ$ and $x$ running through the closed points of~$\tilde \cD$. The special fibre $\tilde \cD_{\fm}$ defines an element 
\[ [\tilde \cD_{\fm}] \coloneqq \sum_x e_x [x] \]
of $Z_0(\tilde \cD)$, with $e_x$ being the ramification index of $\tilde \cD \to \Spec(R)$ at~$x$.

\begin{defn}
	\label{defn:Vm}
	Define the $\bQ$-vector space $V_{D,\fm}$ by
	\[ V_{D,\fm} \coloneqq Z_0(\tilde \cD)/[\tilde \cD_{\fm}] \otimes_{\bZ} \bQ. \]
\end{defn}

We first construct $\sigma_{\fm}$ on the level of degree-0 divisors. Given $F \in \Div^0(Y)$, by \cite[Theorem III.3.6]{lang:arakelov}, there exists a vertical $\bQ$-divisor $\Phi(F)$ on~$\cX$ such that $\Psi(F) \coloneqq \cF + \Phi(F)$ has $\fm$-intersection number zero with every vertical divisor. Here, $\cF$ denotes the extension of~$F$ to a horizontal divisor on~$\cX$ obtained by taking Zariski closures of points. %
Now we define
\begin{align}\label{eq:def-sigma}
\sigma_{\fm}(F) \coloneqq \Psi(F).\tilde \cD = \sum_x i_x(\Psi(F), \tilde \cD) [x] \in V_{D,\fm},
\end{align}
where $x$ runs over all closed points of~$\tilde \cD$. 
The $\bQ$-divisor~$\Phi(F)$ is uniquely determined up to adding $\bQ$-multiples of the complete special fibre~$\cX_{\fm}$, therefore $\sigma_{\fm}(F) = \Psi(F).\tilde \cD$ is well-defined modulo $\cX_{\fm} . \tilde \cD = [\tilde \cD_{\fm}]$, i.e.\ as an element of~$V_{D,\fm}$. The bilinearity of the intersection pairing implies that $\sigma_{\fm}$ is a homomorphism. On principal divisors, $\sigma_{\fm}$ can be computed in terms of evaluation at the cusps and taking orders:

\begin{lemma}
	\label{lem:sigma-of-principal}
	Let $F = \div(f)$ be a principal divisor with support in~$Y$. Then we have
	\[ \sigma_{\fm}(F) = \sum_{Q \in |D|} \sum_{x \in |\tilde \cQ_{\fm}|} \ord_x(f(Q)) [x]. \]
	Here, $f(Q) \in k(Q)^\times$ denotes the evaluation of $f$ at the cusp~$Q$ and $x$ runs over the points of the special fibre of the normalisation~$\tilde \cQ$ of $R$ in $k(Q)$. %
\end{lemma}

\begin{proof}
	Let $\div_{\cX}(f)$ be the divisor on~$\cX$ defined by the rational function~$f$. It has horizontal part~$\cF$ and by \cite[Theorem III.3.1]{lang:arakelov} it has $\fm$-intersection number zero with all vertical divisors, so we can take $\Psi(F) \coloneqq \div_{\cX}(f)$ in the definition of $\sigma_{\fm}(F)$. By definition of the intersection numbers and \eqref{eq:intersection-Dtilde} this implies
	\[ \sigma(F) = \div_{\cX}(f).\tilde \cD = \sum_{x \in |\tilde \cD_{\fm}|} \ord_x(\nu^* f) [x], \]
	and the claim follows by splitting the sum over the components of $\tilde \cD_{\fm} = \coprod_{Q \in |D|} \tilde \cQ_{\fm}$.
\end{proof}

\begin{lemma}
	\label{thm:sigma-well-def}
	The map $\sigma_{\fm}\colon J_Y(K) \to V_{D,\fm}$ defined by \eqref{eq:def-sigma} is well-defined.
\end{lemma}

\begin{proof}
	To see that~$\sigma_{\fm}$ is well-defined on $D$-equivalence classes, assume that $F = \div(f)$ is the principal divisor on~$X$ defined by a function $f \in k(X)^\times$ with $f\vert_D = 1$. Then we have $\ord_x(f(Q)) = \ord_x(1) = 0$ for all $Q \in |D|$ and $x \in |\tilde \cQ_{\fm}|$, hence $\sigma_{\fm}(F) = 0$ by \Cref{lem:sigma-of-principal}.
\end{proof}

\begin{defn}
	\label{def:sigma}
	We call the homomorphism $\sigma_{\fm}\colon J_Y(K) \to V_{D,\fm}$, $[F] \mapsto \Psi(F).\tilde \cD$ of~\Cref{thm:sigma-well-def} the \emph{$D$-intersection map}.
\end{defn}

\begin{prop}
	\label{thm:sigma-independent-of-model}
	The $D$-intersection map $\sigma_{\fm}\colon J_Y(K) \to V_{D,\fm}$ is independent of the choice of regular model~$\cX$.
\end{prop}

\begin{proof}
	Assume that $\cX$ and $\cX'$ are two regular models of~$X$ over~$R$. We want to show that the associated homomorphisms $\sigma_{\fm}, \sigma_{\fm}'\colon J_Y(K) \to V_{D,\fm}$ agree. Since any two regular models can be dominated by a third one \cite[Proposition~4.2]{lichtenbaum}, we may assume that $\cX$ is dominated by~$\cX'$ via a morphism of models $\pi\colon \cX' \to \cX$. Let $\nu\colon \tilde \cD \to \cX$ and $\nu'\colon \tilde \cD \to \cX'$ be the morphisms obtained by extending the inclusion of $D$ into the generic fibre of~$\cX$ and $\cX'$, respectively. The uniqueness statement in the valuative criterion of properness implies $\nu = \pi \circ \nu'$. Now, let~$F$ be a degree-0 divisor on~$Y$, let $\cF$ be its extension to a horizontal divisor on~$\cX$, and let $\Phi(F)$ be a vertical $\bQ$-divisor on~$\cX$ such that $\Psi(F) \coloneqq \cF + \Phi(F)$ has $\fm$-intersection number zero with every vertical divisor. We claim that we can take $\Psi'(F) \coloneqq \pi^* \Psi(F)$ on~$\cX'$ in the definition of~$\sigma_{\fm}'(F)$. Indeed, the horizontal part of $\pi^* \Psi(F)$ is the proper transform of~$\cF$, which equals the horizontal extension~$\cF'$ of~$F$ on~$\cX'$, and for any vertical divisor~$V'$ on~$\cX'$, we have
	\[ i_{\fm}(\pi^* \Psi(F), V') = i_{\fm}(\Psi(F), \pi_* V') = 0 \]
	by the projection formula \cite[Theorem~9.2.12(b)]{liu2006algebraic} and by definition of~$\Psi(F)$. Hence, for any closed point~$x$ on~$\tilde \cD$, if $f \in k(\cX)^\times$ defines $\Psi(F)$ on~$\cX$ near~$\nu(x)$, then $\pi^*f \in k(\cX')^\times$ defines $\Psi'(F)$ near $\nu'(x)$, so we have 
	\[ i_x(\Psi'(F),\tilde \cD) = \ord_x(\nu'^* \pi^*f) = \ord_x(\nu^* f)) = i_x(\Psi(F), \tilde \cD). \]
	Consequently, $\sigma_{\fm}'(F) = \Psi'(F).\tilde \cD = \Psi(F).\tilde \cD = \sigma_{\fm}(F)$.
\end{proof}

\subsection{Computing \texorpdfstring{$\Phi(F)$}{Φ(F)}}
\label{sec:phi}

In the definition of the intersection map, given a divisor $F \in \Div^0(Y)$ with horizontal extension~$\cF$, we need to find a vertical $\bQ$-divisor $\Phi(F)$ such that $\Psi(F) \coloneqq \cF + \Phi(F)$ has $\fm$-intersection number zero with all vertical divisors. We construct $\Phi(F)$ using the intersection matrix of the special fibre.

Let $\Div_{\fm}(\cX)$ be the group of vertical divisors on~$\cX$. We have the reduction component map
\begin{align}
	\label{eq:red-cpt}
	\cpt_{\fm} \colon \Div(X) \to \Div_{\fm}(\cX), \quad F \mapsto \sum_V i_{\fm}(\cF, V) [V], 
\end{align}
where the sum is over all vertical prime divisors. Let $M = (i_{\fm}(V,W))_{V,W}$ be the intersection matrix of the (reduced) special fibre, and let $L^+$ be a generalised inverse of~$-M$, i.e.\ a matrix satisfying $(-M)L^+(-M) = -M$. For example, $L^+$ can be chosen as the Moore--Penrose pseudoinverse of~$-M$. Denote the endomorphism of $\Div_{\fm}(\cX) \otimes \bQ$ represented by~$L^+$ in the standard basis also by~$L^+$.

\begin{lemma}
	\label{lem:phi-formula}
	For $F \in \Div^0(X)$, the vertical $\bQ$-divisor $\Phi(F) \coloneqq L^+(\cpt_{\fm}(F))$ is such that $\Psi(F) \coloneqq \cF + \Phi(F)$ has $\fm$-intersection number zero with all vertical divisors.
\end{lemma}

\begin{proof}
	The intersection matrix~$M$ represents the endomorphism of $\Div_{\fm}(\cX) \otimes \bQ$ given by $E \mapsto \sum_V i_{\fm}(E,V) [V]$, with~$V$ running over all vertical prime divisors. Let $m_V$ be the multiplicity of~$V$ in the special fibre, so that $\cX_{\fm} = \sum m_V V$. Define a degree function on $\Div_{\fm}(\cX) \otimes \bQ$ by setting $\deg(V) \coloneqq m_V$ and extending linearly. Then the endomorphism~$M$ maps to the degree-zero subspace $\Div_{\fm}^0(\cX) \otimes \bQ$:
	\[ \deg(M(E)) = \sum_V m_V i_{\fm}(E,V) = i_{\fm}(E,\cX_{\fm}) = 0. \]
	Here we used that the complete special fibre~$\cX_{\fm}$ has $\fm$-intersection number zero with any vertical divisor \cite[Proposition~III.3.2]{lang:arakelov}. Since the intersection matrix~$M$ has corank~1 \cite[Lemma~III.3.4]{lang:arakelov}, the image of the endomorphism~$M$ is precisely the degree-zero subspace~$\Div_{\fm}^0(\cX) \otimes \bQ$. (In loc.\ cit., the intersection matrix is defined as $(m_V m_W i_{\fm}(V,W))_{V,W}$ but rescaling the rows and columns does not change the corank.) The reduction component map~$\cpt_{\fm}$ defined in \eqref{eq:red-cpt} preserves the degree:
	\[ \deg(\cpt_{\fm}(F)) = \sum_V m_V i_{\fm}(\cF, V) = i_{\fm}(\cF, \cX_{\fm}) = \deg(F). \]
	The last equality follows from the fact that the horizontal extension of a point~$Q$ of the generic fibre intersects the special fibre in $\deg(Q) = [k(Q) : K]$ many points when counted with multiplicities \cite[Proposition~III.2.5]{lang:arakelov}. 
	
	Now assume that $F$ is a divisor on~$X$ of degree~$0$. Then by the above, its reduction $\cpt_{\fm}(F)$ is contained in the degree-zero subspace of $\Div_{\fm}(\cX)$, hence in the image of the endomorphism~$M$. Since $L^+$ is a generalised inverse of~$-M$, the $\bQ$-divisor $\Phi(F) \coloneqq L^+(\cpt_{\fm}(F))$ satisfies $M(\Phi(F)) = -\cpt_{\fm}(F)$. Unfolding the definitions of~$M$ and $\cpt_{\fm}$, this means that $i_{\fm}(\Phi(F),V) = -i_{\fm}(\cF,V)$ for all vertical prime divisors~$V$. Hence, $\Psi(F) \coloneqq \cF + \Phi(F)$ has $\fm$-intersection number zero with all vertical divisors as claimed.
\end{proof}

\begin{lemma}
	\label{lem:phi-via-red}
	For $F \in \Div^0(X)$, the vertical $\bQ$-divisor $\Phi(F) \in \Div_{\fm}(\cX)/[\cX_{\fm}] \otimes \bQ$ for which $\Psi(F) \coloneqq \cF + \Phi(F)$ has $\fm$-intersection number zero with all vertical divisors depends only on the image of~$F$ under the reduction component map $\cpt_{\fm}$ defined in \eqref{eq:red-cpt}.
\end{lemma}

\begin{proof}
	This is an immediate consequence of \Cref{lem:phi-formula}.
\end{proof}

In light of \Cref{lem:phi-via-red}, we also write $\Phi(\cpt_{\fm}(F))$ instead of $\Phi(F)$. More generally:

\begin{defn}
	\label{def:phi-on-components}
	Given a vertical $\bQ$-divisor $E \in \Div_{\fm}^0(\cX)$ of degree zero (with respect to the degree function in the proof of \Cref{lem:phi-formula}), define $\Phi(E) \coloneqq L^+(E) \in \Div_{\fm}(\cX) \otimes \bQ$, so that $\Phi(E)$ is a vertical $\bQ$-divisor satisfying $i_{\fm}(\Phi(E), V) = -\mathrm{mult}_V(E)$ for all vertical divisors~$V$. Note that $\Phi(E)$ is well-defined up to adding $\bQ$-multiples of $\cX_{\fm}$.
\end{defn}

Another consequence of \Cref{lem:phi-formula} is that the denominators of the coefficients of the $\bQ$-divisors $\Phi(F)$ are bounded.
Namely, let $a$ be the least common denominator of all entries of $L^+$.
For an abelian group $B$ let $B_{\tors}$ be its torsion subgroup and $B_{\tf} \coloneqq B/B_{\tors}$.
With $\Phi$ defined as in \Cref{lem:phi-formula}, the $D$-intersection map $\sigma_{\fm}$ in fact takes values in the finitely generated free abelian group
\begin{equation}
\label{def:Z-structure-VD}
	A_{D,\mathfrak{m}} \coloneqq \frac{1}{a} \left( Z_0(\tilde \cD)/[\tilde \cD_{\fm}]\right)_{\tf} \subseteq V_{D,\mathfrak{m}}.
\end{equation}
The $\bZ$-structure $A_{D,\mathfrak{m}}$ of the $\bQ$-vector space $V_{D,\fm}$ will be useful in \S\ref{sec:comparison-with-kummer-map}. 

\subsection{Integral reduction types}
\label{sec:integral-reduction-types}

Let $\cX$ be a regular model of $X$ over~$R$. Let $\cD$ be the closure of~$D$ in~$\cX$ and set $\cY = \cX \smallsetminus \cD$.
Any $K$-rational point $P \in Y(K)$ extends uniquely to an $R$-valued point of~$\cX$ by properness of~$\cX$. Reducing this point modulo~$\fm$ yields a $k$-point of the special fibre. This point is contained in the smooth locus $\cX_{\fm}^{\sm}$ of~$\cX_{\fm}$, see \cite[Lemma~3.1]{liu-tong:neron-models}.

\begin{defn}
	\label{def:reduction-map}
	The composition
	\[ \red_{\fm}\colon Y(K) \hookrightarrow X(K) = \cX(R) \to \cX_{\fm}^{\sm}(k) \]
	sending a $K$-rational point of~$Y$ to its reduction in the special fibre of~$\cX$ is called the \emph{mod-$\fm$ reduction map}.
\end{defn}

A $K$-point of~$Y$ belongs to the subset of $R$-integral points $\cY(R) \subseteq Y(K)$ if and only if its mod-$\fm$ reduction does not lie on~$\cD_{\fm}$. The \emph{reduction type} of such a point remembers the component of the special fibre onto which it reduces.

\begin{defn}
	\label{def:integral-reduction-type}
	An \emph{integral mod-$\fm$ reduction type} for $(\cX,\cD)$ is an irreducible (equivalently: connected) component of $\cX_{\fm}^{\sm} \smallsetminus \cD_{\fm}$ that contains at least one $k$-point. The set of all integral mod-$\fm$ reduction types is denoted by $\cC_{\fm}^0$. We say that an $R$-integral point $P$ of~$\cY$ has reduction type~$\Sigma_{\fm}$ if its mod-$\fm$ reduction lies on~$\Sigma_{\fm}$. The set of all $R$-integral points of reduction type~$\Sigma_{\fm}$ is denoted by $\cY(R)_{\Sigma_{\fm}}$.
\end{defn}

Integral reduction types define a partition of $\cY(R)$:
\[ \cY(R) = \coprod_{\Sigma_{\fm} \in \cC^0_{\fm}} \cY(R)_{\Sigma_{\fm}}. \]

\begin{rem}
	\label{rem:integral-reduction-types-betts}
	Our notion of integral reduction type differs slightly from the reduction types considered in \cite[§6.1.1]{betts:effective}. By only admitting components of the smooth locus~$\cX_{\fm}^{\sm}$ of the special fibre we exclude all components of multiplicity larger than one. The reason for this is that only components of the smooth locus can occur as reduction types of $R$-integral points. 
\end{rem}

The reduction type of a point puts strong constraints on its image under the Abel--Jacobi map.

\begin{prop}
	\label{thm:abel-jacobi-integral-reduction-type}
	Let $P_0\in Y(K)$.
	Then the image of an $R$-integral point of~$\cY$ under the composition
	$\sigma_{\fm}\circ\AJ_{P_0}\colon \cY(R) \to J_Y(K) \to V_{D,\fm}$
	depends only on its reduction type.
More precisely, if $P \in \cY(R)$ has reduction type $\Sigma_{\fm} \in \cC_{\fm}^0$ then %
	\begin{equation}
	\label{eq:sigma-on-integral-reduction-type}
	\sigma_{\fm}(P-P_0) = -\cP_0.\tilde \cD + \Phi(\Sigma_{\fm} - \cpt_{\fm}(P_0)).\tilde\cD.
	\end{equation}
\end{prop}

\begin{proof}
	Let $P$ be an $R$-integral point of reduction type $\Sigma_{\fm} \in \cC_{\fm}^0$. Let $\cP$ and $\cP_0$ be the horizontal divisors on $\cX$ given by taking the Zariski closure of~$P$ and~$P_0$, respectively. We calculate
	\begin{align*}
	\sigma_{\fm}(P-P_0) = \Psi(P-P_0).\tilde\cD = \cP.\tilde \cD - \cP_0.\tilde \cD + \Phi(P-P_0).\tilde\cD.
	\end{align*}
	Since $P$ is an $R$-integral point, its reduction does not lie on~$\cD$, so the first summand vanishes. The second summand is independent of~$P$. The third summand, by \Cref{lem:phi-via-red}, depends on~$P$ only through its reduction component~$\cpt_{\fm}(P)$. By assumption, $P$ reduces onto the component given by the reduction type~$\Sigma_{\fm}$, and the $\fm$-intersection number of $\cP$ with this component equals~1 since $i_{\fm}(\cP,\cX_{\fm}) = [k(P) : K] = 1$, see \cite[Proposition~2.5]{lang:arakelov}. So we have $\cpt_{\fm}(P) = \Sigma_{\fm}$, viewing $\Sigma_{\fm}$ as an element of~$\Div_{\fm}(\cX)$ in the obvious way. According to \Cref{lem:phi-formula}, we can take
	\[ \Phi(P-P_0) = L^+(\cpt_{\fm}(P-P_0)) = \Phi(\Sigma_{\fm} - \cpt_{\fm}(P_0)), \]
	which gives the claimed formula~\eqref{eq:sigma-on-integral-reduction-type}.
\end{proof}

\begin{defn}
	\label{def:local-selmer-set-integral}
	Let $P_0 \in Y(K)$ be a base point and let $\Sigma_{\fm} \in \cC_{\fm}^0$ be an integral mod-$\fm$ reduction type for $(\cX,\cD)$. We define $\fS_{\fm}(P_0,\Sigma_{\fm})$ to be the one-element subset of $V_{D,\fm}$ containing the right hand side of \eqref{eq:sigma-on-integral-reduction-type}, i.e.
	\begin{equation}
	\label{eq:Sm-integral}
	\fS_{\fm}(P_0,\Sigma_{\fm}) =\bigl\{ -\cP_0.\tilde \cD + \Phi(\Sigma_{\fm} - \cpt_{\fm}(P_0)).\tilde\cD \bigr\}.
	\end{equation}
\end{defn}

According to \Cref{thm:abel-jacobi-integral-reduction-type}, we have
\[ (\sigma_{\fm} \circ \AJ_{P_0})(P) \in \fS_{\fm}(P_0,\Sigma_{\fm}) \]
for all $P \in \cY(R)$ of reduction type~$\Sigma_{\fm}$.

\begin{prop}
	\label{thm:Sm-integral-zero}
	Let $P_0 \in Y(K)$ be a base point and let $\Sigma_{\fm} \in \cC_{\fm}^0$ be an integral mod-$\fm$ reduction type. If $P_0$ is itself an $R$-integral point of reduction type~$\Sigma_{\fm}$, then $\fS_{\fm}(P_0,\Sigma_{\fm}) = \{0\}$.
\end{prop}

\begin{proof}
	If $P_0$ is $R$-integral then its closure $\cP_0$ does not meet the closure~$\cD$ of~$D$ in~$\cX$, so that we have $\cP_0.\tilde\cD = 0$. If moreover $P_0$ has reduction type~$\Sigma_{\fm}$, then $\cpt_{\fm}(P_0) = \Sigma_{\fm}$, hence we can choose $\Phi(\Sigma_{\fm} - \cpt_{\fm}(P_0)) = \Phi(0) = 0$, which implies $\fS_{\fm}(P_0,\Sigma_{\fm}) = \{0\}$ by \eqref{eq:Sm-integral}.
\end{proof}

It can happen that the composition $\sigma_{\fm}\circ\AJ_{P_0}\colon \cY(R) \to J_Y(K) \to V_{D,\fm}$
is constant on all of $\cY(R)$, not only on points of a fixed reduction type. The following proposition gives a criterion for when this occurs.

\begin{prop}
	\label{selmer-sets-constant}
	Assume that $\cD$ intersects the special fibre $\cX_{\fm}$ in only one component. Then the set $\fS_{\fm}(P_0,\Sigma_{\fm})$ is equal to $\{-\cP_0.\tilde\cD\}$, independently of the choice of integral reduction type~$\Sigma_{\fm}$.
\end{prop}

\begin{proof}
	According to \eqref{eq:Sm-integral}, $\fS_{\fm}(P_0,\Sigma_{\fm})$ is the singleton set containing an element of the form $-\cP_0.\tilde\cD + V.\tilde \cD$ where $V$ is a vertical $\bQ$-divisor depending on~$\Sigma_{\fm}$. We claim that $V.\tilde\cD = 0$ in $V_{D,\fm}$ for every vertical divisor~$V$. We may assume that $V$ is a prime divisor, i.e.\ a component of the special fibre. Let $m_V$ be the multiplicity of~$V$ in~$\cX_{\fm}$. By assumption, either $\cD$ does not intersect~$V$, or all points of the intersection $\cD \cap \cX_{\fm}$ are contained in~$V$. In the former case, we trivially have $V.\tilde \cD = 0$, while in the latter case we have
	\[ V.\tilde \cD = \tfrac1{m_V} \cX_{\fm}.\tilde\cD = \tfrac1{m_V} [\tilde\cD_{\fm}], \]
	but in $V_{D,\fm}$ this is zero as well.
\end{proof}

\subsection{$D$-transversal models}
\label{sec:normal-models}

Let $\cX$ be a regular model of~$X$ over~$R$, and let $\cD$ and $\cY$ be as before.
We also want to partition the set $Y(K)$, not just $\cY(R)$, by reduction type.
But $P\in Y(K)$ may reduce to the reduction of a cusp.
If this happens, we want to record the cusp to which $P$ reduces, so we need the model $\cX$ to be $D$\emph{-transversal}:

\begin{defn}
	\label{def:D-normal}
	\leavevmode
	\begin{enumerate}
		\item Let $x$ be a point in $(\cX_{\fm}^{\sm} \cap \cD)(k)$. We say that $\cX$ is \emph{$D$-transversal at~$x$} if $i_x(\cX_{\fm}, \cD) = 1$.
		\item We say that $\cX$ or $(\cX,\cD)$ is \emph{$D$-transversal} if it is $D$-transversal at all $x\in (\cX_{\fm}^{\sm} \cap \cD)(k)$.
	\end{enumerate}
\end{defn}

By \cite[Lemma 9.1.8(b)]{liu2006algebraic}, if the model $\cX$ is $D$-transversal at $x \in (\cX_{\fm}^{\sm} \cap \cD)(k)$ then $\cD$ is regular (equivalently, normal) at $x$. In particular, no two components of~$\cD$ intersect at~$x$, so $x$ lies in the closure of a unique point $Q \in |D|$. Note that~$\cD$ may be non-normal at points where it intersects the special fibre with higher multiplicity (e.g., singular points of~$\cX_{\fm}$) or at points which are not $k$-rational. However, $K$-points of $Y$ reduce to smooth $k$-rational points of the special fibre, so a transversality condition at these points is all we need.

\begin{rem}
	\label{rem:rnc-models}
	\cite[Appendix B]{Bet18} defines a regular normal crossings (RNC) model~$\cX$ as a regular model for which $\cD + \cX_{\fm}$ is a normal crossings divisor on~$\cX$. Every RNC model is $D$-transversal, but the converse is not true.
\end{rem}

\begin{prop}
	\label{thm:D-normal-desingularisation}
	If $R$ is excellent, then any regular model $(\cX,\cD)$ is dominated by a $D$-transversal regular model.
\end{prop}

\begin{proof}
	As $R$ is excellent, by \cite[Theorem 9.2.26]{liu2006algebraic} successive blow-ups of $\cX$ in closed points\footnote{As we only need $\cD$ to be regular at all $x\in\cX_{\fm}^{\sm}(k)$, blow-ups in those points $x \in (\cX_{\fm}^{\sm} \cap \cD)(k)$ with $i_x(\cX_{\fm}, \cD) \geq 2$ should suffice, but we did not check whether the analogue of \cite[Lemma 9.2.32]{liu2006algebraic} is true.} yield a regular model $\cX'$ of $X$ in which the closure $\cD'$ of $D$ is normal.
	Now we blow up again in every $x \in ({\cX'_{\fm}}^{\!\sm} \cap \cD')(k)$ such that $i_x(\cX'_{\fm}, \cD')\geq 2$ to obtain the model $\cX''$.
	We claim that $\cX''$ is a $D$-transversal regular model of $X$.
	Indeed, away from those $x$ the model $\cX'$ is already $D$-transversal.
	For $x \in ({\cX'_{\fm}}^{\!\sm} \cap \cD')(k)$, let $Q$ be the unique cusp such that $x$ lies in the closure $\cQ'$ of $Q$.
	If $2\leq i_x(\cX'_{\fm}, \cD') = i_x(\cX'_{\fm}, \cQ')$, let $E$ denote the exceptional divisor on $\cX''$ of the blow-up in $x$.
	Then $\cQ''$ intersects $\cX''_{\fm}$ in the intersection point of $E$ with the strict transform of ${\cX'_{\fm}}^{\!\sm}$, which is not a point in ${\cX''_{\fm}}^{\sm}$, so there is no $D$-transversality left to check.
\end{proof}

\begin{defn}
	\label{def:minimal-D-normal-model}
	A regular model $(\cX,\cD)$ of $(X,D)$ over~$R$ is the \emph{minimal regular $D$-transversal model} if it is dominated by any other such model. 
\end{defn}

It is clear that a minimal regular $D$-transversal model, if it exists, is unique up to unique isomorphism.

\begin{prop}
	\label{thm:minimal-D-normal-model}
	If $Y$ is hyperbolic and $R$ is excellent, then a minimal regular $D$-transversal model exists.%
\end{prop}

For the proof, we need a criterion characterising exceptional divisors on a $D$-transversal model whose contraction is again $D$-transversal.

\begin{defn}
	\label{def:D-exceptional}
	Let $\cX$ be a $D$-transversal model and let $E$ be an exceptional divisor on $\cX$.
	We say $E$ is \emph{$D$-exceptional} if the contraction of $E$ is again a $D$-transversal model.
\end{defn}

\begin{lemma}
	\label{lem:Dexc-div}
	Let $\cX$ be a $D$-transversal model.
	\begin{enumerate}
		\item The blow-up of $\cX$ in a closed point $x$ is again $D$-transversal.
		\item Let $E$ be an exceptional divisor on $\cX$.
		Then $E$ is $D$-exceptional if and only if (at least) one of the following conditions is satisfied:
		\begin{enumerate}[label=(\alph*),ref=(\alph*)]
			\item \label{item:E-not-defined-over-k}
			$\rH^0(E,\mathcal{O}_E)\neq k$.
			\item \label{item:E-higher-multiplicity}
			$E$ has multiplicity $\geq 2$ as a component of $\cX_{\fm}$.
			\item \label{item:E-D-dont-intersect}
			$E$ and $\cD$ do not intersect.
			\item \label{item:smooth-k-point}
			$E$ meets $\cD$ in a unique closed point $x$, which is $k$-rational and lies in $\cX_{\fm}^{\sm}$.
		\end{enumerate}
	\end{enumerate}
\end{lemma}

\begin{proof}
	(1) Let $\pi\colon\cX'\to\cX$ be the blow-up of $\cX$ in $x$, let $E= \pi^{-1}(x)$ denote the exceptional locus of $\pi$ and let $\cD'$ be the closure of $D$ in $\cX'$.
	As $\pi$ is an isomorphism above~$\cX \smallsetminus \{x\}$, we only need to check that $\cX'$ is $D$-transversal at points in~$E$. 
	Now if $x \not\in (\cX_{\fm}^{\sm} \cap \cD)(k)$ then $E$ is not defined over $k$ or does not have multiplicity $1$ as a component of $\cX'_{\fm}$ or does not intersect $\cD'$, and in all cases $\cX'$ is automatically $D$-transversal.
	So assume $x \in (\cX_{\fm}^{\sm} \cap \cD)(k)$.
	The projection formula \cite[Theorem~III.4.1]{lang:arakelov} applied to the divisors $\cX_{\fm}$ on~$\cX$ and $\cD'$ on~$\cX'$ yields 
	\[ \pi_*(\pi^* \cX_{\fm}.\cD') = \cX_{\fm}.\pi_*(\cD'). \]
	Using that $\pi^*(\cX_{\fm}) = \cX_{\fm}'$ and $\pi_*(\cD') = \cD$, comparing coefficients of~$[x]$ yields
	\begin{equation}
		\label{eq:proj-formula-refined}
		\sum_{x' \in E} i_{x'}(\cX_{\fm}', \cD') [k(x') : k(x)] = i_x(\cX_{\fm}, \cD).
	\end{equation}
	As $\cX$ is $D$-transversal at~$x$, the right hand side is equal to~$1$, which implies that there is exactly one closed point~$x'$ where $E$ and $\cD'$ intersect, and this point is $k$-rational and satisfies $i_{x'}(\cX'_{\fm},\cD')=1$, i.e.\ $\cX'$ is $D$-transversal at $x'$.

	(2) Let $\pi\colon\cX\to\cX'$ be the contraction of $E$, let $x' \coloneqq \pi(E)$, and let $\cD'$ be the closure of $D$ in $\cX'$.
	As $\pi$ is an isomorphism above $\cX' \smallsetminus \{x'\}$, we see that $\cX'$ is automatically $D$-transversal except possibly at $x'$.
	Equation \eqref{eq:proj-formula-refined} applied to $\pi\colon \cX\to \cX'$ yields
	\begin{equation}\label{eq:proj-formula-refined2}
		 \sum_{x \in E} i_x(\cX_{\fm},\cD) [k(x):k(x')] = i_{x'}(\cX'_{\fm},\cD').
	\end{equation}
	
	Now in cases \ref{item:E-not-defined-over-k}, \ref{item:E-higher-multiplicity}, or \ref{item:E-D-dont-intersect}, we have $x'\not\in ({\cX'_{\fm}}^{\!\sm} \cap \cD')(k)$ and hence $\cX'$ is automatically $D$-transversal.
	In case \ref{item:smooth-k-point}, we have $x'\in ({\cX'_{\fm}}^{\!\sm} \cap \cD')(k)$, so $\cX$ is $D$-transversal at $x$ by assumption and $i_{x'}(\cX'_{\fm},\cD') = 1$ by \eqref{eq:proj-formula-refined2}, as desired.
	
	Conversely, suppose that $\cX'$ is $D$-transversal and assume we are not in cases \ref{item:E-not-defined-over-k}, \ref{item:E-higher-multiplicity}, or \ref{item:E-D-dont-intersect}.
	That is, we assume that $\rH^0(E,\mathcal{O}_E)=k(x')=k$, that $E$ has multiplicity~$1$ in $\cX_{\fm}$, which is equivalent by \cite[Proposition 9.2.23]{liu2006algebraic} to $\cX'_{\fm}$ being regular at $x'$, and that $x'$ lies on $\cD'$.
	Then $x'\in ({\cX'_{\fm}}^{\!\sm} \cap \cD')(k)$, so $i_{x'}(\cX'_{\fm},\cD') = 1$ by $D$-transversality, and \eqref{eq:proj-formula-refined2} implies condition \ref{item:smooth-k-point}.
\end{proof}

\begin{proof}[Proof of \Cref{thm:minimal-D-normal-model}]
	\Cref{thm:D-normal-desingularisation} ensures the existence of a $D$-transversal regular model.
	By repeatedly contracting all $D$-exceptional divisors, we obtain a relatively minimal $D$-transversal model, i.e.\ one that does not dominate any other such model.
	We show that any two relatively minimal $D$-transversal models $\cX_1$ and $\cX_2$ are isomorphic, and therefore minimal such models.
	Assume the contrary.
	As $\cX_1$ is relatively minimal, the birational map $f\colon \cX_1 \dashrightarrow \cX_2$ is not defined everywhere, say not at the closed point $x_0$.
	We now refine the argument of the corrected version of \cite[Lemma 9.3.20]{liu2006algebraic}, published as an erratum \cite{liu-erratum} on the author's webpage.
	Namely, there exists a morphism $g\colon\widetilde{\cZ}\to \cX_2$ that is a finite sequence of blow-ups in closed points
	\[
		g\colon \widetilde{\cZ} = \cZ_n \to \ldots \to \cZ_1 \to \cZ_0=\cX_2
	\]
	and a morphism $h\colon \widetilde{\cZ} \to \cX_1$ such that the diagram
	\begin{equation*}
		\begin{tikzcd}
		\widetilde{\cZ} \arrow[d, "g"'] \arrow[rd, "h"] &       \\
		\cX_2 \arrow[r, "f^{-1}"', dashed]                    & \cX_1
		\end{tikzcd}
	\end{equation*}
	commutes.
	Let $E$ be an exceptional divisor on $\widetilde{\cZ}$ contained in $h^{-1}(x_0)$ and let $\Gamma_i\subseteq\cZ_i$ denote the exceptional divisor of $\cZ_i\to\cZ_{i-1}$.
	By \Cref{lem:Dexc-div}, the model $\widetilde{\cZ}$ is $D$-transversal and the divisors $E$ and all $\Gamma_i$ are $D$-exceptional.
	If $E$ and $\Gamma_n$ do not meet, the image of $E$ in $\cZ_{n-1}$ is again $D$-exceptional and will also be denoted by $E$.
	Let $m\leq n$ be the smallest integer such that $\Gamma_m$ and $E$ are disjoint.
	If $m=1$, then $E$ is a $D$-exceptional divisor on $\cZ_0=\cX_2$, contradicting the relative minimality of $\cX_2$.
	If $m\geq 2$, then $E$ and $\Gamma_{m-1}$ are exceptional divisors on $\cZ_{m-1}$ such that $E\cap\Gamma_{m-1}\neq\emptyset$.
	Pick a point $x\in E\cap\Gamma_{m-1}$ and let $k_1=\rH^0(\Gamma_{m-1},\cO_{\Gamma_{m-1}})$ and $k_2=\rH^0(E,\cO_E)$.
	Then $k_i\subseteq k(x)$ and
	\[
		(\Gamma_{m-1} + E)^2 = \Gamma_{m-1}^2 + E^2 + 2\Gamma_{m-1}\cdot E \geq  -[k_1:k]-[k_2:k]+2[k(x):k] \geq 0.
	\]
	By \cite[Theorem 9.1.23]{liu2006algebraic}, we have equality above, thus $k_1=k(x)=k_2$, and we have $\cZ_{m-1,\fm}=i(\Gamma_{m-1}+E)$ for some integer $i\geq 1$. %
	The adjunction formula \cite[Propositions 9.1.35 and 9.3.10(a)]{liu2006algebraic} shows $g=0$; in particular, if $X$ has genus $g \geq 1$ we have arrived at a contradiction.
	Assume $g=0$ from now on.
	As the smooth locus of $\cZ_{m-1,\fm}$ cannot be empty, we have $i=1$.
	As $X$ is geometrically connected, so is the special fibre $\cZ_{m-1,\fm}$ by \cite[Lemma 8.3.6(b)]{liu2006algebraic} and hence $k_1=k_2=k$.
	As $\Gamma\in\{\Gamma_{m-1},E\}$ is $D$-exceptional, \Cref{lem:Dexc-div}(2) shows that $\Gamma\cdot\cD\leq 1$, where $\cD$ is the closure of $D$ inside the model $\cZ_{m-1}$.
	Now $Y$ is hyperbolic, hence $\# D(\overline{K}) \geq 3$, and using \cite[Proposition 9.1.30]{liu2006algebraic} we arrive at the contradiction
	\[
	3 \leq \# D(\overline{K}) = \sum_{Q\in|D|} [k(Q):K] = \cZ_{m-1,\fm}\cdot \cD = (\Gamma_{m-1}+E)\cdot \cD \leq 2. \qedhere
	\]
\end{proof}

\begin{rem}
	In all applications, $R$ will be the localisation of the ring of integers in a number field at a prime ideal or a completion thereof and hence excellent, so the existence of a $D$-transversal regular model is guaranteed by \Cref{thm:minimal-D-normal-model}.
	Working with one such model is usually enough, but the minimal $D$-transversal regular model minimizes the number of possible rational reduction types, compare \S\ref{sec:rational-reduction-types}.
\end{rem}

\subsection{Rational reduction types}
\label{sec:rational-reduction-types}

The $K$-rational points $Y(K)$ can also be partitioned by reduction type, according to their mod-$\fm$ reduction on a suitable regular model. Here we want to work with a $D$-transversal regular model $(\cX,\cD)$, see \Cref{def:D-normal}. 
A $K$-point of~$Y$, unlike an $R$-integral point, may reduce onto the special fibre of~$\cD$. In this case, the reduction type of the point records to which point on $\cD_{\fm}$ it reduces; otherwise, it records only the component. 

\begin{defn}
	\label{def:rational-reduction-type}
	\begin{enumerate}[label=(\alph*)]
		\item Let $(\cX,\cD)$ be a regular $D$-transversal model of $(X,D)$. A \emph{rational mod-$\fm$ reduction type} for $(\cX,\cD)$ is either an irreducible/connected component of $\cX_{\fm}^{\sm} \smallsetminus \cD_{\fm}$ containing at least one $k$-point, or a $k$-point on the special fibre of~$\cD$ at which~$\cX_{\fm}$ is smooth. 
		In the latter case, we call the reduction type \emph{cuspidal}; otherwise we call it a \emph{component reduction type}.
		The set of rational mod-$\fm$ reduction types is denoted by $\cC_{\fm} \coloneqq \cC_{\fm}^0 \sqcup \cC_{\fm}^1$ where
		\begin{align*}
			\cC_{\fm}^0 &\coloneqq \{\text{components of $\cX_{\fm}^{\sm} \smallsetminus \cD$ containing a $k$-point}\},\\
			\cC_{\fm}^1 &\coloneqq (\cX_{\fm}^{\sm} \cap \cD)(k).
		\end{align*}
		For a reduction type~$\Sigma_{\fm} \in \cC_{\fm}$, denote by $Y(K)_{\Sigma_{\fm}}$ the set of all $K$-rational points of~$Y$ whose mod-$\fm$ reduction lies on~$\Sigma_{\fm} \in \cC_{\fm}^0$, respectively, is equal to $\Sigma_{\fm} \in \cC_{\fm}^1$. 
		\item If no model of $X$ is specified and $Y$ is hyperbolic, a \emph{rational mod-$\fm$ reduction type} for $(X,D)$ is a rational mod-$\fm$ reduction type for the minimal regular $D$-transversal model $(\cX_{\min},\cD_{\min})$, which exists by \Cref{thm:minimal-D-normal-model}.
	\end{enumerate}
\end{defn}

For any regular $D$-transversal model $(\cX,\cD)$ of $(X,D)$, the rational mod-$\fm$ reduction types for $(\cX,\cD)$ define a partition of $Y(K)$:
\[ Y(K) = \coprod_{\Sigma_{\fm} \in \cC_{\fm}} Y(K)_{\Sigma_{\fm}}. \]

\begin{rem}
	\label{rem:rational-reduction-types-betts}
	Our rational mod-$\fm$ reduction types are a slight modification of the reduction types considered in \cite[§6.1.2]{betts:effective}. One difference is that we only require the model to be $D$-transversal, which is weaker than requiring the divisor $\cX_{\fm} + \cD$ to have normal crossings everywhere. Moreover, our $\cC_{\fm}^1$ contains only $k$-points rather than all closed points of $\cD_{\fm}$, and by requiring $\cX_{\fm}$ to be smooth at those points we exclude points where $\cD$ intersects the special fibre $\cX_{\fm}$ in a component of higher multiplicity. Since every $K$-point of~$Y$ reduces onto a $k$-point of the smooth locus of~$\cX_{\fm}$, our reduction types are the only ones that can occur. 
\end{rem}

For $K$-rational points, the reduction type constrains the image in $V_{D,\fm}$ to a translate of a subgroup of rank~$\leq 1$. 

\begin{prop}
	\label{thm:abel-jacobi-rational-reduction-type}
	Let $P_0 \in Y(K)$, let $(\cX,\cD)$ be a regular $D$-transversal model of $(X,D)$ and let $\Sigma_{\fm} \in \cC_{\fm}$ be a rational mod-$\fm$ reduction type for $(\cX,\cD)$. Consider the composition
	\begin{equation}
	\label{eq:rational-points-to-Vm}
	Y(K) \overset{\AJ_{P_0}}{\lto} J_Y(K) \overset{\sigma_{\fm}}{\lto} V_{D,\fm}.
	\end{equation}
	\begin{enumerate}[label=(\alph*)]
		\item \label{item:integral-red-type}
		If $\Sigma_{\fm} \in \cC_{\fm}^0$ is a component then the map~\eqref{eq:rational-points-to-Vm} is constant on~$Y(K)_{\Sigma_{\fm}}$. More precisely, if $P \in Y(K)$ has reduction type~$\Sigma_{\fm}$ then
		\begin{equation}
		\label{eq:sigma-on-C0}
		\sigma_{\fm}(P-P_0) = -\cP_0.\tilde{\cD}  + \Phi(\Sigma_{\fm} - \cpt_{\fm}(P_0)).\tilde{\cD}.
		\end{equation}
		\item \label{item:cuspidal-red-type}
		If $\Sigma_{\fm} \in \cC_{\fm}^1$ is a point in $(\cX_{\fm}^{\sm} \cap \cD)(k)$, then the image of $Y(K)_{\Sigma_{\fm}}$ under~\eqref{eq:rational-points-to-Vm} is contained in a translate of a subgroup of~$V_{D,\fm}$ of rank~$\leq 1$. More precisely, if $P \in Y(K)$ has reduction type~$\Sigma_{\fm}$ then
		\begin{equation}
		\label{eq:sigma-on-C1}
		\sigma_{\fm}(P-P_0) \in \bZ \cdot [\Sigma_{\fm}] - \cP_0.\tilde{\cD} + \Phi(\cpt(\Sigma_{\fm}) - \cpt_{\fm}(P_0)).\tilde{\cD}.
		\end{equation}
	\end{enumerate}
\end{prop}

\begin{proof}
	In case~\ref{item:integral-red-type}, $\Sigma_{\fm}$ is an integral reduction type, so $Y(K)_{\Sigma_{\fm}} = \cY(R)_{\Sigma_{\fm}}$ and the claim follows from \Cref{thm:abel-jacobi-integral-reduction-type}. Suppose we are in case~\ref{item:cuspidal-red-type}. Let $P$ be a $K$-rational point of reduction type~$\Sigma_{\fm}$, let $\cP$ and $\cP_0$ be the horizontal extension of~$P$ and $P_0$ in~$\cX$, respectively. %
	Then
	\begin{align*}
		\sigma_{\fm}(P-P_0) = \Psi(P-P_0).\tilde{\cD} =\cP.\cD - \cP_0.\cD + \Phi(P-P_0).\tilde{\cD}.
	\end{align*}
	Since $P$ reduces to $\Sigma_{\fm} \in \cD_{\fm}$ and $\cX$ is $D$-transversal, the divisor~$\cP$ intersects~$\tilde{\cD}$ precisely in the point $\Sigma_{\fm}$, so we have $\cP.\tilde{\cD} \in \bZ \cdot [\Sigma_{\fm}]$. For the third summand, by \Cref{lem:phi-via-red} we can take
	\[ \Phi(P-P_0) = L^+(\cpt_{\fm}(P-P_0)) = \Phi(\cpt(\Sigma_{\fm}) - \cpt_{\fm}(P_0)), \]
	which shows the formula~\eqref{eq:sigma-on-C1}.
\end{proof}

\begin{rem}
	\label{rem:rank-leq1-subgroup}
	In case~\ref{item:cuspidal-red-type} of \Cref{thm:abel-jacobi-rational-reduction-type}, the image of $Y(K)_{\Sigma_{\fm}}$ under the map~\eqref{eq:rational-points-to-Vm} is contained in a translate of the group $\bZ \cdot [\Sigma_{\fm}] \subseteq V_{D,\fm}$. This is a free subgroup of rank~1 unless $D$ consists of a single $K$-rational point, in which case $V_{D,\fm}$ itself is trivial. Indeed, if $[\Sigma_{\fm}] = 0$ in~$V_{D,\fm}$, then $V_{D,\fm} = 0$, so $\Sigma_{\fm}$ is the only closed point on $\tilde{\cD}$.
	As $\cX$ is $D$-transversal at $\Sigma_{\fm}$, this shows that $D$ contains only a single point~$Q$ and that $\tilde{\cD}=\cD=\cQ$.
	The degree $[k(Q) : K]$ equals the $\fm$-intersection number $i_{\fm}(\cX_{\fm}, \cD)$, which must be equal to one as a consequence of the $D$-transversality condition and $\Sigma_{\fm} \in (\cX_{\fm}^{\sm} \cap \cD)(k)$.
\end{rem}

\begin{rem}
	\label{rem:nc-condition-relevance}
	The $D$-transversality of $\cX$ is important in the proof of \Cref{thm:abel-jacobi-rational-reduction-type}. Otherwise, knowing that $P \in Y(K)$ reduces to $\Sigma_{\fm}$ on~$\cD$ only constrains $\cP.\tilde\cD$ to be an element of the subgroup generated by all points of~$\tilde \cD$ lying over~$\Sigma_{\fm}$, which could have rank larger than one.
\end{rem}

\begin{defn}
	\label{def:local-selmer-set-rational}
	Let $P_0 \in Y(K)$ be a base point. Let $(\cX,\cD)$ be a regular $D$-transversal model of $(X,D)$ over~$R$ and let $\Sigma_{\fm} \in \cC_{\fm}$ be a rational mod-$\fm$ reduction type for $(\cX,\cD)$. We define the subset $\fS_{\fm}(P_0,\Sigma_{\fm}) \subseteq V_{D,\fm}$ as follows:
	\begin{enumerate}[label=(\alph*)]
		\item \label{item:rational-Sigma-component}
		If $\Sigma_{\fm} \in \cC_{\fm}^0$ is a component, then $\fS_{\fm}(P_0,\Sigma_{\fm})$ is defined as the one-element subset of~$V_{D,\fm}$ containing the right hand side of~\eqref{eq:sigma-on-C0}, i.e.
		\begin{equation}
			\label{eq:Sm-rational-C0}
			\fS_{\fm}(P_0,\Sigma_{\fm}) = \bigl\{-\cP_0.\tilde{\cD}  + \Phi(\Sigma_{\fm} - \cpt_{\fm}(P_0)).\tilde{\cD} \bigr\}.
		\end{equation}
		\item \label{item:rational-Sigma-cusp}
		If $\Sigma_{\fm} \in \cC_{\fm}^1$ is a point in $(\cX_{\fm}^{\sm} \cap \cD)(k)$, then $\fS_{\fm}(P_0,\Sigma_{\fm})$ is defined as the set on the right hand side of~\eqref{eq:sigma-on-C1}, i.e.
		\begin{equation}
			\label{eq:Sm-rational-C1}
			\fS_{\fm}(P_0,\Sigma_{\fm}) = \bZ \cdot [\Sigma_{\fm}] - \cP_0.\tilde{\cD} + \Phi(\cpt(\Sigma_{\fm}) - \cpt_{\fm}(P_0)).\tilde{\cD}.
		\end{equation}
	\end{enumerate}
\end{defn}

According to \Cref{thm:abel-jacobi-rational-reduction-type}, we have
\[ (\sigma_{\fm} \circ \AJ_{P_0})(P) \in \fS_{\fm}(P_0,\Sigma_{\fm}) \]
for all $P \in Y(K)$ of reduction type~$\Sigma_{\fm}$.

\begin{rem}
	\label{rem:redundant-reduction-types}
	When $\Sigma_{\fm}$ is a component reduction type, it often happens that the singleton set $\fS(P_0,\Sigma_{\fm})$ of \eqref{eq:Sm-rational-C0} is contained in the $\fS(P_0,\Sigma_{\fm}')$ of \eqref{eq:Sm-rational-C1} for a cuspidal reduction type~$\Sigma_{\fm}'$. 
	Specifically, when $x \in \cC_{\fm}^1$ is a point in $(\cX_{\fm}^{\sm} \cap \cD)(k)$ and $\cpt(x) \in \cC_{\fm}^0$ is the component of the special fibre on which it lies, then $\fS_{\fm}(P_0, \cpt(x)) \subseteq \fS_{\fm}(P_0, x)$.	Moreover, if $\cD$ intersects the special fibre $\cX_{\fm}$ in only one component, the set $\fS_{\fm}(P_0,C)$ is in fact the same for each component $C \in \cC_{\fm}^0$ by \Cref{selmer-sets-constant}, so we have $\fS_{\fm}(P_0,C) \subseteq \fS_{\fm}(P_0,x)$ for \emph{all} component reduction types $C \in \cC_{\fm}^0$. These observations can be useful in applications of the Affine Chabauty method. They imply that components containing a point in $(\cX_{\fm}^{\sm} \cap \cD)(k)$ need not be considered separately as rational reduction types; and if $\cD$ intersects $\cX_{\fm}$ in only one component and $(\cX_{\fm} \cap \cD)(k) \neq \emptyset$ (so that $\cC_{\fm}^1 \neq \emptyset$), no component reduction types need to be considered at all, as they are already covered by any cuspidal reduction type.
\end{rem}

\subsection{Comparison with the Kummer map}
\label{sec:comparison-with-kummer-map}

Assume that $K$ is a non-archimedean local field and $R$ is its valuation ring of residue characteristic $\ell$.
Let $K^s/K$ be a separable closure and let $G_K = \Gal(K^s/K)$ be the absolute Galois group of~$K$.
Let $p$ be a prime.
The short exact sequences of commutative group schemes
\[ 0 \lto J_Y[p^n] \lto J_Y \overset{\cdot p^n}{\lto} J_Y \lto 0 \]
induce the \emph{Kummer map}
\begin{equation}
	\label{eq:kummer-map}
	\kappa\colon J_Y(K) \lto \rH^1(G_K, V_p J_Y),
\end{equation}
where the cohomology on the right hand side is continuous Galois cohomology of the $\bQ_p$-linear Tate module $V_p J_Y = \bigl( \varprojlim_n J_Y(K^s)[p^n] \bigr) \otimes_{\bZ_p} \bQ_p$ of~$J_Y$.
In this section, we compare the $D$-intersection map $\sigma_{\fm}$ with the Kummer map $\kappa$.

\subsubsection{The case $\ell\neq p$}

Here we show an ``independence of $p$'' result for the Kummer map, namely that $V_{D,\fm}$ provides a $\bQ$-form of the $\bQ_p$-vector space $\rH^1(G_K, V_p J_Y)$, and that the Kummer map~\eqref{eq:kummer-map} factors through $V_{D,\fm}$ via the $D$-intersection map $\sigma_{\fm}$.
Hence one may regard $\sigma_{\fm}$ as a ``motivic'' Kummer map of $J_Y$.

\begin{thm}
	\label{kummer-comparison}
	Assume $\ell\neq p$.
	Then there is a $\bQ$-linear map $\varphi_{\bQ}\colon V_{D,\fm} \to \rH^1(G_K, V_p J_Y)$ inducing an isomorphism
	\[ V_{D,\fm} \otimes_{\bQ} \bQ_p \cong \rH^1(G_K, V_p J_Y) \]
	and fitting into the commutative diagram
	\begin{equation} 
	\label{diag:Kummer-compare}
	\begin{tikzcd}[column sep=small]
		& J_Y(K) \arrow[dl,"\sigma_{\fm}"'] \drar{\kappa} & \\
		V_{D,\fm} \arrow[rr,"\varphi_{\bQ}"] && \rH^1(G_K, V_p J_Y).
	\end{tikzcd}
	\end{equation}
\end{thm}

\begin{proof}
We work with the integral structure $A_{D,\fm}$ defined in \eqref{def:Z-structure-VD} and construct a $\bZ$-linear map $\varphi\colon A_{D,\fm} \to \rH^1(G_K, V_p J_Y)$.
The theorem then follows by setting $\varphi_{\bQ} \coloneqq \varphi\otimes\bQ$.
We construct $\varphi \coloneqq \iota\circ\kappa\circ\beta$ as follows: %
first we use that $\rH^i(G_K, V_p J) = 0$ for $i = 0,1$ so that the map induced by $T_D \hookrightarrow J_Y$ is an isomorphism $\iota \colon \rH^1(G_K, V_p T_D) \xrightarrow{~\sim~} \rH^1(G_K, V_p J_Y)$.
Secondly, Kummer theory for the extensions $k(Q)/K$ gives the Kummer isomorphism $\kappa \colon T_D(K)^{\wedge p}_{\bQ_p} \xrightarrow{~\sim~} \rH^1(G_K, V_p T_D)$.
Here, $T_D(K)^{\wedge p}_{\bQ_p} = T_D(K)^{\wedge p} \otimes_{\bZ_p} \bQ_p $ and $T_D(K)^{\wedge p} = (\varprojlim_n T_D(K)/p^n J_Y(K))$ is the pro-$p$ completion of $T_D(K)$.
Thirdly, picking uniformisers for the fields $k(Q)$ gives a map $\beta\colon A_{D,\fm} \to T_D(K)^{\wedge p}_{\bQ_p}$, which is independent of the choice of uniformisers because the characteristic of the residue field of $K$ is different from $p$, thus Kummer theory for $k(Q)$ is given by its valuation.
Note that $(A_{D,\fm})^{\wedge p}_{\bQ_p} = A_{D,\fm} \otimes_{\bZ} \bQ_p$ and that $\beta$ induces an isomorphism $\beta \colon (A_{D,\fm})^{\wedge p}_{\bQ_p} \xrightarrow{~\sim~} T_D(K)^{\wedge p}_{\bQ_p}$.

The commutativity of \eqref{diag:Kummer-compare} now follows from the functoriality of the Kummer maps and of $p$-adic completion together with the diagram
\begin{equation*}
	\begin{tikzcd}
	T_D(K) \arrow[r, hook] \arrow[d] & J_Y(K) \arrow[d, "\sigma_{\fm}"] \\
	T_D(K)^{\wedge p}_{\bQ_p}        & {A_{D,\fm}}, \arrow[l, "\beta"]  
	\end{tikzcd}
\end{equation*}
which is commutative by \Cref{lem:sigma-of-principal}.
\end{proof}

Now let $\cX$ be a $D$-transversal model of $X$ over $R$.
Let us compare the image of $\fS_{\fm}(P_0,\Sigma_{\fm})$ under $\varphi$ with the \emph{cohomological} local Selmer set $\fS_{\fm}^{\coh}(P_0,\Sigma_{\fm})$ defined in \cite[\S 6.1]{betts:effective}.
If $\Sigma_{\fm}$ is an integral (resp.\ rational) reduction type, the latter is defined as the Zariski closure of the image of $\cY(R)_{\Sigma_{\fm}}$ (resp.\ $Y(K)_{\Sigma_{\fm}}$) under the composition 
$\kappa \circ \AJ_{P_0} \colon Y(K) \to J_Y(K) \to \rH^1(G_K, V_p J_Y)$, where the vector space $\rH^1(G_K, V_p J_Y)$ is viewed as an affine $\bQ_p$-scheme in the obvious way.

\begin{cor}
	\label{lem:compare-local-Selmer}
	The set $\fS_{\fm}^{\coh}(P_0,\Sigma_{\fm})$ is equal to the Zariski closure of $\varphi_{\bQ}(\fS_{\fm}(P_0,\Sigma_{\fm}))$, so is either a singleton or an affine line.
	Moreover, inside $J_Y(K)$ we have
	\begin{align}
	\label{eq:Selmer-preimages-compare}
		\kappa^{-1}(\fS_{\fm}^{\coh}(P_0,\Sigma_{\fm})) = \sigma_{\fm}^{-1}(\fS_{\fm}(P_0,\Sigma_{\fm})).
	\end{align}
\end{cor}

\begin{proof}
The first and last assertion follow from the commutativity of \eqref{diag:Kummer-compare} together with \Cref{thm:abel-jacobi-integral-reduction-type} and \Cref{thm:abel-jacobi-rational-reduction-type} once we know that $\sigma_{\fm}^{-1}(\fS_{\fm}(P_0,\Sigma_{\fm}))$ is sufficiently large.	
If $\Sigma_{\fm}$ is non-cuspidal, it admits a smooth $k$-point, so by Hensel's Lemma there exists an $R$-point $P \in \cY(R)$ reducing onto $\Sigma_{\fm}$, hence the singleton set $\fS_{\fm}(P_0,\Sigma_{\fm})$ (resp.\ $\fS_{\fm}^{\coh}(P_0,\Sigma_{\fm})$) contains $\sigma_{\fm}(\AJ_{P_0}(P))$ (resp.\ $\kappa(\AJ_{P_0}(P))$) as desired.
Now assume that $\Sigma_{\fm} \in (\cX_{\fm}^{\sm} \cap \cQ)(k)$ is a $k$-point on the horizontal extension of a cusp $Q$ defining a cuspidal reduction type. 
Let $t_Q$ be a local equation for~$\cQ$ near~$\Sigma_{\fm}$, then $t_Q$ reduces to a uniformiser at~$\Sigma_{\fm}$ by $D$-transversality, hence $t_Q$ defines a bijection between the residue disc of $\Sigma_{\fm}$ in $\cX$ and $\fm$.
Let $P$ be a $K$-rational point in this residue disc.
Then
\[
i_{\Sigma_{\fm}}(\cP,\cQ)  = \ord_{\Sigma_{\fm}}(t_Q\vert_{\cP}) = v_K(t_Q(P)), %
\]
by \cite[Lemma~III.2.2]{lang:arakelov}.
As $t_Q$ parametrises the residue disc, \eqref{eq:sigma-on-C1} shows that
\begin{align*}
\sigma_{\fm}(\AJ_{P_0}(Y(K)_{\Sigma_{\fm}})) &= 
\bZ_{>0} \cdot [\Sigma_{\fm}] - \cP_0.\tilde{\cD} + \Phi(\cpt(\Sigma_{\fm}) - \cpt_{\fm}(P_0)).\tilde{\cD} \subseteq \fS_{\fm}(P_0,\Sigma_{\fm}).%
\end{align*}
We conclude using \eqref{eq:Sm-rational-C1} and \Cref{kummer-comparison}, which also shows that $\fS_{\fm}^{\coh}(P_0,\Sigma_{\fm})$ is an affine line.
\end{proof}

\begin{rem}
	\Cref{lem:compare-local-Selmer} reproves and sharpens \cite[Lemma 6.1.4]{betts:effective} in the case $U = V_p J_Y$.
	Moreover, \cite[Remark 1.1.3]{BD:refined} prove an ``independence of $p$'' result similar to \Cref{kummer-comparison} for the nonabelian Kummer map.
\end{rem}

\subsubsection{The case $\ell=p$}

Now let $K/\bQ_p$ be finite. %
For a $\bQ_p$-linear representation $V$ of $G_K$, we write $\rH^1_f(G_K,V)\subseteq \rH^1(G_K,V)$ for the subspace of crystalline classes.
We abbreviate $\rH^1/\rH^1_f(G_K,V)$ for $\rH^1(G_K,V)/\rH^1_f(G_K,V)$ and $\bar{\kappa}$ for the composition of $\kappa$ with $\rH^1(G_K,V) \to \rH^1/\rH^1_f(G_K,V)$.

\begin{thm}
	\label{kummer-comparison-p}
	There is a $\bQ$-linear map $\varphi_{\bQ}\colon V_{D,\fm} \to \rH^1/\rH^1_f(G_K, V_p J_Y)$ making the diagram
	\begin{equation} 
	\label{diag:Kummer-compare-p}
	\begin{tikzcd}[column sep=small]
	& J_Y(K) \arrow[dl,"\sigma_{\fm}"'] \drar{\bar{\kappa}} & \\
	V_{D,\fm} \arrow[rr,"\varphi_{\bQ}"] && \rH^1/\rH^1_f(G_K, V_p J_Y)
	\end{tikzcd}
	\end{equation}
	commute.
\end{thm}

\begin{proof}
	We again work with the integral structure $A_{D,\fm}$, construct a $\bZ$-linear map $\varphi\colon A_{D,\fm} \to \rH^1/\rH^1_f(G_K, V_p J_Y)$, and set $\varphi_{\bQ} \coloneqq \varphi\otimes\bQ$.
	We construct $\varphi$ as follows:
	Kummer theory for the extensions $k(Q)/K$ gives the Kummer isomorphism $\kappa \colon T_D(K)^{\wedge p}_{\bQ_p} \xrightarrow{~\sim~} \rH^1(G_K, V_p T_D)$ identifying crystalline classes with the $\kappa$-images of integral points.
	Thus $\bar{\kappa}\colon T_D(K)^{\wedge p}_{\bQ_p} \to \rH^1/\rH^1_f(G_K, V_p T_D)$ is given by taking valuations.
	Picking uniformisers for the fields $k(Q)$ gives a map $\beta\colon A_{D,\fm} \to T_D(K)^{\wedge p}_{\bQ_p}$ such that $\bar{\kappa}\circ\beta \colon A_{D,\fm} \to T_D(K)^{\wedge p}_{\bQ_p}$ is independent of the choice of uniformisers and an isomorphism after tensoring with $\bQ_p$.
	Then define $\varphi\coloneqq \iota\circ\bar{\kappa}\circ\beta$, where $\iota$ is induced from $V_p T_D \to V_p J_Y$.
	
	To show the commutativity of \eqref{diag:Kummer-compare-p}, we use that \eqref{eq:generalised-jacobian-exact-sequence} gives the short exact sequence
	\begin{equation}
	\label{eq:generalised-jacobian-exact-sequence-Qp}
	0 \to T_D(K)^{\wedge p}_{\bQ_p} \to J_Y(K)^{\wedge p}_{\bQ_p} \to J(K)^{\wedge p}_{\bQ_p} \to 0
	\end{equation}
	of finite dimensional $\bQ_p$-vector spaces.
	For the Jacobian $J$, the Kummer map $\kappa\colon J(K)^{\wedge p}_{\bQ_p} \to \rH^1_f(G_K,V_p J)$ is an isomorphism \cite[p.\ 352]{BK90} and hence $\bar{\kappa}\colon J(K)^{\wedge p}_{\bQ_p} \to \rH^1/\rH^1_f(G_K,V_p J)$ is the zero map.
	Thus the image of any splitting $\chi\colon J(K)^{\wedge p}_{\bQ_p} \to J_Y(K)^{\wedge p}_{\bQ_p}$ of \eqref{eq:generalised-jacobian-exact-sequence-Qp} is contained in the kernel of $\bar{\kappa}\colon J_Y(K)^{\wedge p}_{\bQ_p} \to \rH^1/\rH^1_f(G_K,V_p J_Y)$.
	As $\sigma_{\fm}\colon T_D(K)^{\wedge p}_{\bQ_p} \to (A_{D,\fm})^{\wedge p}_{\bQ_p}$ is surjective, we can choose a splitting $\chi$ of \eqref{eq:generalised-jacobian-exact-sequence-Qp} in such a way that its image is contained in the kernel of $\sigma_{\fm}$.
	The commutativity of \eqref{diag:Kummer-compare-p} now follows from the commutativity of 
	\begin{equation} 
	\label{diag:Kummer-compare-p-proof}
	\begin{tikzcd}[column sep=small]
	& J_Y(K)^{\wedge p}_{\bQ_p} \arrow[dl,"\sigma_{\fm}"'] \drar{\bar{\kappa}} & \\
	(A_{D,\fm})^{\wedge p}_{\bQ_p} \arrow[rr,"\varphi"] && \rH^1/\rH^1_f(G_K, V_p J_Y),
	\end{tikzcd}
	\end{equation}
	which we can show by hand:
	if $\hat{F}\in J_Y(K)^{\wedge p}_{\bQ_p}$, let $\hat{H}\in J(K)^{\wedge p}_{\bQ_p}$ denote its image, let $\hat{G} \coloneqq \hat{F} - \chi(\hat{H})\in T_D(K)^{\wedge p}_{\bQ_p}$ and write $\hat{F} = \hat{G} + \chi(\hat{H})$.
	Then
	\[
		\varphi\sigma_{\fm}(\hat{F}) = \varphi\sigma_{\fm}(\hat{G}) + \varphi\sigma_{\fm}\chi(\hat{H}) = \bar{\kappa}(\hat{G}) = \bar{\kappa}(\hat{F}) - \bar{\kappa}\chi(\hat{H}) = \bar{\kappa}(\hat{F}),
	\]
	where we used the functoriality of $\bar{\kappa}$ and that the image of $\chi$ is in the kernels of $\sigma_{\fm}$ and of $\bar{\kappa}$.
\end{proof}

The corollary below follows immediately from the commutativity of \eqref{diag:Kummer-compare-p}.

\begin{cor}
	\label{lem:compare-local-Selmer-p}
	Inside $J_Y(K)$ we have
	\begin{align*}
	\label{eq:Selmer-preimages-compare-p} \sigma_{\fm}^{-1}(\{0\}) \subseteq \kappa^{-1}(\rH^1_f(G_K,V_p J_Y)).
	\end{align*}
\end{cor}

\subsection{Comparison with local Néron symbols}
\label{sec:local-neron-symbols}

Assume that for each cusp $Q \in |D|$, the normalisation~$\tilde \cQ$ of~$\Spec R$ in~$k(Q)$ contains just a single closed point in the special fibre. This is satisfied for example when all cusps are $K$-rational or when~$R$ is henselian (e.g., complete). In this case, the $D$-intersection map defined in \Cref{sec:intersection map} has another interpretation in terms of local Néron symbols, which are related to local heights. Even though we do not use this explicitly in the following, we explain how our $D$-intersection map can be viewed as a generalisation of a certain other map encoding the local Néron pairings between divisors supported on~$Y = X \smallsetminus D$ and on~$D$.

The \emph{local Néron symbol} $\langle E,F\rangle_{\fm}$, when suitably normalised, is a rational number defined for any pair of degree-zero divisors $E,F \in \Div^0(X)$ with disjoint support. The pairing is symmetric and bilinear. It can be constructed using arithmetic intersection theory \cite[§III.5]{lang:arakelov}. Let~$\cX$ be a regular model of~$X$ over~$R$. Let $\cE$ and $\cF$ be the horizontal extensions of~$E$ and~$F$ on~$\cX$, and let $\Phi(E)$ and $\Phi(F)$ be vertical $\bQ$-divisors on~$\cX$ such that $\Psi(E) \coloneqq \cE + \Phi(E)$ and $\Psi(F) \coloneqq \cF + \Phi(F)$ have $\fm$-intersection number zero with every vertical divisor. Then the local Néron symbol is defined by
\[ \langle E, F \rangle_{\fm} \coloneqq i_{\fm}(\Psi(E), \Psi(F)). \]
This is independent of the choice of regular model \cite[§III.5]{lang:arakelov}. For a curve over a global field, the (suitably normalised) local Néron symbols provide the contributions at the finite places to the global Néron symbol, which is the negative of the real-valued Néron--Tate height pairing due to a theorem of Faltings \cite{faltings:calculus} and Hriljac \cite{hriljac}. Similarly, they provide the contributions at places $v \nmid p$ of the $\bQ_p$-valued Coleman--Gross height pairing \cite{CG89}. We refer to \cite[§4.1]{vBHM:explicit-arithmetic-intersection-theory} for a more detailed discussion.

Let $\Div^0(D)$ be the abelian group of degree zero divisors on~$X$ with support contained in~$D$, and let $V_{D,\fm}' \coloneqq \Div^0(D)_{\bQ}^\vee$ be the $\bQ$-vector space of linear maps $\Div^0(D)_{\bQ} \to \bQ$. We define a homomorphism $\sigma'_{\fm}\colon J_Y(K) \to V_{D,\fm}'$. On degree zero divisors $F \in \Div^0(Y)$ we set
\begin{equation}
	\label{eq:def-sigma'}
	\sigma'_{\fm}(F)(E) \coloneqq \langle E,F \rangle_{\fm},
\end{equation}
In other words, $\sigma'_{\fm}(F)$ is the datum of the value of the local Néron symbol $\langle E,F \rangle_{\fm}$ for every degree zero divisor~$E$ supported in~$D$. 

\begin{lemma}
	\label{lem:sigma'-well-defined}
	The map $\sigma'_{\fm}$ is well-defined on $D$-equivalence classes, hence descends to a homomorphism $\sigma'_{\fm}\colon J_Y(K) \to V_{D,\fm}'$.
\end{lemma}

\begin{proof}
	When $F = \div(f)$ for some rational function~$f$ with $f\vert_D = 1$, we have
	\[ \sigma'_{\fm}(\div(f))(E) = \langle E, \div(f) \rangle_{\fm} = v_{\fm}(f(E)) = v_{\fm}(1) = 0, \]
	where $v_{\fm}$ denotes the discrete valuation on~$K$, and $f(E) \coloneqq \prod_Q f(Q)^{n_Q}$ when $E = \sum n_Q [Q]$ over~$\Kbar$. Here we have $f(E) = 1$ since~$E$ has support in~$D$ and $f\vert_D = 1$.
\end{proof}

\begin{prop}
	\label{thm:neron-symbol-interpretation}
	Assume that for each cusp $Q \in |D|$, the normalisation~$\tilde \cQ$ of~$R$ in~$k(Q)$ contains just a single closed point in the special fibre. Then there is an isomorphism $V_{D,\fm} \cong V_{D,\fm}'$ which identifies the $D$-intersection map $\sigma_{\fm}\colon J_Y(K) \to V_{D,\fm}$ with the map $\sigma'_{\fm}\colon J_Y(K) \to V_{D,\fm}'$.
\end{prop}

\begin{proof}
	Recall that $V_{D,\fm} = Z_0(\tilde \cD)/[\tilde \cD_{\fm}] \otimes \bQ$. For a closed point $x \in \lvert\tilde \cD_{\fm}\rvert$, let $Q_x \in |D|$ be the cusp for which~$x$ lies on~$\tilde \cQ_x$. By assumption, the map $x \mapsto Q_x$ is a bijection between closed points on~$\tilde \cD$ and on~$D$, so we have an isomorphism
	\begin{equation}
		\label{eq:first-isom-V-V'}
		\phi\colon V_{D,\fm} \to Z_0(D)/[D_1] \otimes \bQ
	\end{equation}
	where $[D_1] \coloneqq \sum_{Q \in |D|} e_Q [Q]$ and $e_Q$ denotes the ramification index of $\tilde \cQ \to \Spec(R)$ at the unique closed point. We have another isomorphism
	\[ \psi \colon Z_0(D)/[D_1] \otimes \bQ \to V_{D,\fm}' \]
	given by $\psi(Q)(E) \coloneqq f_Q \mult_Q(E)$ for $Q \in |D|$ and $E \in \Div^0(D)$, where $f_Q$ denotes the degree of the residue field extension of $\tilde \cQ \to \Spec(R)$ at the unique closed point. To see that this is well-defined, note that
	\begin{equation}
		\label{eq:second-isom-V-V'}
		\psi([D_1])(E) = \sum_Q e_Q f_Q \mult_Q(E) = \sum_Q [k(Q) : K] \mult_Q(E) = \deg(E) = 0
	\end{equation}
	for $E \in \Div^0(D)$. An inverse of~$\psi$ is given by sending $g \in V_{D,\fm}'$ to $\sum_Q \tilde g(Q)/f_Q [Q]$ where $\tilde g$ is any extension of~$g$ from degree-zero divisors to arbitrary divisors supported in~$D$. The composition of the isomorphisms~\eqref{eq:first-isom-V-V'} and~\eqref{eq:second-isom-V-V'} is the claimed isomorphism $\psi \circ \phi \colon V_{D,\fm} \to V_{D,\fm}'$. To show that $\sigma'_{\fm} = \psi \circ \phi \circ \sigma_{\fm}$, suppose that $F$ is a degree-zero divisor on~$Y$ and $E$ is a degree-zero divisor on~$X$ with support contained in~$D$. Let~$\cX$ be a regular model of~$X$ over~$R$ and let $\Psi(E) = \cE + \Phi(E)$ and $\Psi(F) = \cF + \Phi(F)$ be as in the definition of the local Néron symbol. After blowing up~$\cX$ if necessary we may assume that the closure~$\cD$ of~$D$ in~$\cX$ is normal, so that $\tilde \cD = \cD$ and the generalised intersection numbers on~$\tilde \cD$ are actual intersection numbers on~$\cD$. We have
	\begin{align*}
		(\psi \circ \phi \circ \sigma_{\fm})(F)(E) &= (\psi \circ \phi)(\Psi(F).\cD)(E) \\
		&= \sum_{x \in \lvert \cD_{\fm} \rvert} i_x(\Psi(F), \cD) \psi([Q_x])(E) \\
		&= \sum_{x \in \lvert \cD_{\fm} \rvert} i_x(\Psi(F), \cQ_x) f_{Q_x} \mult_{Q_x}(E) \\
		&= \sum_{x \in \lvert \cD_{\fm} \rvert} f_{Q_x} i_x(\Psi(F), \cE) \\
		&= i_{\fm}(\Psi(F), \cE) \\
		&= i_{\fm}(\Psi(F), \Psi(E))\\
		&= \langle E, F \rangle_{\fm}\\
		&= \sigma'_{\fm}(F)(E),
	\end{align*}
	which shows that indeed $\sigma_{\fm}$ and $\sigma'_{\fm}$ are identified via the given isomorphism.
\end{proof}

When the assumption of \Cref{thm:neron-symbol-interpretation} is not satisfied, we can pass to a finite field extension $L/K$ over which all cusps become rational, so that the base change $X_L$ satisfies the assumption over all discrete valuation rings of $L$ dominating $R$.
This leads to a relation between the global $D$-intersection map and local N\'eron symbols as explained in \Cref{rem:global-sigma-Neron} below.

\section{Global theory: Selmer subsets of the generalised Jacobian}
\label{sec:global-theory}

Let $K$ be a number field, let $X$ be a smooth projective curve over~$K$, let $D \neq \emptyset$ be a finite set of closed points of~$X$, let $n\coloneqq D(\overline{K})$ be the number of geometric cusps, and set $Y = X \smallsetminus D$. The ring of integers of~$K$ is denoted by~$\cO_K$.
We let $n_1(D) \coloneqq \#D(\bR)$ be the number of real cusps and $n_2(D) = \frac12\# (D(\bC) \smallsetminus  D(\bR))$ the number of conjugate pairs of complex cusps, where $D$ is viewed as a scheme over $\bQ$ (rather than over $K$).
We thus have $ [K:\bQ]n = n_1(D) + 2n_2(D)$.
We start by defining a $\bQ$-vector space~$V_D$ and a homomorphism $\sigma\colon J_Y(K) \to V_D$. In \Cref{sec:S-integral-reduction-types} we define $S$-integral reduction types with respect to a regular model over~$\cO_K$ and use the map~$\sigma$ to constrain the image of points of a fixed reduction type under an Abel--Jacobi map $Y \to J_Y$.

\begin{defn}
	\label{def:normalisation-global}
	Let $\pi\colon \tilde\cD \to \Spec(\cO_K)$ be the normalisation of $\Spec(\cO_K)$ in~$D$. Concretely, if $D = \{Q_1,\ldots,Q_r\}$ and $\cO_{k(Q_i)}$ denotes the ring of integers of the residue field $k(Q_i)$, then
	\[ \tilde \cD = \coprod_{i=1}^r \Spec(\cO_{k(Q_i)}). \]
	We also write $\tilde\cQ_i = \Spec(\cO_{k(Q_i)})$.
\end{defn}

\subsection{The global $D$-intersection map}
\label{sec:global-D-intersection-map}

Consider the group $Z_0(\tilde \cD)$ of zero-cycles on~$\tilde \cD$. Any zero-cycle on $\Spec(\cO_K)$ (i.e.\ a fractional ideal of $\cO_K$) pulls back along $\pi\colon \tilde \cD \to \Spec(\cO_K)$ to a zero-cycle on~$\tilde \cD$.

\begin{defn}
	\label{def:V-global}
	Define the $\bQ$-vector space
	\[ V_D \coloneqq Z_0(\tilde \cD)/\pi^* Z_0(\Spec(\cO_K)) \otimes_{\bZ} \bQ. \]
\end{defn}

The local $D$-intersection map of \Cref{sec:intersection map} generalises to the global situation as follows. Choose a regular model $\cX$ of $X$ over~$\cO_K$. With the same definitions as in \Cref{sec:intersection-numbers}, we have intersection numbers $i_x(E,F)$ at closed points of~$\cX$ and we have an intersection cycle
\[ E.F = \sum_x i_x(E,F) [x] \]
for any two ($\bQ$-)divisors on~$\cX$ without common component in their support. We also have $\fq$-intersection numbers $i_{\fq}(E,F)$ for all primes~$\fq$ of $\cO_K$, defined whenever $E$ and $F$ have no common horizontal component in their support. Via the valuative criterion of properness, the inclusion $D \hookrightarrow \cX$ extends to a morphism of $\cO_K$-schemes
\[ \nu\colon \tilde \cD \to \cX, \]
and we define the generalised intersection numbers $i_x(E,\tilde \cD)$ at closed points~$x$ of~$\tilde \cD$ and the generalised intersection cycle $E.\tilde \cD = \sum_x i_x(E,\tilde \cD) [x] \in Z_0(\tilde \cD)$ just like in the local case. 

\begin{lemma}
	\label{lem:global-phi}
	For any degree-0 divisor~$F$ on~$Y$ with horizontal extension~$\cF$ on~$\cX$ there exists a vertical $\bQ$-divisor $\Phi(F)$ on~$\cX$ such that for all primes~$\fq$ of~$\cO_K$, the divisor $\Psi(F) \coloneqq \cF + \Phi(F)$ has $\fq$-intersection number zero with all vertical divisors over~$\fq$. The divisor $\Phi(F)$ is unique up to adding $\bQ$-divisors which are pulled back along $\cX \to \Spec(\cO_K)$.
\end{lemma}

\begin{proof}
	By \cite[Theorem~III.3.6]{lang:arakelov}, there exists, for each prime~$\fq$ of~$\cO_K$, a vertical $\bQ$-divisor $\Phi_{\fq}(F)$ on~$\cX$ over the prime~$\fq$ such that $\Psi_{\fq}(F) \coloneqq \cF + \Phi_{\fq}(F)$ has $\fq$-intersection number zero with every vertical divisor on~$\cX$ over~$\fq$. For primes~$\fq$ where~$\cX$ has good reduction, we can take $\Phi_{\fq}(F) = 0$. Then $\Phi(F) \coloneqq \sum_{\fq} \Phi_{\fq}(F)$ has the required properties. The uniqueness statement follows from the fact that $\Phi_{\fq}(F)$ is unique up to adding $\bQ$-multiples of the complete fibre~$\cX_{\fq}$.
\end{proof}

We first define $\sigma\colon J_Y(K) \to V_D$ on the level of divisors. Given a degree-0 divisor~$F$ on~$Y$, choose $\Phi(F)$ as in \Cref{lem:global-phi} and set $\Psi(F) \coloneqq \cF + \Phi(F)$. Then define
\begin{equation}
	\label{eq:global-sigma}
	\sigma(F) \coloneqq \Psi(F).\tilde \cD, %
\end{equation}
which is well-defined modulo $\pi^*Z_0(\Spec(\cO_K))$ and hence as an element of $V_D$ by the uniqueness statement in \Cref{lem:global-phi}.
The global variant of \Cref{lem:sigma-of-principal} reads as follows:

\begin{lemma}
	\label{lem:global-sigma-on-principal}
	Let $F = \div(f)$ be a principal divisor with support in~$Y$. Then we have
	\[ \sigma(F) = \sum_{Q \in |D|}\; \sum_{x \in \tilde\cQ_{\cl}} \ord_x(f(Q)) [x], \]
	where in the inner sum $x$ runs over the closed points of $\tilde\cQ = \Spec(\cO_{k(Q)})$. 
\end{lemma}

\begin{proof}
	We can choose $\Psi(F) = \div_{\cX}(f)$ and proceed as in the proof of \Cref{lem:sigma-of-principal}.
\end{proof}

As in the local case, it follows from \Cref{lem:global-sigma-on-principal} that the map~\eqref{eq:global-sigma} is well-defined on $D$-equivalence classes of divisors, i.e.\ it descends to $J_Y(K)$.

\begin{defn}
	\label{def:global-sigma}
	We call the homomorphism $\sigma\colon J_Y(K) \to V_D$, $[F] \mapsto \Psi(F).\tilde\cD$ the \emph{global $D$-intersection map}.
\end{defn}

\Cref{thm:sigma-independent-of-model} easily carries over to the global situation, so the global $D$-intersection map is also independent of the choice of regular model~$\cX$.
The global map~$\sigma$ can be viewed as the direct sum of local maps~$\sigma_{\fq}$ as follows. We have a natural isomorphism $Z_0(\tilde\cD) \cong \bigoplus_{\fq} Z_0(\tilde\cD_{\fq})$, with $\fq$ running over all primes of~$\cO_K$ and $\cD_{\fq}$ denoting the fibre of~$\cD$ over~$\fq$. With $V_{D,\fq} \coloneqq Z_0(\tilde\cD_{\fq})/[\tilde \cD_{\fq}] \otimes_{\bZ} \bQ$ this induces an isomorphism
\begin{equation}
	\label{eq:V-direct-sum}
	V_D \cong \bigoplus_{\fq} V_{D,\fq}.
\end{equation}
For each prime~$\fq$ of $\cO_K$, the theory of \Cref{sec:local-theory} applied to the (uncompleted) discrete valuation ring~$\cO_{K,\fq}$ yields a local $D$-intersection map $\sigma_{\fq}\colon J_Y(K) \to V_{D,\fq}$.

\begin{prop}
	\label{thm:global-sigma-sum-of-local}
	Under the identification~\eqref{eq:V-direct-sum}, the global $D$-intersection map $\sigma\colon J_Y(K)\to V_D$ equals the sum of the local maps~$\sigma_{\fq}\colon J_Y(K) \to V_{D,\fq}$ over all primes~$\fq$ of~$\cO_K$.
\end{prop}

\begin{proof}
	If $\cX$ is a regular model of~$X$ over~$\cO_K$, then for each prime~$\fq$ of~$\cO_K$, the base change $\cX_{\cO_{K,\fq}} = \cX \otimes_{\cO_K} \cO_{K,\fq}$ is a regular model over~$\cO_{K,\fq}$. The maps $\nu_{\fq}\colon \tilde\cD_{\cO_{K,\fq}} \to \cX_{\cO_{K,\fq}}$ used to define the generalised intersection numbers on $\cX_{\cO_{K,\fq}}$ are obtained by base change from $\nu\colon \tilde \cD \to \cX$. For a degree-0 divisor~$F$ on $Y$ and $\Psi(F) = \cF + \sum_{\fq} \Phi_{\fq}$ as in \Cref{lem:global-phi}, the part of $\sigma(F) = \Psi(F).\tilde\cD$ belonging to the fibre over~$\fq$ is $(\cF + \Phi_{\fq}(F)).\tilde\cD_{\cO_{K,\fq}}$, which is precisely $\sigma_{\fq}(F)$.
\end{proof}

\begin{prop}
	\label{thm:ker-sigma}
	The kernel of $\sigma\colon J_Y(K) \to V_D$ is finitely generated of rank
	\[ \rank \ker(\sigma) = n_1(D) + n_2(D) - \#|D| - \rank \cO_K^\times + r, \]
	where $n_1(D) = \#D(\bR)$ is the number of real cusps, $n_2(D)$ is the number of conjugate pairs of complex cusps, and $r = \rank J(K)$ is the Mordell--Weil rank of the Jacobian of ~$X$.
\end{prop}

\begin{proof}
	Recall from \Cref{sec:generalised-jacobian} that the generalised Jacobian fits into a short exact sequence
	\[ 0 \to T_D \to J_Y \to J \to 0, \]
	which remains exact on the level of $K$-points. The $K$-points of the torus $T_D$ are $T_D(K) = \cO(D)^\times/K^\times$. An element of $\cO(D)^\times = \prod_{Q \in |D|} k(Q)^\times$ defines a zero-cycle on each $\tilde \cQ$, hence a zero-cycle on~$\tilde\cD$. Elements of~$K^\times$ give rise to zero-cycles which are pulled back from~$\Spec(\cO_K)$. So we get a well-defined divisor map $\div\colon T_D(K) \to V_D$.
	This gives rise to the following diagram with exact rows:
	\[
	\begin{tikzcd}
		0 \rar & T_D(K)	\dar["\div"] \rar & J_Y(K) \dar["\sigma"] \rar & J(K) \rar \dar & 0 \\
		0 \rar & V_D \rar[equal] & V_D \rar & 0 & 
	\end{tikzcd}
	\]
	The left square is commutative. This follows from \Cref{lem:global-sigma-on-principal} and the fact that the map $T_D(K) \to J_Y(K)$ sends a tuple $(a_Q)_Q \in \prod k(Q)^\times$ to the class of the principal divisor $\div(f)$ defined by any rational function~$f$ on~$X$ that interpolates the values $f(Q) = a_Q$ for $Q \in |D|$. The snake lemma yields an exact sequence
	\begin{equation}
		\label{eq:ker-sigma-snake-lemma}
		0 \to \ker(\div) \to \ker(\sigma) \to J(K) \to \coker(\div) \to \coker(\sigma) \to 0.
	\end{equation}
	We claim that $\coker(\div)$ is torsion. Let $c$ be a zero-cycle on~$\tilde \cD$ with $\bQ$-coefficients, representing an element of~$V_D$. Let $m$ be a positive integer such that $mc$ has $\bZ$-coefficients. Write $mc = \sum_{Q\in|D|} c_Q'$ with $c_Q'$ a zero-cycle on $\tilde\cQ = \Spec \cO_{k(Q)}$. Let $h$ be the least common multiple of the class numbers $\#\Cl(k(Q))$ for $Q \in |D|$, so that $h c_Q'$ is a principal divisor on~$\tilde\cQ$, say $\div(a_Q)$ with $a_Q \in k(Q)^\times$. Then $mhc$ is the image of $(a_Q)_{Q}$ under $\div\colon T_D(K) \to V_D$, hence~$c$ is killed by $mh$ in $\coker(\div)$. This shows that the cokernel is torsion. Since $J(K)$ is finitely generated, its image in $\coker(\div)$ is finite, so the exact sequence~\eqref{eq:ker-sigma-snake-lemma} implies
	\begin{equation}
		\label{eq:ker-sigma-sum}
		\rank \ker(\sigma) = \rank \ker(\div) + \rank J(K).
	\end{equation}
	We are reduced to determining the rank of $\ker(\div)$. Consider the following commutative diagram with exact rows
	\[
	\begin{tikzcd}
		0 \rar & K^\times \dar \rar & \cO(D)^\times \rar \dar & T_D(K) \dar["\div"] \rar & 0 \\
		0 \rar & Z_0(\Spec(\cO_K)) \rar & Z_0(\tilde \cD) \rar & V_D
	\end{tikzcd}
	\]
	The snake lemma yields an exact sequence
	\[ 0 \to \cO_K^\times \to \cO(\tilde \cD)^\times \to \ker(\div) \to \Cl(K). \]
	Since the class group $\Cl(K)$ is finite, we get
	\begin{equation}
		\label{eq:ker-div}
		\rank \ker(\div) = \rank \cO(\tilde \cD)^\times - \rank \cO_K^\times.
	\end{equation}
	Finally, the rank of $\cO(\tilde \cD)^\times$ can be computed via Dirichlet's unit theorem. For $Q \in |D|$, let $n_1(Q)$ and $n_2(Q)$ be the number of real embeddings and the number of conjugate pairs of complex embedddings of $k(Q)$, respectively. Then we have
	\begin{align}
		\label{eq:rank-units}
		\rank \cO(\tilde\cD)^\times = \sum_{Q \in |D|} \rank \cO_{k(Q)}^\times = \sum_{Q \in |D|} (n_1(Q) + n_2(Q) - 1) = n_1(D) + n_2(D) - \#|D|. 
	\end{align}
	The claimed rank of $\ker(\sigma)$ is obtained by combining \eqref{eq:ker-sigma-sum}, \eqref{eq:ker-div}, \eqref{eq:rank-units}.
\end{proof}

The global $D$-intersection map behaves well under finite extensions of the base field:

\begin{lemma}
	\label{extending-base-field}
	Let $L/K$ be a finite extension. Let $\pi_L\colon \tilde \cD_L \to \Spec(\cO_L)$ be the normalisation of $\cO_L$ in $D_L \coloneqq D \otimes_K L$ and let $\pi_{L/K}\colon \tilde \cD_L \to \tilde \cD$ be the induced morphism of normalisations:
	\begin{equation}\label{diag:pi-commute} \begin{tikzcd}
		\tilde \cD_L \rar["\pi_{L/K}"] \dar["\pi_L"] & \tilde \cD \dar["\pi"] \\
		\Spec(\cO_L) \rar & \Spec(\cO_K).
	\end{tikzcd}  \end{equation}
	The map $\pi_{L/K}$ induces an injective homomorphism
	\[ \pi_{L/K}^* \colon V_D = Z_0(\tilde\cD)/\pi^* Z_0(\Spec(\cO_K)) \otimes_{\bZ} \bQ \to Z_0(\tilde \cD_L)/\pi_L^*Z_0(\Spec(\cO_L)) \otimes_{\bZ} \bQ = V_{D_L} \]
	making the following diagram commute:
	\begin{equation}\label{diag:sigma-BC} \begin{tikzcd}
		J_Y(K) \rar[hook] \dar["\sigma"] & J_Y(L) \dar["\sigma_L"] \\
		V_D \rar["\pi_{L/K}^*"] & V_{D_L}.
	\end{tikzcd}
	\end{equation}
\end{lemma}

\begin{proof}
	The commutativity of \eqref{diag:pi-commute} shows that $\pi_{L/K}^* \colon Z_0(\tilde\cD) \to Z_0(\tilde \cD_L)$ induces a homomorphism $V_D \to V_{D_L}$.
	Let $\cX$ be a regular model of $X$ over $\cO_K$ and let $\cX_L$ be a desingularisation of $\cX \otimes_{\cO_K} \cO_L$. By a slight abuse of notation, we denote the morphism $\cX_L \to \cX$ also by $\pi_{L/K}$, resulting in the following commutative diagram:
	\begin{equation}\label{diag:normal-commute} \begin{tikzcd}
		\tilde \cD_L \dar["\nu_L"] \rar["\pi_{L/K}"] & \tilde \cD \dar["\nu"] \\
		\cX_L \rar["\pi_{L/K}"] & \cX.
	\end{tikzcd}
	\end{equation}
	Fix $F \in \Div^0(Y)$, with image $F_L$ in $\Div^0(Y_L)$. Let $\Phi(F)$ be a vertical $\bQ$-divisor on~$\cX$ such that for all primes~$\fq$ of~$\cO_K$, the divisor $\Psi(F) \coloneqq \cF + \Phi(F)$ has $\fq$-intersection number zero with all vertical divisors on~$\cX$ over~$\fq$. By \cite[Theorem~III.4.5]{lang:arakelov}, the divisor $\Psi_L(F_L) \coloneqq \pi_{L/K}^* \Psi(F)$ on~$\cX_L$ has $\fq'$-intersection number zero with all vertical divisors on $\cX_L$ over~$\fq'$, for all primes~$\fq'$ of~$L$, so we can use it to compute $\sigma_L(F_L)$:
	\begin{align}\label{eq:sigmaL}
		\sigma_L(F_L) = \Psi_L(F_L).\tilde \cD_L = \pi_{L/K}^* \Psi(F).\tilde \cD_L
			= \sum_{x' \in (\tilde \cD_L)_{\cl}} i_{x'}(\pi_{L/K}^\ast \Psi(F),\tilde\cD_L) [x'].
	\end{align}
	Fix $x' \in (\tilde \cD_L)_{\cl}$ and let $x=\pi_{L/K}(x')\in\tilde{\cD}_{\cl}$.
	Let $f$ be an equation for $\Psi(F)$ in a neighbourhood of $\nu(x)$.
	Then $\pi_{L/K}^\ast f$ is an equation for $\pi_{L/K}^\ast\Psi(F) = \Psi_L(F_L)$ in a neighbourhood of $\nu_L(x')$.
	Using \eqref{eq:intersection-Dtilde} and the commutativity of \eqref{diag:normal-commute} we compute 
	\begin{align*}
		i_{x'}(\pi_{L/K}^\ast \Psi(F),\tilde\cD_L) = \ord_{x'}((\nu_L)^\ast \pi_{L/K}^\ast f) = \ord_{x'}(\pi_{L/K}^\ast \nu^\ast f) = e(x'|x) i_x(\Psi(F),\tilde{\cD}),
	\end{align*}
	where $e(x'|x)$ denotes the ramification index of $\pi_{L/K}\colon\tilde\cD_L \to \tilde\cD$ at $x'$.
	Plugging this into \eqref{eq:sigmaL} we get
	\begin{align*}
		\sigma_L(F_L) &= \sum_{x \in \tilde{\cD}_{\cl}} \sum_{x'|x} e(x'|x) i_x(\Psi(F),\tilde{\cD}) [x'] = \sum_{x \in \tilde{\cD}_{\cl}} i_x(\Psi(F),\tilde{\cD}) \pi_{L/K}^\ast([x])
		= \pi_{L/K}^\ast \sigma(F),
	\end{align*}
	which means that \eqref{diag:sigma-BC} commutes.
	
	For the injectivity of $\pi_{L/K}^\ast \colon V_D \to V_{D_L}$, we have to show that the diagram derived from \eqref{diag:pi-commute}
		\begin{equation*}\label{diag:pi-commute-Z0} \begin{tikzcd}
	Z_0(\Spec(\cO_K))\otimes_{\bZ}\bQ  \rar[hook] \dar["\pi^\ast", hook] & Z_0(\Spec(\cO_L))\otimes_{\bZ}\bQ \dar["\pi_{L}^\ast", hook] \\
	Z_0(\tilde \cD)\otimes_{\bZ}\bQ  \rar["\pi_{L/K}^\ast", hook] & Z_0(\tilde \cD_L)\otimes_{\bZ}\bQ.
	\end{tikzcd}  \end{equation*}
	is Cartesian.
	We may assume that $L/K$ is Galois.
	Now it is enough to note that the horizontal arrows map into the $\Gal(L/K)$-invariants, with the top arrow being an isomorphism onto the $\Gal(L/K)$-invariants.
\end{proof}

\begin{rem}\label{rem:global-sigma-Neron}
	The global $D$-intersection map is related to N\'eron symbols, and thus to (real-valued and $p$-adic) heights as follows.
	Choosing an extension $L/K$ over which all cusps become rational, \Cref{extending-base-field} together with \eqref{eq:V-direct-sum} shows that we have an injective map $V_D \hookrightarrow \bigoplus_{\fq'} \Div^0(D_{L})^\vee$ with $\fq'$ running over the primes of $L$.
	By \Cref{thm:neron-symbol-interpretation}, the composition
	\[
		J_Y(K) \xrightarrow{~\sigma~} V_D \hookrightarrow \bigoplus_{\fq'} \Div^0(D_{L})^\vee
	\]
	is given by
	\[
		[F] \mapsto \left(E \mapsto \langle E,F_{L} \rangle_{\fq'}\right)_{\fq'}.
	\]
	Thus $\sigma(F)$ recovers the values of all local N\'eron symbols $\langle E,F_{L} \rangle_{\fq'}$ with $E$ supported on the cusps and defined over $L$.
\end{rem}

\subsection{$S$-integral reduction types}
\label{sec:S-integral-reduction-types}

Let $S$ be a finite set of primes of~$\cO_K$. Let $\cX$ be a regular model of~$X$ over the ring of integers $\cO_K$, let $\cD$ be the closure of~$D$ in~$\cX$ and set $\cY = \cX \smallsetminus \cD$. We assume that $\cX$ is $D$-transversal over primes in~$S$, i.e.\ the base change
$\cX_{\cO_{K,\fq}}$ to the local ring~$\cO_{K,\fq}$ is $D$-transversal for all $\fq \in S$ (see \Cref{def:D-normal}).
The set of $S$-integral points $\cY(\cO_{K,S})$ can be partitioned according to their mod-$\fq$ reductions for all primes~$\fq$ as follows. 

\begin{defn}
	\label{def:S-integral-reduction-type}
	An \emph{$S$-integral reduction type} for $(\cX,\cD)$ is a collection $\Sigma = (\Sigma_{\fq})_{\fq}$ consisting of
	\begin{itemize}
		\item for each prime $\fq \not\in S$: an integral mod-$\fq$ reduction type $\Sigma_{\fq}$ for $(\cX,\cD)$, see \Cref{def:integral-reduction-type}; %
		\item for each prime $\fq \in S$: a rational mod-$\fq$ reduction type $\Sigma_{\fq}$ for $(\cX,\cD)$, see \Cref{def:rational-reduction-type}. %
	\end{itemize}
	An $S$-integral point $P \in \cY(\cO_{K,S})$ has reduction type~$\Sigma = (\Sigma_{\fq})_{\fq}$ if for all primes~$\fq$ of~$\cO_K$, the image of~$P$ under the mod-$\fq$ reduction map
	$ \red_{\fq}\colon \cY(\cO_{K,S}) \hookrightarrow \cX(\cO_{K,S}) = \cX(\cO_K) \to \cX(\bF_{\fq}) $
	lies on the component resp.\ is equal to the point~$\Sigma_{\fq}$. The set of all $S$-integral points of reduction type~$\Sigma$ is denoted by $\cY(\cO_{K,S})_{\Sigma}$.
\end{defn}

This defines a partition
\[ \cY(\cO_{K,S}) = \coprod_{\Sigma} \cY(\cO_{K,S})_{\Sigma} \]
with $\Sigma$ running over all $S$-integral reduction types for $(\cX,\cD)$. For a prime~$\fq$ where $\cX$ has good reduction and which is not contained in~$S$, there is only one possible choice for $\Sigma_{\fq}$, namely the unique component of the mod-$\fq$ fibre $\cX_{\fq} \smallsetminus \cD$. In particular, there are only finitely many $S$-integral reduction types for $(\cX,\cD)$.

\subsection{Selmer sets}
\label{sec:selmer-sets}

Let $\cX/\cO_K$ be as before and fix an $S$-integral reduction type $\Sigma = (\Sigma_{\fq})_{\fq}$ for $(\cX,\cD)$. We also fix a base point $P_0 \in Y(K)$.

Recall from §\ref{sec:global-D-intersection-map} the $\bQ$-vector space $V_D = Z_0(\tilde \cD)/\pi^* Z_0(\Spec(\cO_K)) \otimes_{\bZ} \bQ$ and its direct sum decomposition into the local vector spaces $V_{D,\fq} = Z_0(\tilde \cD_{\fq})/[\tilde \cD_{\fq}] \otimes_{\bZ} \bQ$ with $\fq$ running over all primes $\fq$ of~$\cO_K$:
\begin{equation*}
	\label{eq:V-sum-of-Vq}
	V_D \cong \bigoplus_{\fq} V_{D,\fq}.
\end{equation*}
For all primes~$\fq$ we have constructed a subset $\fS_{\fq}(P_0,\Sigma_{\fq}) \subseteq V_{D,\fq}$, see \Cref{def:local-selmer-set-integral} (if $\fq \not \in S$) and \Cref{def:local-selmer-set-rational} (if $\fq \in S$). They provide the local conditions to define a global subset~$\fS(P_0,\Sigma)$ of~$V_D$.

\begin{defn}
	\label{def:S-Sigma-global}
	Define the subset $\fS(P_0,\Sigma)$ of~$V_D$ by
	\[ \fS(P_0,\Sigma) \coloneqq \prod_{\fq} \fS_{\fq}(P_0,\Sigma_{\fq}). \]
	In other words, an element $c = (c_{\fq})_{\fq}$ of~$V_D$ belongs to $\fS(P_0,\Sigma)$ if and only $c_{\fq} \in \fS_{\fq}(P_0,\Sigma_{\fq})$ for all primes~$\fq$.
\end{defn}

Note that $\prod_{\fq} \fS_{\fq}(P_0,\Sigma_{\fq})$ is indeed contained in the direct sum of the vector spaces~$V_{D,\fq}$ since it follows from \Cref{thm:Sm-integral-zero} that $\fS_{\fq}(P_0,\Sigma_{\fq}) = \{0\}$ for all primes $\fq \not\in S$ such that $\cX$ has good reduction at~$\fq$ and $P_0$ is integral at~$\fq$.

\begin{defn}
	\label{def:C-Sigma}
	Denote by $C(\Sigma)$ the set of primes $\fq \in S$ for which $\Sigma_{\fq}$ is cuspidal, i.e.\ an element of $\cC_{\fq}^1 = (\cX_{\fq}^{\sm} \cap \cD)(\bF_{\fq})$. The \emph{cuspidal part} $\Sigma^{\csp}$ of~$\Sigma$ is its image under the projection to $\prod_{\fq \in C(\Sigma)} \cC_{\fq}^1$.
\end{defn}

\begin{prop}
	\label{thm:rank-of-S-Sigma}
	Assume that $D$ does not consist of a single $K$-rational point. Then $\fS(P_0,\Sigma)$ is a translate of a subgroup $U(\Sigma^{\csp})$ of $V_D$ which is free of rank~$\#C(\Sigma)$. 
\end{prop}

\begin{proof}
	The assumption on~$D$ ensures that the vector spaces $V_D$ and $V_{D,\fq}$ are non-trivial. For $\fq \in C(\Sigma)$, the set $\fS_{\fq}(P_0,\Sigma_{\fq})$ is a translate of the rank~1 subgroup $\bZ\cdot [\Sigma_{\fq}] \subseteq V_{D,\fq}$, see \Cref{rem:rank-leq1-subgroup}; for all other~$\fq$ it is a singleton.
\end{proof}

\begin{defn}
	\label{def:selmer-set}
	Define the \emph{Selmer set} associated to the reduction type~$\Sigma$ and base point~$P_0$ as the preimage of $\fS(P_0,\Sigma)$ in $J_Y(K)$ under the $D$-intersection map:
	\[ \Sel(P_0,\Sigma) \coloneqq \sigma^{-1}(\fS(P_0,\Sigma)) \subseteq J_Y(K). \]
\end{defn}

\begin{prop}
	\label{thm:map-to-selmer-set}
	The Abel--Jacobi map $\AJ_{P_0}\colon Y(K) \to J_Y(K)$ maps all points of reduction type~$\Sigma$ into $\Sel(P_0,\Sigma)$:
	\[ \AJ_{P_0}(\cY(\cO_{K,S})_{\Sigma}) \subseteq \Sel(P_0,\Sigma). \]
\end{prop}

\begin{proof}
	If $P$ has reduction type~$\Sigma = (\Sigma_{\fq})_{\fq}$ then, by definition, its mod-$\fq$ reduction type is $\Sigma_{\fq}$, for all primes~$\fq$ of~$\cO_K$. By the local theory, specifically \Cref{thm:abel-jacobi-integral-reduction-type} and \Cref{thm:abel-jacobi-rational-reduction-type}, this implies $\sigma_{\fq}(\AJ_{P_0}(P)) \in \fS_{\fq}(P_0,\Sigma_{\fq})$ for all~$\fq$. Since $\sigma_{\fq}$ equals the $\fq$-component of the global map~$\sigma$ (see \Cref{thm:global-sigma-sum-of-local}), the claim follows.
\end{proof}

\begin{prop}
	\label{thm:rank-of-selmer-set}
	Assume that $D$ does not consist of a single $K$-rational point. The Selmer set~$\Sel(P_0,\Sigma)$ is either empty or a translate of a subgroup of $J_Y(K)$ of rank $n_1(D) + n_2(D) - \#|D| - \rank \cO_K^\times + r + \#C(\Sigma)$.
\end{prop}

\begin{proof}
	By \Cref{thm:rank-of-S-Sigma}, the set $\fS(P_0,\Sigma)$ is a translate of a subgroup $U$ of $V_D$ which has rank~$\#C(\Sigma)$. If $\Sel(P_0,\Sigma)$ is not empty, then it is a translate of the subgroup $\sigma^{-1}(U)$ of $J_Y(K)$. It remains to determine the rank of $\sigma^{-1}(U)$. Let $\sigma_U: \sigma^{-1}(U) \to U$ be the restriction of~$\sigma$ and consider the short exact sequence
	\begin{equation}
		\label{eq:sigmaU-sequence}
		0 \lto \ker(\sigma) \lto \sigma^{-1}(U) \overset{\sigma_U}{\lto} U \lto \coker(\sigma_U) \lto 0.
	\end{equation}
	It was shown in the proof of \Cref{thm:ker-sigma} that the cokernel of $\div\colon T_D(K) \to V_D$ and thus also the cokernel of $\sigma$ are torsion. Since $\coker(\sigma_U)$ embeds into $\coker(\sigma)$, the cokernel of~$\sigma_U$ is torsion as well. As a quotient of~$U$ it is also finitely generated, hence finite. Now the exact sequence~\eqref{eq:sigmaU-sequence} shows
	\[ \rank \sigma^{-1}(U) = \rank \ker(\sigma) + \rank U \]
	and the claim follows from $\rank U = \#C(\Sigma)$ and the computation of the rank of $\ker(\sigma)$ in \Cref{thm:ker-sigma}.
\end{proof}

\begin{cor}
	\label{cor:image-of-AJ}
	Let $\cX/\cO_{K,S}$ be a regular model of $X$ over the ring of $S$-integers. Fix a base point $P_0 \in Y(K)$. The image of the set of $S$-integral points $\cY(\cO_{K,S})$ under the Abel--Jacobi embedding $\AJ_{P_0}\colon Y(K) \to J_Y(K)$ is contained in a finite union of translates of subgroups of rank at most $n_1(D) + n_2(D) - \#|D| - \rank \cO_K^\times + r + \#S$.
\end{cor}

\begin{proof}
	We can assume that the set of cusps~$D$ does not consist only of a single $K$-rational point, for otherwise the generalised Jacobian of~$Y$ coincides with the Jacobian of the compactification~$X$, whose group of $K$-points has rank~$r$ by definition. Let $\cX$ be a regular model of $X$ and extend it to a regular model over~$\Spec(\cO_K)$ that is $D$-transversal over primes in~$S$, compare \Cref{sec:normal-models}. The set $\cY(\cO_{K,S})$ is partitioned into finitely many subsets $\cY(\cO_{K,S})_{\Sigma}$ by reduction types (\Cref{def:S-integral-reduction-type}), and each non-empty $\cY(\cO_{K,S})_{\Sigma}$ is mapped under the Abel--Jacobi map to a translate of a subgroup of rank $n_1(D) + n_2(D) - \#|D| - \rank \cO_K^\times + r + \#C(\Sigma)$ by \Cref{thm:rank-of-selmer-set}. Here, $C(\Sigma)$ is the set of primes~$\fq \in S$ at which the local reduction type~$\Sigma_{\fq}$ is a point on the closure of a cusp. In particular, $\#C(\Sigma)$ is at most~$\#S$.
\end{proof}

\begin{rem}
	We compare the Selmer set $\Sel(P_0,\Sigma)\subseteq J_Y(K)$ with the \emph{cohomological} Selmer set $\Sel_{\Sigma,U_1} \subseteq \rH^1(G_K,V_p J_Y)$ of \cite[\S 6.2]{betts:effective}, where $U_1=V_p J_Y$ is the abelianisation of the $\bQ_p$-prounipotent fundamental group of $Y$. Assume that all primes $\fq\mid p$ are primes of good reduction for $(\cX,\cD)$, that $P_0$ is integral at those primes, and that they are not contained in~$S$. The set $\Sel_{\Sigma,U_1}$ is defined as the set of those global cohomology classes whose restriction at primes $\fq\mid p$ is crystalline, and whose restriction lies in $\fS^{\coh}_{\fq}(P_0,\Sigma_{\fq}) \subseteq \rH^1(G_{K_{\fq}},V_p J_Y)$ for all primes $\fq\nmid p$; see §\ref{sec:comparison-with-kummer-map} for the definition of $\fS_{\fq}^{\coh}(P_0,\Sigma_{\fq})$.
	(\cite[\S 6]{betts:effective} also assumes that the base point $P_0$ is $S$-integral and that $K=\bQ$ but the definition makes sense more generally.)
	Our assumptions ensure that there is only one reduction type $\Sigma_{\fq}$ at primes $\fq\mid p$ and $\fS_{\fq}(P_0,\Sigma_{\fq}) = \{0\}$. %
	Let $\kappa\colon J_Y(K) \to \rH^1(G_K, V_p J_Y)$ denote the global Kummer map, which is defined in the same way as in \eqref{eq:kummer-map} and compatible with the local Kummer maps.
	Now \Cref{lem:compare-local-Selmer} and \Cref{lem:compare-local-Selmer-p} show
	\begin{align}\label{eq:Selmer-compare}
	\Sel(P_0,\Sigma) \subseteq \kappa^{-1}(\Sel_{\Sigma,U_1}).
	\end{align}
	The Selmer set $\Sel(P_0,\Sigma)$ can thus be viewed as a geometric incarnation of the cohomological Selmer set. 
	In this sense our method refines the abelian affine case of \cite{betts:effective} that was also considered in \cite[\S 4]{LLM:LQChabAffine}. In particular, our method is at least as strong as Kim's non-abelian Chabauty method in depth~1 in its refined version by Betts--Dogra \cite{BD:refined}.
\end{rem}

\subsection{Reduction types for more general models}
\label{sec:general-models}

Let~$S$ be a finite set of primes of~$\cO_K$. Let $Y = X \smallsetminus D$ as before and let $\cX$ be a model of $X$ over~$\cO_K$ \cite[Definition~10.1.1]{liu2006algebraic} which is not necessarily regular and not necessarily $D$-transversal over primes in~$S$. Let~$\cD$ be the closure of~$D$ in~$\cX$ and set $\cY = \cX \smallsetminus \cD$. We can still partition the set of $S$-integral points~$\cY(\cO_{K,S})$ into subsets whose images under an Abel--Jacobi map $Y(K) \to J_Y(K)$ are strongly constrained, but we have to use reduction types on a suitably desingularised model.
Let $\pi\colon \cX' \to \cX$ be a morphism of models over~$\cO_K$ where $\cX'$ is regular and $D$-transversal over primes in~$S$.
Such a desingularisation can be obtained as follows: by repeatedly blowing up in the singular locus and normalising, one obtains a regular model \cite[Cor.~8.3.51]{liu2006algebraic}, and then one can use \Cref{thm:D-normal-desingularisation}.
Let $\cD'$ be the closure of~$D$ in~$\cX'$ and set $\cY' = \cX' \smallsetminus \cD'$. We have the inclusion $\cY(\cO_{K,S}) \subseteq \cY'(\cO_{K,S})$, which can be strict, for example if $\pi$ contracts a vertical component onto a point lying on~$\cD$. But by allowing only certain reduction types, we can find $\cY(\cO_{K,S})$ inside $\cY'(\cO_{K,S})$ and thereby reduce the computation of $S$-integral points on~$\cY$ to the case of a model which is regular and $D$-transversal over primes in~$S$. 

\begin{defn}
	\label{def:S-integral-reduction-type-on-desingularisation}
	An \emph{$S$-integral reduction type for $(\cX,\cD)$} (specified on~$\cX'$) is a collection $\Sigma = (\Sigma_{\fq})_{\fq}$ where $\Sigma_{\fq}$ is:
	\begin{itemize}
		\item for primes $\fq \not\in S$: an irreducible (equivalently: connected) component of ${\cX'_{\fq}}^{\sm} \smallsetminus \pi^{-1}(\cD)$ containing an $\bF_{\fq}$-point;
		\item for primes $\fq \in S$: either a component of ${\cX'_{\fq}}^{\sm} \smallsetminus \cD'$ containing an $\bF_{\fq}$-point, or a point in $({\cX'_{\fq}}^{\sm} \cap \cD')(\bF_{\fq})$.
	\end{itemize}
\end{defn}

For a prime~$\fq$, the composition
\[ Y(K) \overset{\red_{\fq}}{\lto} {\cX'_{\fq}}^{\sm}(\bF_{\fq}) \overset{\pi}{\lto} \cX_{\fq}(\bF_{\fq}) \]
of the mod-$\fq$ reduction map for the model~$\cX'$ with the morphism $\pi\colon \cX' \to \cX$ equals the mod-$\fq$ reduction map for~$\cX$, taking values in possibly singular $\bF_{\fq}$-points of~$\cX_{\fq}$. 
It is clear from the definition that the $S$-integral points of~$\cY$ are exactly those $K$-points which have an $S$-integral reduction type for $(\cX,\cD)$ as defined above.
Thus, computing $\cY(\cO_{K,S})$ is equivalent to computing $\cY'(\cO_{K,S})_{\Sigma}$ for all reduction types $\Sigma$ which are $S$-integral for $(\cX,\cD)$.

\section{$p$-adic integrals of logarithmic differentials}
\label{sec:integration}

Let $K$ be a complete subfield of $\bC_p$ (such as~$\bQ_p$ or a finite extension thereof), let $X$ be a smooth projective curve of genus~$g$ over~$K$, let $D \neq \emptyset$ be a finite set of closed points of~$X$ and set $Y = X \smallsetminus D$. Denote by $n \coloneqq D(\Kbar)$ the number of geometric cusps. A \emph{logarithmic differential} on $(X,D)$ is a meromorphic differential on~$X$ which is everywhere regular except for possibly simple poles at the points of~$D$. The $K$-vector space of logarithmic differentials on~$(X,D)$ is $\rH^0(X, \Omega^1(D))$. By the Riemann--Roch theorem, its dimension is given by
\begin{equation}
	\label{eq:dimension-of-log-differentials}
	\dim_K \rH^0(X, \Omega^1(D)) = g + n - 1 .
\end{equation}

\subsection{The abelian integral}
\label{sec:abelian-integral}

We recall the construction of $p$-adic integrals of log differentials. There are different possible approaches; for our purposes it is most convenient to work with the abelian integral constructed by Zarhin \cite{zarhin} and Colmez \cite{colmez}. With this approach it is not necessary to assume that~$X$ has good reduction.

The abelian integral is first constructed on the generalised Jacobian~$J_Y$ using the logarithm map of the $p$-adic Lie group $J_Y(K)$, and then pulled back to~$Y(K)$. The construction by Zarhin \cite{zarhin} requires the choice of a branch of the $p$-adic logarithm $\log\colon K^\times \to K$. This choice is uniquely determined by specifying $\log(p) \in K$. The construction by Colmez \cite{colmez} avoids such a choice by treating $\log(p)$ as a formal symbol, so that the abelian integrals have values in $K_\mathrm{st} \coloneqq K[\log(p)]$. This specialises to Zarhin's definition upon choosing a value for $\log(p) \in K$. For simplicity, we choose to work with the \emph{Iwasawa logarithm}, defined by $\log(p) = 0$. Once this choice is made, the logarithm map of the Lie group $J_Y(K)$, which a priori is only defined on an open neighbourhood of the identity, can be extended to all of $J_Y(K)$, resulting in a homomorphism
\[ \log_{J_Y} \colon J_Y(K) \to \Lie(J_Y). \]
The Lie algebra of $J_Y$ is the dual of the space of invariant differential forms on~$J_Y$. Assume that $Y(K) \neq \emptyset$. Pullback along an Abel--Jacobi map $Y \to J_Y$ defines an isomorphism between invariant differentials on~$J_Y$ and log differentials on $(X,D)$. This isomorphism does not depend on the base point of the Abel--Jacobi map. So the logarithm of~$J_Y$ can be viewed as a map
\begin{equation}
	\label{eq:log-JY}
	\log_{J_Y}\colon J_Y(K) \to \rH^0(X, \Omega^1(D))^\vee.
\end{equation}
This allows us to define, for any degree zero divisor~$F$ on $\Div^0(Y)$, representing an element of~$J_Y(K)$, and for any log differential~$\omega$ on $(X,D)$ the integral
\[ \int_F \omega \coloneqq \log_{J_Y}(F)(\omega) \in K. \]
Pulling this back to~$Y$, we can now define
\[ \int_P^Q \omega \coloneqq \int_{[Q]-[P]} \omega \]
for any two points $P,Q \in Y(K)$ and $\omega \in \rH^0(X, \Omega^1(D))$.

\begin{example}
	\label{ex:logarithm}
	On $Y = \bG_m = \bP^1 \smallsetminus \{0,\infty\}$ we have
	\[ \int_a^b \frac{\rd z}{z} = \log\left(\frac{b}{a}\right) \]
	for all $a,b \in K^\times$, where $\log\colon K^\times \to K$ is the chosen branch of the $p$-adic logarithm. 
\end{example}

\subsection{Tiny integrals}
\label{sec:tiny-integrals}

Assume that $K$ is a finite extension of~$\bQ_p$. Let $R$ be the valuation ring of~$K$ with maximal ideal~$\fm$ and finite residue field~$k$. Let $\cX$ be a regular model of~$X$ over $R$, let $\cD$ be the closure of~$D$ in~$\cX$ and set $\cY = \cX \smallsetminus \cD$. We have the mod-$\fm$ reduction map
\[ \red_{\fm}\colon \cY(R) \to \cY_{\fm}^{\sm}(k) \]
taking values in the smooth $k$-points of the special fibre, see \Cref{def:reduction-map}. The fibres of this map are called \emph{residue discs}. A $p$-adic integral whose endpoints lie in the same residue disc is called a \emph{tiny integral}. They can be computed by formally integrating a power series, as we now explain. 

Let $\tilde P \in \cY_{\fm}^{\sm}(k)$ be a smooth $k$-point on the special fibre of the affine curve. Let~$t$ be a rational function on~$\cX$ which is regular on a neighbourhood of~$\tilde P$ and reduces to a uniformiser of $\tilde P$ on $\cX_{\fm}$. This can always be achieved by rescaling~$t$ by an appropriate power of a uniformiser of~$K$. Let $U(\tilde P) = \red_{\fm}^{-1}(\tilde P)$ be the residue disc of~$\tilde P$. The uniformiser $t$ induces a bijection
\[ t\colon U(\tilde P) \xrightarrow{\sim} \fm. \]
Let $\omega$ be a log differential on~$(X,D)$. 
Then $\omega$ restricts to a holomorphic differential on the residue disc (in a rigid analytic sense), thus after multiplying by some constant in $K^\times$ we have $\omega = f(t) \rd t$ for some power series $f(t) = \sum_n a_n t^n \in \mathcal{O}_K[\![t]\!]$. %
For $P,Q\in U(\tilde{P})$, we can calculate the tiny integral by taking a formal antiderivative $F=\sum_n \tfrac{a_n}{n+1} t^{n+1}\in K[\![t]\!]$ of $f$:
\begin{align}\label{eq:tiny-int-power-series}
	\int_P^Q \omega = F(t(Q))-F(t(P)).
\end{align}

In particular, if $\omega\neq 0$ and if $P$ is fixed, then the function $Q\mapsto \int_P^Q \omega$ has only finitely many zeros inside the disc $U(\tilde{P})$.
Indeed, by \eqref{eq:tiny-int-power-series}, the function $Q\mapsto \int_P^Q \omega$ is given by evaluating the power series $F$ (plus a constant) at elements of $\fm$, and looking at the Newton polygon gives the result.
For a more precise version, see \Cref{lem:newtpolyg}.

\section{Construction of Chabauty functions}
\label{sec:chabauty-functions}

We now prove \Cref{thm:existence-of-chabauty-function-intro} over number fields and describe the affine variant of Restriction of Scalars Chabauty.
Let $K$ be a number field.
Let $X$ be a smooth projective curve over~$K$, let $D \neq \emptyset$ be a finite set of closed points of~$X$ and set $Y = X \smallsetminus D$.
Let $n=\# D(\overline{K})$.
Let $S$ be a finite set of primes of~$\cO_K$ and let $\cX$ be a regular model of~$X$ over~$\cO_K$ that is $D$-transversal over primes in~$S$.
Fix a base point~$P_0 \in Y(K)$ and an $S$-integral reduction type $\Sigma = (\Sigma_{\fq})_{\fq}$ for $(\cX,\cD)$ (see \Cref{def:S-integral-reduction-type}). We are interested in the set $\cY(\cO_{K,S})_{\Sigma}$ consisting of all $S$-integral points of reduction type~$\Sigma$. Recall that $C(\Sigma)$ denotes the set of primes $\fq \in S$ for which $\Sigma_{\fq}$ is cuspidal, and $\Sigma^{\csp} \coloneqq (\Sigma_{\fq})_{\fq \in C(\Sigma)}$ is called the cuspidal part of~$\Sigma$ (\Cref{def:C-Sigma}).

\subsection{Existence of Chabauty functions}
\label{sec:existence-of-chabauty-function}

\begin{thm}
	\label{thm:existence-of-chabauty-function-nf}
	With the notation as above, let~$\fp \not \in S$ be a prime of~$\cO_K$. Assume
	\begin{equation}
	\label{eq:chabauty-condition-over-number-field}
	r + \#C(\Sigma) + ([K : \bQ]-1)n < g + \rank \cO_K^\times + \#|D| + n_2(D) - 1.
	\end{equation}
	Then there exist a non-zero log differential $0 \neq \omega \in \rH^0(X_{K_{\fp}}, \Omega^1(D))$ and a constant $c \in K_{\fp}$, both depending on~$\Sigma$, such that the function $\rho\colon \cY(\cO_{\fp}) \to K_{\fp}$ given by
	\begin{equation}
	\label{eq:rho-function}
	\rho(P) \coloneqq \int_{P_0}^P \omega - c
	\end{equation}
	vanishes on $\cY(\cO_{K,S})_{\Sigma}$. The differential~$\omega$ can be chosen uniformly for all reduction types having the same cuspidal part.
\end{thm}

Hypothesis \eqref{eq:chabauty-condition-over-number-field} specialises to \eqref{eq:chabauty-condition-intro} for $K=\bQ$, so \Cref{thm:existence-of-chabauty-function-intro} is a special case of \Cref{thm:existence-of-chabauty-function-nf}.
Note that the function $\rho\colon \cY(\cO_{\fp}) \to K_{\fp}$ in~\eqref{eq:rho-function} has only finitely many zeros by \S\ref{sec:tiny-integrals}.

\begin{proof}
	If $D$ is just a single $K$-rational point, Condition~\eqref{eq:chabauty-condition-over-number-field} becomes $r + \#C(\Sigma) < g$, and already the weaker condition $r < g$ is sufficient for the classical Chabauty method to apply, yielding a non-zero holomorphic differential $0 \neq \omega \in \rH^0(X_{K_{\fp}}, \Omega^1)$ such that $\int_{P_0}^P \omega = 0$ for all rational points $P \in X(K)$. Hence, from now on assume~$n > 1$. 
	In \Cref{sec:selmer-sets} we constructed the Selmer set $\Sel(P_0,\Sigma) \subseteq J_Y(K)$ with the property that the Abel--Jacobi map $\AJ_{P_0}\colon \cY(\cO_{K,S}) \to J_Y(K)$ maps all points of reduction type~$\Sigma$ into $\Sel(P_0,\Sigma)$ (see \Cref{thm:map-to-selmer-set}). 
	The set $\Sel(P_0,\Sigma)$ is the preimage under $\sigma\colon J_Y(K) \to V_D$ of a translate of a certain subgroup $U = U(\Sigma^{\csp}) \subseteq V_D$.
	Let $W \coloneqq \sigma^{-1}(U) \subseteq J_Y(K)$, so that $\Sel(P_0,\Sigma)$ is either empty or a translate $s_0 + W$ of~$W$.
	Consider the following diagram:
	\begin{equation}
		\label{eq:chabauty-diagram}
		\begin{tikzcd}
			&[-15pt] \cY(\cO_{K,S})_{\Sigma} \dar["\AJ_{P_0}"] \rar[hook] & \cY(\cO_{K,S}) \dar["\AJ_{P_0}"] \rar[hook] & \cY(\cO_{\fp}) \dar["\AJ_{P_0}"] \drar["\int_{P_0}"]& \\
			s_0 + W \rar[equal] & \Sel(P_0, \Sigma) \rar[hook] & J_Y(K) \rar[hook] & J_Y(K_{\fp}) \rar["\log_{J_Y}"] & \rH^0(X_{K_{\fp}}, \Omega^1(D))^\vee
		\end{tikzcd}
	\end{equation}
	When $\rank_{\bZ} W$ is strictly less than $\dim_{K_{\fp}} \rH^0(X_{K_{\fp}}, \Omega^1(D))$, then the image of~$W$ in $\rH^0(X_{K_{\fp}}, \Omega^1(D))^\vee$ under the logarithm map is contained in a proper $K_{\fp}$-linear subspace, so there exists a non-zero $0 \neq \omega \in \rH^0(X_{K_{\fp}}, \Omega^1(D))$ with $\langle \log_{J_Y}(w), \omega \rangle = 0$ for all $w \in W$. Then, since $\Sel(P_0,\Sigma)$ is a translate of~$W$ (or empty), the map $s \mapsto \langle \log_{J_Y}(s), \omega \rangle$ is constant on $\Sel(P_0,\Sigma)$. Pulling this back under the diagonal map in the diagram~\eqref{eq:chabauty-diagram} shows that the map $P \mapsto \int_{P_0}^P \omega$ is constant on $\cY(\cO_{K,S})_{\Sigma}$.
	By \Cref{thm:rank-of-selmer-set}, the condition that the rank of~$W$ be strictly less than the dimension of $\rH^0(X_{K_{\fp}}, \Omega^1(D))$ reads
	\[ n_1(D) + n_2(D) - \#|D| - \rank \cO_K^\times + r + \#C(\Sigma) < g + n -1, \]
	which, using $n_1(D) + 2n_2(D) = [K:\bQ] n$, simplifies to \eqref{eq:chabauty-condition-over-number-field}.
\end{proof}

Of course, when the base point~$P_0$ is itself an $S$-integral point of reduction type~$\Sigma$, then the constant~$c$ in~\eqref{eq:rho-function} is necessarily zero.

\begin{rem}\label{rem:more-red-types-than-eta-c}
	The proof above shows that the differential $\omega$, the constant $c$ and hence the function $\rho$ in \eqref{eq:rho-function} depend on~$\Sigma$ only through $\Sel(P_0,\Sigma)$.
	That is, if $\Sel(P_0,\Sigma)=\Sel(P_0,\Sigma')$ for two reduction types $\Sigma$ and $\Sigma'$, then we get a single Chabauty function $\rho$ that vanishes on $\cY(\cO_{K,S})_\Sigma \cup \cY(\cO_{K,S})_{\Sigma'}$.
	The equality $\Sel(P_0,\Sigma)=\Sel(P_0,\Sigma')$ is guaranteed if  $\fS_{\fq}(P_0,\Sigma)=\fS_{\fq}(P_0,\Sigma')$ for all $\fq$.
\end{rem}

\subsection{Affine Restriction of Scalars Chabauty}
\label{sec:number-fields}

Instead of locating $\cY(\cO_{K,S})$ inside $\cY(\cO_{\fp})$ for some suitable prime~$\fp$ of $\cO_K$ as in \Cref{sec:existence-of-chabauty-function}, it can be advantageous to consider all primes~$\fp$ of~$\cO_K$ lying over a fixed rational prime~$p$ at once and try to locate $\cY(\cO_{K,S})$ inside $\prod_{\fp\mid p} \cY(\cO_{\fp})$. This method is known as Restriction of Scalars Chabauty. It appeared first in \cite{siksek:number-field-chabauty} in the context of rational points on projective curves.

The relevant diagram for Restrictions of Scalars Chabauty in our setting looks as follows:
\begin{equation}
\label{eq:ros-chabauty-diagram}
\begin{tikzcd}[row sep=large]
\cY(\cO_{K,S})_{\Sigma} \dar["\AJ_{P_0}"] \rar[hook] & \cY(\cO_{K,S}) \dar["\AJ_{P_0}"] \rar[hook] & \prod_{\fp\mid p} \cY(\cO_{\fp}) \dar["\prod_{\fp\mid p} \AJ_{P_0}"] \drar["\prod_{\fp\mid p}\int_{P_0}"]& \\
\Sel(P_0, \Sigma) \rar[hook] & J_Y(K) \rar[hook] & \prod_{\fp\mid p} J_Y(K_{\fp}) \rar["\log_{J_Y}"] & \prod_{\fp\mid p} \rH^0(X_{K_{\fp}}, \Omega^1(D))^\vee
\end{tikzcd}
\end{equation}
Here, $p$ is a rational prime not divisible by any prime in~$S$.
Let $d \coloneqq [K : \bQ]$.
Since $\prod_{\fp\mid p} \cY(\cO_{\fp})$ is now $d$-dimensional as a rigid analytic space over~$\bQ_p$, at least $d$ independent functions are needed to cut out a finite subset. The following theorem shows under which condition we find the right number of functions.

\begin{thm}
	\label{thm:ros-chabauty}
	Let $p$ be a rational prime as above. Assume that the inequality
	\begin{equation}
	\label{eq:ros-chabauty-condition}
	r + \#C(\Sigma) \leq d(g-2) + n_2(D) + \#|D| + \rank \cO_K^\times
	\end{equation}
	is satisfied. Then there exist $d$ linearly independent (over~$\bQ_p$) tuples of logarithmic differentials $(\omega_{\fp}^{(k)})_{\fp} \in \prod_{\fp\mid p} \rH^0(X_{K_{\fp}}, \Omega^1(D))$ and $p$-adic constants $c_k \in \bQ_p$ ($k=1,\ldots,d$) such that the functions $\rho_k\colon \prod_{\fp\mid p} \cY(\cO_{\fp}) \to \bQ_p$ given by
	\begin{equation}
	\label{eq:ros-chabauty-function}
	\rho_k((P_{\fp})_{\fp\mid p}) = \sum_{\fp\mid p} \Tr_{K_{\fp}/\bQ_p} \left(\int_{P_0}^{P_{\fp}} \omega^{(k)}_{\fp}\right) - c_k
	\end{equation}
	vanish on the image of the diagonal embedding of $\cY(\cO_{K,S})_{\Sigma}$ in $\prod_{\fp\mid p} \cY(\cO_{\fp})$.
\end{thm}

\begin{proof}
	If $D$ consists just of a single $K$-rational point, Condition~\eqref{eq:ros-chabauty-condition} is equivalent to $r + \#C(\Sigma) \leq d(g-1)$, and already the weaker condition $r \leq d(g-1)$ is enough to employ the classical Restriction of Scalars Chabauty method for the $K$-rational points on~$X$. So we can assume that $n > 1$. Consider the diagram~\eqref{eq:ros-chabauty-diagram}. The Selmer set $\Sel(P_0,\Sigma)$ constructed in \Cref{sec:selmer-sets} is the preimage under $\sigma\colon J_Y(K) \to V_D$ of a translate of a certain subgroup $U = U(\Sigma^{\csp}) \subseteq V_D$. Let $W \coloneqq \sigma^{-1}(U) \subseteq J_Y(K)$, so that $\Sel(P_0,\Sigma)$ is either empty or a translate of~$W$. Let $d'$ be the codimension of the image of~$W \otimes_{\bZ} \bQ_p$ in the $\bQ_p$-vector space $\prod_{\fp\mid p} \rH^0(X_{K_{\fp}}, \Omega^1(D))^\vee$ along the bottom row of Diagram~\eqref{eq:ros-chabauty-diagram}. Then there exist $d'$ independent $\bQ_p$-linear maps 
	\begin{equation}
	\label{eq:fk-maps}
	f_k\colon \prod_{\fp\mid p} \rH^0(X_{K_{\fp}}, \Omega^1(D))^\vee \to \bQ_p \qquad (k=1,\ldots,d')
	\end{equation}
	vanishing on the image of~$W$.
	For a finite-dimensional $K_{\fp}$-vector space~$H$ one has a canonical isomorphism of $\bQ_p$-vector spaces
	\[ \Res_{K_{\fp}/\bQ_p}(H^\vee) \cong \Res_{K_{\fp}/\bQ_p}(H)^\vee \]
	via $\phi \mapsto \Tr_{K_{\fp}/\bQ_p} \circ \phi$. Here, $\Res_{K_{\fp}/\bQ_p}$ means we view a $K_{\fp}$-vector space as a $\bQ_p$-vector space via restriction of scalars, and $\Tr_{K_{\fp}/\bQ_p}$ denotes the trace map $K_{\fp} \to \bQ_p$. Applying this with $H = \rH^0(X_{K_{\fp}}, \Omega^1(D))$ and taking duals, one sees that each of the linear maps $f_k$ in \eqref{eq:fk-maps} is given by 
	\[ f_k\colon (\phi_{\fp})_{\fp} \mapsto \sum_{\fp\mid p} (\Tr_{K_{\fp}/\bQ_p}\circ \phi_{\fp})(\omega_{\fp}^{(k)}) \]
	for a unique tuple of log differentials $(\omega_{\fp}^{(k)})_{\fp} \in \prod_{\fp\mid p} \rH^0(X_{K_{\fp}}, \Omega^1(D))$. Since $\Sel(P_0,\Sigma)$ is a translate of $W$ (or empty), each map $f_k$ is constant on the Selmer set, say with value~$c_k \in \bQ_p$. Pulling $f_k$ back along the diagonal arrow in Diagram~\eqref{eq:ros-chabauty-diagram} results in the map
	\[ \prod_{\fp\mid p} \cY(\cO_{\fp}) \to \bQ_p, \qquad (P_{\fp})_{\fp\mid p} \mapsto \sum_{\fp\mid p} \Tr_{K_{\fp}/\bQ_p} \left(\int_{P_0}^{P_{\fp}} \omega^{(k)}_{\fp}\right) \]
	whose restriction to $\cY(\cO_{K,S})$ is constant with value $c_k$ by commutativity of the diagram. Hence, we find $d'$ many functions $\rho_k$ of the claimed form. It remains to see that $d' \geq d$. %
	By definition, $d'$ can be bounded below by
	\begin{align*}
	d (g + n - 1) - \rank W = d(g-1) -r + n_2(D) + \#|D| + \rank \cO_K^\times - \#C(\Sigma),
	\end{align*}
	where we used $n_1(D) + 2n_2(D) = dn$ and \Cref{thm:rank-of-selmer-set}.
	Now \eqref{eq:ros-chabauty-condition} ensures that $d' \geq d$.
\end{proof}

\begin{rem}
	\label{rem:ros-choice-of-prime}
	When applying \Cref{thm:ros-chabauty} it may be convenient to choose a prime~$p$ that is totally split in~$K$, so that $K_{\fp} = \bQ_p$ for all $\fp\mid p$. Then the integrals occurring in the Chabauty functions \eqref{eq:ros-chabauty-function} are all defined over~$\bQ_p$ and the trace map can be omitted. Most implementations of Coleman integration in computer algebra systems are currently restricted to~$\bQ_p$.
\end{rem}

\begin{rem}
	\label{rem:special-case-linear-quadratic-chabauty-over-nf}
	In the case that $\cY\colon y^2 = f(x)$ with $f \in \cO_K[x]$ monic, squarefree and of degree $2g+2$, condition~\eqref{eq:ros-chabauty-condition} with $S = \emptyset$ specialises to $\rank J(K) \leq dg - \rank \cO_K^\times$. In this setting, functions $\prod_{\fp\mid p} \cY(\cO_{\fp}) \to \bQ_p$ vanishing on $\cY(\cO_K)$ are constructed in terms of $p$-adic heights in \cite[Theorem~5.1]{GM:LinearQuadraticChabauty}.
\end{rem}

\begin{rem}
	\label{rem:unlikely-intersections}
	Since $\prod_{\fp\mid p} \cY(\cO_{\fp})$ is a $d$-dimensional rigid analytic space, at least~$d$ linearly independent Chabauty functions are needed to cut out a finite subset. However, having~$d$ functions does not guarantee finiteness since the functions could conspire to have many zeros in common. The phenomenon of such unlikely intersections was first observed by Siksek \cite[§2]{siksek:number-field-chabauty} and studied in detail by Dogra \cite{dogra:unlikely-intersections}. For affine curves of genus zero, limitations to the Restrictions of Scalars Chabauty--Skolem method (see \Cref{rem:circle}) were studied in \cite{triantafillou:ROS-Chabauty}.
\end{rem}

\section{Bounds on the number of $S$-integral points}
\label{sec:bounds}

Suppose we are in the setting of \Cref{sec:chabauty-functions} with $K=\bQ$, so we have a curve $Y = X \smallsetminus D$ over $\bQ$, a finite set of primes~$S$, a regular model $\cY = \cX \smallsetminus \cD$ over~$\bZ$ that is $D$-transversal over primes in~$S$, an auxiliary prime $p \not\in S$ and a base point $P_0 \in Y(\bQ)$. If \eqref{eq:uniform-chabauty-condition-intro} holds, then \Cref{thm:existence-of-chabauty-function-intro} applies for every $S$-integral reduction type $\Sigma$. 
Consequently, the set of $S$-integral points $\cY(\bZ_S)$ is contained in a finite union of zero loci of functions $\rho\colon \cY(\bZ_p) \to \bQ_p$ of the form \eqref{eq:rho-function-intro}. %
We give a bound on the number of zeros of such a function, extending the results of Coleman as outlined in \cite[\S 5]{mcpoonen} for integrals of holomorphic differentials to those of logarithmic differentials. %
Following \cite[Appendix A]{mcpoonen}, we allow $p$ to be a prime of bad reduction.

\subsection{Reduction of log differentials}
We start by discussing the reduction of log differentials modulo~$p$. This requires some care when~$p$ is a prime of bad reduction since the sheaf of Kähler differentials on $\cX_{\bZ_p}$ fails to be a line bundle in this case and one has to pass through the canonical sheaf instead. We work in a general local setting. 

As in \S\ref{sec:intersection-numbers}, let $(R,\fm,k)$ be a discrete valuation ring with uniformizer $\pi$ and fraction field $K$.
Let $X/K$ be a smooth projective curve and let $\cX/R$ be a regular model of $X$ with special fibre $\cX_{\fm}$.
Let $\emptyset\neq D\subseteq X$ be a finite set of closed points, $n=\# D(\overline{K})$, and $\cD$ the Zariski closure of $D$ in $\cX$.
Let $Y = X \smallsetminus D$ and $\cY = \cX \smallsetminus \cD$.
Let $0\neq\omega\in \rH^0(X, \Omega^1(D))$ be a nonzero log differential. The canonical sheaf $\omega_{\cX/R}$ is a line bundle on the model~$\cX$ whose restriction to the generic fibre agrees with the sheaf $\Omega^1_{X/K}$ of Kähler differentials, so we can view $\omega$ as a meromorphic section of $\omega_{\cX/R}$ and let $\cK=\div_{\cX}(\omega)$ be the associated canonical divisor.
Write $\cK = \cH + \cV$ for a horizontal divisor $\cH$ and a vertical divisor $\cV$.
Note that, if $\div(\omega)=\sum_{P\in |X|} a_P P$, then $\cH = \sum_{P\in |X|} a_P \cP$.
In particular, since $\omega$ is a logarithmic differential, we have $a_P \geq 0$ for $P \in |Y|$ and $a_P \geq -1$ for $P\in D$, so $\cH = \cH_e - \cH_p$ with $\cH_e$ effective and $\cH_p = \sum_{P\in D'} \cP$ with $D'\subseteq D$ consisting of the poles of $\omega$.

Let $C$ be a component of $\cX_{\fm}$ that occurs with multiplicity~$1$ in $\div_{\cX}(\pi)=\cX_{\fm}$.
Then we can multiply $\omega$ with an appropriate power of $\pi$ to obtain a log differential $\omega^C$ such that $C$ does not occur in $\cK^C \coloneqq \div_{\cX}(\omega^C)$.
We again write $\cK^C = \cH + \cV^C$ with $\cH$ the same horizontal divisor as before and $\cV^C$ vertical.
We reduce $\omega^C$ to $C^{\sm} = C \cap \cX^{\sm}$ as follows:
By removing all components and singular points of $\cX_{\fm}$ except $C^{\sm}$, we obtain an open subscheme $\cX^C \subseteq \cX$.
We may view $\omega^C$ as a (meromorphic) differential form on $\cX^C$.
As $C$ does not occur in $\cK^C$, the restriction $\widetilde{\omega}^C \coloneqq \omega^C|_{C^{\sm}}$ is a nonzero (meromorphic) differential form on $C^{\sm}$.
Note that the restriction of $\widetilde{\omega}^C$ to $C^{\sm} \cap \cY$ is holomorphic.

In conclusion, for every nonzero log differential $0 \neq \omega \in \rH^0(X, \Omega^1(D))$ and for every multiplicity-1 component $C$ of~$\cX_{\fm}$, there is a scalar multiple~$\omega^C$ of~$\omega$ which reduces to a nonzero meromorphic differential $\widetilde{\omega}^C$ on~$C^{\sm}$ whose restriction to $C^{\sm} \cap \cY$ is holomorphic.

\subsection{Bounding zeros of a log differential modulo~$p$}
Keep the notation from the previous subsection. We need two lemmas on the number of zeros of the reduction of a log differential. Given a multiplicity-1 component $C$ of $\cX_{\fm}$ and given~$\omega$ with reduction $\widetilde{\omega}^C$ on~$C^{\sm}$, let 
\[ n_C \coloneqq \sum_{\widetilde{R} \in (C^{\sm} \cap \cY)(k)} \ord_{\widetilde{R}}(\widetilde{\omega}^C), \]
so $n_C$ is the number of zeros of $\widetilde{\omega}^C$ in $(C^{\sm} \cap \cY)(k)$ counted with multiplicities.

\begin{lemma}\label{lem:A3}
	We have $n_C \leq \cH_e \cdot C$.
\end{lemma}

\begin{proof}
	For $P \in (C^{\sm} \cap \cY)(k)$ we have
	\[
		i_P(\cK^C,C) = i_P(\cH,C) = i_P(\cH_e,C) - i_P(\cH_p,C) = i_P(\cH_e,C),
	\]
	where the first equality is true because $P$ does not lie on any component of $\cV^C$ as $C$ does not occur in $\cV^C$, and the last equality is true because $P\in \cY(k)$ does not lie in the horizontal closure of a pole.
	As $\cH_e$ is effective, summing over all $P$ yields
	\[
		n_C = \sum_{P\in (C^{\sm} \cap \cY)(k)} i_P(\cK^C,C) \leq \sum_{P\in C} i_P(\cH_e,C) [k(P):k] = \cH_e \cdot C. \qedhere
	\]
\end{proof}

\begin{lemma}\label{lem:A4}
	With the above notation, we have
	\[
		\sum_{C \textnormal{ of multiplicity } 1} n_C \leq 2g-2+n.
	\]
\end{lemma}

\begin{proof}
	The divisor $\cH_e$ is effective, hence $\cH_e\cdot C\geq 0$ for every component $C$ of $\cX_{\fm} = \sum_C m_C C$.
	By \Cref{lem:A3}, we have
	\[
		\sum_{C \text{ of multiplicity } 1} n_C \leq \sum_{C \text{ of multiplicity } 1} \cH_e\cdot C \leq \sum_{\text{ all } C} m_C \cH_e\cdot C = \cH_e \cdot \cX_{\fm} = (\cK + \cH_p - \cV) \cdot \cX_{\fm}.
	\]
	Here $\cK\cdot \cX_{\fm} = 2g-2$ by the adjunction formula \cite[Proposition 9.1.35]{liu2006algebraic}, and $\cV\cdot\cX_{\fm} = 0$ as $\cV$ is vertical \cite[Proposition 9.1.21(a)]{liu2006algebraic}.
	Finally, 
	\[ \cH_p\cdot \cX_{\fm} = \sum_{P \in D'} \cP\cdot \cX_{\fm} = \sum_{P \in D'} [k(P):K] \leq \sum_{P \in D} [k(P):K] = n \]
	by \cite[Proposition 9.1.30]{liu2006algebraic}, which finishes the proof.
\end{proof}

\subsection{Bounding the number of zeros of a logarithmic abelian integral}
We stay in the local setting as above but specialise to the case $R = \bZ_p$, so $Y = X \smallsetminus D$ is a curve over~$\bQ_p$ with regular model $\cY = \cX \smallsetminus \cD$ over~$\bZ_p$, and $\omega \in \rH^0(X, \Omega^1(D))$ is a nonzero log differential. Fix a base point $P_0 \in Y(\bQ_p)$. We prove a bound on the number of points~$P$ in a residue disc for which $\int_{P_0}^P \omega$ is equal to a prescribed constant. We need the following lemma, which is taken from \cite[Lemma 5.1]{mcpoonen} and proved by looking at the Newton polygon. %
\begin{lemma}\label{lem:newtpolyg}
	Let $f(t)\in\bQ_p[\![t]\!]$ be such that $f'(t)\in\bZ_p[\![t]\!]$ and let $m$ be the order of vanishing of $f'(t) \pmod p$ at $t=0$.
	If $m<p-2$, then $f$ has at most $m+1$ zeros in $p\bZ_p$.
\end{lemma}

\begin{prop}
	\label{lem:integral-bound-disc}
	Let $C \subseteq \cX_{p}^{\sm}$ be a multiplicity-1 component, let $\widetilde{R} \in (C^{\sm} \cap \cY)(\bF_p)$ and let $m = \ord_{\widetilde{R}}(\widetilde{\omega}^C)$. Assume $m < p-2$. Then, for any $c\in \bQ_p$, the number of $P \in \cY(\bZ_p)$ reducing to $\widetilde{R}$ and satisfying $\int_{P_0}^P \omega = c$ is at most $m+1$.
\end{prop}

\begin{proof}
	We may rescale $\omega$ (and $c$) and assume $\omega = \omega^C$. Note that $m\geq 0$ as $\widetilde{R}$ is not contained in the support of $\cD_{p}$.
	By Hensel's lemma, $\widetilde{R}$ lifts to a point $R\in\cY(\bZ_p)$. Let $t$ be a rational function on $X$ that reduces to a uniformiser on $\cX_{p}$ at $\widetilde{R}$.
	The function $t$ maps the residue disc of $\widetilde{R}$ bijectively to $p\bZ_p$.
	The order of vanishing of $f'(t) ~(\modulo p)$ at $t=0$ is equal to $m$ and the result follows from \Cref{lem:newtpolyg}.
\end{proof}

By summing over all residue discs of $\bF_p$-points lying on a fixed component we obtain the following.

\begin{prop}
	\label{thm:integral-bound-component}
	Let $C \subseteq \cX_{p}^{\sm}$ be a multiplicity-1 component and let $n_C$ be the number of zeros of $\widetilde{\omega}^C$ in $(C^{\sm} \cap \cY)(\bF_p)$ counted with multiplicities. Assume $n_C < p-2$. Then, for any $c\in \bQ_p$, the number of $P \in \cY(\bZ_p)$ reducing onto $C$ and satisfying $\int_{P_0}^P \omega = c$ is at most $\#(C^{\sm} \cap \cY)(\bF_p) + n_C$.
\end{prop}

\begin{proof}
	Every point $P \in \cY(\bZ_p)$ reducing onto~$C$ reduces to a unique $\widetilde{R} \in (C^{\sm} \cap \cY)(\bF_p)$. Let $m_{\widetilde{R}} = \ord_{\widetilde{R}}(\widetilde{\omega}^C)$, so that $n_C = \sum m_{\widetilde{R}}$. The assumption on~$p$ implies $m_{\widetilde{R}} < p-2$, so \Cref{lem:integral-bound-disc} implies that the number of $P \in \cY(\bZ_p)$ reducing to $C$ and satisfying $\int_{P_0}^P \omega = c$ is equal to
	\begin{align*}
		\sum_{\widetilde{R} \in (C^{\sm} \cap \cY)(\bF_p)} \#\left\{ P \in \cY(\bZ_p) \;\Big\vert\; \red_p(P) = \widetilde{R} \text{ and } \int_{P_0}^P \omega = c \right\} \leq \sum_{\widetilde{R} \in (C^{\sm} \cap \cY)(\bF_p)} (m_{\widetilde{R}} + 1),
	\end{align*}
	which is equal to $\#(C^{\sm} \cap \cY)(\bF_p) + n_C$.
\end{proof}

\subsection{Effective bounds on the number of $S$-integral points}
We return to the global situation, %
so $Y = X \smallsetminus D$ is an affine curve over a number field $K$ with regular model $\cY = \cX \smallsetminus \cD$ over~$\cO_K$ which is $D$-transversal over primes in~$S$. We use the results of the previous subsections to prove a Coleman bound on the number of $S$-integral points. %

\begin{thm}\label{thm:zerobound-fixed-reduction-type}
	Let $\Sigma$ be an $S$-integral reduction type such that the Affine Chabauty Condition~\eqref{eq:chabauty-condition-over-number-field} is satisfied. Let $\fp \not\in S $ be a prime with $K_{\fp}=\bQ_p$ and assume $p > 2g+n$.
	Let $C \subseteq \cX_{\fp}$ be the component of the mod-$\fp$ fibre determined by the reduction type~$\Sigma$.
	Then
	\[
	\# \cY(\cO_{K,S})_{\Sigma} \leq \#(C^{\sm} \cap \cY)(\bF_p) + 2g-2+n.
	\]
\end{thm}

\begin{proof}
	By \Cref{thm:existence-of-chabauty-function-nf}, there exist a non-zero log differential~$\omega \in \rH^0(X_{K_{\fp}}, \Omega^1(D))$ and a constant~$c$ such that $\int_{P_0}^P \omega = c$ for all $P \in \cY(\cO_{K,S})_{\Sigma}$. Hence, we have the inclusion
	\[ \cY(\cO_{K,S})_{\Sigma} \subseteq \left\{ P \in \cY(\bZ_p) \;\Big\vert\; \red_p(P) \in C \text{ and } \int_{P_0}^P \omega = c \right\}. \]
	We will bound the size of the latter set using \Cref{thm:integral-bound-component}. Rescaling $\omega$ and~$c$ by a power of~$p$ yields a log differential $\omega^C$ and a constant~$c^C$ such that the restriction $\widetilde{\omega}^C \coloneqq \omega^C\vert_{C^{\sm}}$ is defined and nonzero. Let $n_C$ be the number of zeros of $\widetilde{\omega}^C$ in $(C^{\sm} \cap \cY)(\bF_p)$ counted with multiplicity. \Cref{lem:A4} implies $n_C \leq 2g-2+n$, and this is $< p-2$ by the assumption on~$p$. So \Cref{thm:integral-bound-component} applies and yields
	\begin{equation}
		\label{eq:bound-with-fixed-reduction-type}
		\#\cY(\cO_{K,S})_{\Sigma} \leq \#(C^{\sm} \cap \cY)(\bF_p) + n_C.
	\end{equation}
	The claim now follows from $n_C \leq 2g-2+n$.
\end{proof}

\begin{proof}[Proof of \Cref{thm:bound}]
	The idea is to bound $\cY(\bZ_S)_{\Sigma}$ via \eqref{eq:bound-with-fixed-reduction-type} for each $S$-integral reduction type and sum over all~$\Sigma$. Let us first consider the situation with fixed prime-to-$p$ reduction type $\Sigma' \in \prod_{\ell \neq p} \cC_{\ell}$. %
	Let $\cY(\bZ_S)_{\Sigma'}$ be the set of $S$-integral points whose prime-to-$p$ reduction type is equal to~$\Sigma'$. For a multiplicity-1 component $C \subseteq \cX_{p}$ with $(C^{\sm} \cap \cY)(\bF_p) \neq \emptyset$, write $\Sigma = (\Sigma',C)$ for the reduction type whose prime-to-$p$ part is~$\Sigma'$ and whose $p$-part is $\Sigma_p = C^{\sm} \smallsetminus \cD$. Since $p \not \in S$, the cuspidal part of~$\Sigma$ (\Cref{def:C-Sigma}) is determined by $\Sigma'$. The Affine Chabauty Condition~\eqref{eq:uniform-chabauty-condition-intro} implies \eqref{eq:chabauty-condition-over-number-field}, so by \Cref{thm:existence-of-chabauty-function-nf} there exist a non-zero log differential~$\omega = \omega(\Sigma') \in \rH^0(X_{\bQ_p}, \Omega^1(D))$ and constants $c_{\Sigma} \in \bQ_p$ such that $\int_{P_0}^P \omega = c_{\Sigma}$ for all $P \in \cY(\bZ_S)_{\Sigma}$ of reduction type $\Sigma = (\Sigma',C)$. Let $n_C$ be the number of zeros of $\widetilde{\omega}^C$ in $(C^{\sm} \cap \cY)(\bF_p)$ counted with multiplicity, so we have
	\[ \#\cY(\bZ_S)_{\Sigma} \leq \#(C^{\sm} \cap \cY)(\bF_p) + n_C \]
	for all $\Sigma = (\Sigma',C)$ by \eqref{eq:bound-with-fixed-reduction-type}. Summing over all~$C$ and using \Cref{lem:A4} we obtain %
	\begin{equation}
		\label{eq:bound-fixed-prime-to-p-reduction-part}
		\#\cY(\bZ_S)_{\Sigma'} \leq \sum_C (\#(C^{\sm} \cap \cY)(\bF_p) + n_C) \leq \#\cY^{\sm}_{p}(\bF_p) + 2g - 2 + n.
	\end{equation}
	Finally, since $\#\cY(\bZ_S) = \sum_{\Sigma'} \#\cY(\bZ_S)_{\Sigma'}$, a bound on $\#\cY(\bZ_S)$ is obtained by multiplying the right hand side of~\eqref{eq:bound-fixed-prime-to-p-reduction-part} by the number of all possible prime-to-$p$ reduction parts, which is
	\[ \prod_{\ell \neq p} \# \cC_{\ell} = \prod_{\ell \in S}(\#\cC_{\ell}^0 + \#\cC_{\ell}^1) \prod_{\ell \not\in S \cup \{p\}} \#\cC_{\ell}^0 = \prod_{\ell\in S} (n_\ell+\# (\cX_{{\ell}}^{\sm} \cap \cD)(\bF_\ell)) \prod_{\ell \not\in S \cup \{p\}} n_{\ell}. \qedhere \]
\end{proof}

\begin{rem}
	\label{rem:improved-bound}
	The bound in \Cref{thm:bound} can be improved by observing that different reduction types can lead to the same combination $(\omega,c)$ and hence to the same Chabauty function $\rho$ as in \eqref{eq:rho-function-intro}.
	Namely, let $T$ denote the set of all primes $\ell$ for which the horizontal divisor $\cD$ intersects $\cX_{\ell}$ in more than one component. Note that $T$ is a subset of the set of primes of bad reduction for $\cX$. In particular, $T$ is finite. We can assume that $(\cX_{\ell}^{\sm} \cap \cD)(\bF_{\ell}) \neq \emptyset$ for all $\ell \in S$. If two $S$-integral reduction types $\Sigma$ and $\Sigma'$ have the same cuspidal part and satisfy $\Sigma_\ell=\Sigma'_\ell$ for every $\ell\in T$, then they give rise to the same pair $(\omega,c)$. 
	Indeed, by \Cref{selmer-sets-constant} we then have $\fS_\ell(P_0,\Sigma_\ell)=\fS_\ell(P_0,\Sigma'_\ell)$ for all primes $\ell$, so we get the same $(\omega,c)$ by \Cref{rem:more-red-types-than-eta-c}. %
	Moreover, by \Cref{rem:redundant-reduction-types}, at primes in~$S \smallsetminus T$ any cuspidal reduction type automatically covers all component reduction types, so it is enough to consider $\Sigma$ with $S \smallsetminus T \subseteq C(\Sigma)$. Moreover, by the same remark, at primes $\ell$ in $S \cap T$ only those components not containing a point in $(\cX_{\ell}^{\sm} \cap \cD)(\bF_{\ell})$ need to be considered separately.	This shows that the bound of \Cref{thm:bound} can be improved to
	\begin{align}
		\label{eq:improved-bound}
		\#\cY(\bZ_S) \leq (\#\cY^{\sm}_{p}(\bF_p) + 2g - 2 + n) \sum_{S \smallsetminus T \subseteq S_0 \subseteq S} \; \prod_{\ell\in S_0} \#(\cX_{\ell}^{\sm} \cap \cD)(\bF_\ell) \prod_{\ell \in T\smallsetminus(S_0\cup\{p\})} n_{\ell}'
	\end{align}
	where we set $n_{\ell}' = n_{\ell}$ for $\ell \not\in S$ but define $n_{\ell}'$ as the number of components of~$\cX_{\ell}$ containing at least one smooth $\bF_{\ell}$-point but no smooth $\bF_{\ell}$-point on~$\cD$ for $\ell \in S$. In the case $S = \emptyset$ this simplifies to
	\begin{equation}
		\label{eq:improved-bound-integral-points}
		\#\cY(\bZ) \leq (\#\cY^{\sm}_{p}(\bF_p) + 2g - 2 + n) \prod_{\ell \in T \smallsetminus \{p\}} n_{\ell}'.
	\end{equation}
	We will use the refined bound \eqref{eq:improved-bound-integral-points} in \Cref{thm:even-deg-hyperell-bound}.
\end{rem}

\begin{example}
	\label{ex:genus1-bound}
	Consider the affine curve $\cY \subseteq \bA^2_{\bZ}$ defined by the equation
	\[ y^3 = x^3 + 7x^2 + x. \]
	Our results can be used to bound the number of solutions to this equation in $\bZ[1/q]$ for any fixed prime~$q$. A smooth compactification of $Y = \cY_{\bQ}$ is the rank~1 elliptic curve $X: Y'^2 + 7 Y' Z'^3 = X'^3 - Z'^3$, into which $Y$ embeds via $(x,y) \mapsto [y : 1 : x]$. Let $\cX/\bZ$ be the Weierstrass model defined by the same equation. This model is regular, the primes of bad reduction are $3$ and $5$, and both bad fibres have only one component. In the coordinates $x' = X'/Z'$ and $y' = Y'/Z'$, the divisor of cusps in~$X$ is $D = \{ y' = 0 \}$, consisting of the rational point $Q_1 = (1,0)$ and the quadratic point $Q_2$ with $Q_2(\Qbar) = \{ (\zeta_3,0), (\zeta_3^{-1}, 0)\}$, where $\zeta_3$ is a primitive third root of unity. The horizontal closures $\cQ_1$ and $\cQ_2$ are normal and only intersect in the fibre over~$3$, so the model $\cX$ is $D$-transversal at all primes except~$3$. We take $S = \{q\}$ for a prime $q \neq 3$. (In order to treat the case $q = 3$ one could blow up at the point in the mod-3 fibre where $\cQ_1$ and $\cQ_2$ meet and work with the resulting $D$-transversal model.) Let $\Sigma$ be an $S$-integral reduction type. By \Cref{rem:redundant-reduction-types} it is enough to consider reduction types with $\Sigma_q$ cuspidal. The number of such reduction types is $\#\cD(\bF_q)$, which is~1 or 3, depending on whether $q$ splits in $\bQ(\zeta_3)$. The Affine Chabauty Condition~\eqref{eq:chabauty-condition-intro} is satisfied:
	\[ r + \#C(\Sigma) = 1 + 1 < 1 + 1 + 1 - 1 = g + \#|D| + n_2(D) - 1. \]
	Thus, by \Cref{thm:zerobound-fixed-reduction-type}, there are at most $\#\cY(\bF_p) + 3$ $S$-integral points of reduction type~$\Sigma$, for any prime $p > 2g+n = 5$ with $p \neq q$. Assume that $q \neq 7$, then taking $p = 7$, we have $\#\cY(\bF_p) = 3$, so we obtain the bound $\cY(\bZ[1/q])_{\Sigma} \leq 6$. Multiplying with the number of reduction types which are cuspidal at~$q$, we obtain
	\[ \#\cY(\bZ[\tfrac1{q}]) \leq \begin{cases}
		6, & \text{if $q \equiv -1 \bmod 3$},\\
		18, & \text{if $q \equiv +1 \bmod 3$}.
	\end{cases} \]
\end{example}

\section{Application: even degree hyperelliptic curves}
\label{sec:examples}

\subsection{Recovering Linear Quadratic Chabauty}
\label{sec:recovering-linear-quadratic-chabauty}

Let $\cY/\bZ$ be a hyperelliptic curve defined by an affine equation $y^2 = f(x)$ with $f \in \bZ[x]$ squarefree of degree~$2g+2$. Assume that the leading coefficient of~$f$ is a square in~$\bZ$, so that $Y = \cY_{\bQ}$ has two rational cusps at infinity. Suppose that the Mordell--Weil rank~$r$ equals~$g$. Let~$p$ be a prime of good reduction and fix $P_0 \in \cY(\bZ)$ for simplicity. In \cite{GM:LinearQuadraticChabauty}, the authors use $p$-adic heights to construct a nontrivial function $\rho \colon \cY(\bZ_p) \to \bQ_p$ and a finite set~$T$ such that $\rho(\cY(\bZ)) \subseteq T$. On the other hand, our \Cref{thm:existence-of-chabauty-function-nf} (applied to a desingularised model, see §\ref{sec:general-models}) shows the existence of a log differential $0 \neq \omega \in \rH^0(X_{\bQ_p}, \Omega^1(D))$ and finitely many constants $c_{\Sigma}$ indexed by integral reduction types such that $\int_{P_0}^P \omega = c_{\Sigma}$ holds for all $P \in \cY(\bZ)_{\Sigma}$. It is natural to ask whether we get the same functions, i.e. whether the function $\rho$ of \cite{GM:LinearQuadraticChabauty} is given by the integral of a log differential. This is indeed the case and becomes clear once we recall the construction of~$\rho$. 

For each prime~$\ell$ we have a local Coleman--Gross height pairing $h_{\ell}$ defined on degree zero divisors on $X_{\bQ_{\ell}}$ with disjoint support, taking values in~$\bQ_p$ \cite{CG89}. For $\ell = p$ the pairing depends on a choice of subspace $W \subseteq \rH^1_{\dR}(X/\bQ_p)$ which is complementary to the space of holomorphic differentials and isotropic with respect to the cup product pairing. The choice is used to associate to each degree zero divisor $E$ on $X_{\bQ_p}$ a uniquely determined meromorphic differential $\omega_E$ with at most simple poles and with residue divisor equal to~$E$. The local height pairing at~$p$ is then defined as
\[ h_p(E,F) \coloneqq \int_F \omega_E. \]
For $\ell \neq p$ the height pairing is defined in terms of intersection theory on a regular model of~$X$ over~$\bZ_{\ell}$ by
\[ h_{\ell}(E,F) \coloneqq -\log(\ell)\cdot i_{\ell}(\cE, \Psi_{\ell}(F)) \]
where $\cE$ is the horizontal extension of~$E$ and $\Psi_{\ell}(F) = \cF + \Phi_{\ell}(F)$ is as in §\ref{sec:intersection map}. The global Coleman--Gross height pairing is then defined for degree zero divisors with disjoint support on~$X$ as the sum of the local height pairings over all places:
\[ h(E,F) \coloneqq \sum_{\ell} h_{\ell}(E,F). \]
The global height pairing vanishes when one of the divisors is principal, so it is well-defined on the Jacobian of~$X$. This is not true for the local height pairings, although one can show that when $E$ is supported on cusps, the function $F \mapsto h_{\ell}(E,F)$ is well-defined on the \emph{generalised} Jacobian~$J_Y$.

On a hyperelliptic curve $y^2 = f(x)$ with $\deg(f) = 2g+2$ as above we have two $\bQ$-rational points at infinity. Denote them by $\infty_+$ and $\infty_-$ and fix the degree zero divisor $E \coloneqq \infty_+ - \infty_-$ as first argument of the height pairings. The function $h(E,-)\colon J(\bQ) \to \bQ_p$ is a well-defined homomorphism on the Jacobian. One may assume that the closure of $J(\bQ)$ has finite index in $J(\bQ_p)$, for otherwise one can find a non-zero holomorphic differential whose integral vanishes on all of $X(\bQ)$ by the classical Chabauty method. When the closure of $J(\bQ)$ has finite index in $J(\bQ_p)$, every homomorphism $J(\bQ) \to \bQ_p$ is given by the integral of a holomorphic differential, in particular there exists $\omega_{\hol} \in \rH^0(X_{\bQ_p}, \Omega^1)$ such that
\[ h(\infty_+ - \infty_-, F) = \int_F \omega_{\hol} \quad \text{for all $F \in J(\bQ)$}. \]
At $\ell = p$, the local height $h_p(\infty_+ - \infty_-,-)$ is given by the integral of a logarithmic differential $\omega_{E}$ with residues $1$ and $-1$ at $\infty_+$ and $\infty_-$, respectively. Subtracting the local height at~$p$ from the global height and taking $F \coloneqq P - P_0$, we have
\[ \int_{P_0}^P \omega_{\hol} - \int_{P_0}^{P} \omega_E = \sum_{\ell \neq p} h_{\ell}(\infty_+ - \infty_-, P - P_0) \quad \text{for $P \in Y(\bQ)$}. \]
The expression on the left hand side defines a function $\rho\colon \cY(\bZ_p) \to \bQ_p$, whereas the right hand side takes only finitely many values for $P \in \cY(\bZ)$, corresponding to the possible reduction types of~$P$. 

The function $\rho$ obtained in this way is given as the integral of the log differential $\omega \coloneqq \omega_{\hol} - \omega_{E}$ and thus of the same form as the function obtained from our \Cref{thm:existence-of-chabauty-function-nf}. Apart from working also for primes~$p$ of bad reduction, our approach has the advantage that it yields directly the log differential $\omega$ without the need to determine $\omega_{\hol}$ and $\omega_E$ individually. Note that the splitting of $\omega$ into
$\omega_{\hol}$ and $\omega_E$ depends on the choice of complementary subspace $W \subseteq \rH^1_{\dR}(X/\bQ_p)$ but $\omega$ itself is unique up to scaling, independently of any choices.

The height-based approach can be generalised to more general affine curves admitting a divisor of degree zero supported in the cusps. This is satisfied whenever $D$ contains at least two closed points, and in this case there are $\#|D|-1$ independent such divisors. On the other hand, our Affine Chabauty Condition~\eqref{eq:uniform-chabauty-condition-intro} in the case $r = g$ and $S = \emptyset$ reads $0 < \#|D| + n_2(D) - 1$, so our approach yields more independent Chabauty functions in general and also applies when there is just a single cusp with non-totally real residue field, for example for even-degree hyperelliptic curves $y^2 = f(x)$ where $f$ has negative leading coefficient. One reason for this is that the local height pairing $h_{\ell}$ at a prime $\ell \neq p$ only captures a total intersection number $i_{\ell}(-,-)$ in the mod-$\ell$ fibre of a regular model, whereas our $D$-intersection map captures the intersection numbers $i_x(-,-)$ at individual closed points~$x$, which by \eqref{eq:m-intersection-number} is more refined information. It seems likely that the height-based method corresponds to the Chabauty--Kim method for the quotient of the abelianised fundamental group of~$Y$ described in \cite[Remark~4.3]{LLM:LQChabAffine}, which is an extension of the Tate module of $J$ by a Tate representation, whereas our method corresponds to the full abelianisation, which is an extension of the Tate module by an \emph{Artin--Tate} representation. In line with Dogra's terminology \cite{dogra:generalised-height-pairings}, one might call the weaker variant ``skimmed linear Chabauty''.

On the other hand, using the methods of this paper we cannot handle integral points on hyperelliptic curves of \emph{odd} degree with $r = g$ as in \cite{BBM16}. In this case there is only one point at infinity and the proof of Theorem~3.1 of loc.\ cit.\ uses the function $x \mapsto h(x-\infty,x-\infty)$ in which both arguments of the height pairing depend on~$x$, leading to functions involving not just single but double Coleman integrals. An ``Affine Quadratic Chabauty'' method should be able to recover the results of \cite{BBM16} as a special case.

\subsection{A bound for integral points on hyperelliptic curves over $\bQ$}
\label{sec:split-hyperelliptic}

Consider a hyperelliptic curve $\cY/\bZ$ defined by an affine equation $y^2 = f(x)$ with $f \in \bZ[x]$ monic and squarefree of degree~$2g+2$ and suppose that the Mordell--Weil rank~$r$ equals~$g$. As discussed in §\ref{sec:recovering-linear-quadratic-chabauty}, we obtain a non-zero log differential whose integral takes only finitely many values on integral points. We show that in many cases it takes just one value. With the results from §\ref{sec:bounds} this leads to a good bound on $\#\cY(\bZ)$.

\begin{thm}
	\label{thm:even-deg-hyperell-bound}
	Let $f=x^{2g+2} + f_{2g+1}x^{2g+1}+\dots+f_0\in\bZ[x]$ be squarefree with $2g+2\geq 4$. %
	Assume that the reduction of $f$ modulo every odd prime is not a square.
	Assume that either the coefficient $f_{2g+1}$ is odd, or that $f=4P+Q^2$ with $P,Q\in\bZ[x]$ and $Q$ monic of degree $g+1$ such that the reduction of $y^2 + Q(x)y - P(x) \in \bZ[x][y]$ modulo $2$ does not have a zero $y_0\in\bF_2[x]$ with $\deg(y_0)\leq g+1$.
	Let $p>2g+2$ be a prime such that the reduction of $f$ modulo $p$ remains separable.
	Assume that the hyperelliptic curve given by $y^2 = f(x)$ over~$\bQ$ has Mordell--Weil rank equal to $g$.
	Then the affine curve
	\begin{align}\label{eq:hyp-curve}
		\cY \colon\quad y^2 = f(x)
	\end{align}
	has at most $\# \cY(\bF_p) + 2g$ integral points.
\end{thm}

We remark that these technical assumptions on $f$ modulo all primes can be checked very quickly.

\begin{proof}
	Let $\cX/\bZ$ be the projective scheme given by the affine equation \eqref{eq:hyp-curve} as in \cite[Definition 1.1]{liu:minmodels}, i.e.\ we work in weighted projective space with weights $(1,g+1,1)$.
	Let $X$ be its generic fibre and let $D=\{\infty_+,\infty_-\}$ consist of the two points $\infty_{\pm} \coloneq (1:\pm 1:0)$ at infinity, which are $\bQ$-rational.
	Let $Y=X\smallsetminus D$ and $\cY=\cX\smallsetminus\cD$ as usual.
	We are interested in $\cY(\bZ)$.
	Assume there exists a rational point $P_0\in Y(\bQ)$, otherwise there is nothing to prove.
	Then $Y$ satisfies \eqref{eq:uniform-chabauty-condition-intro} for empty $S$, but we cannot directly apply \Cref{thm:bound} as $\cX$ might not be regular.
	Instead, we need to work with a desingularisation of $\cX$ as in \S\ref{sec:general-models}.
	The assumptions on $f$ imply that
	\begin{enumerate}
		\item $\cX_p$ is smooth;
		\item for odd $\ell$, the special fibre $\cX_\ell$ is irreducible and $\infty_{\pm}$ reduce to smooth points of $\cX_\ell$, and the scheme $\cX$ is normal at the points of $\cX_\ell$ by \cite[Lemma 2.2]{liu:minmodels}.
		In this case, we say $\cX$ satisfies $(\ast)$ at $\ell$;
		\item if $f_{2g+1}$ is odd, then $\cX$ satisfies $(\ast)$ at $2$ by \cite[Lemma 2.3(1)(c)]{liu:minmodels};
		\item otherwise, let $\cX'$ be given by the affine equation
		\[
			\cX'\colon\quad y^2 + Q(x)y = P(x)
		\]
		and $\cX'\to\cX$ be given by $(x,y)\mapsto(x,2y+Q(x))$.
		Then $\cX'$ satisfies $(\ast)$ at $2$ by \cite[Lemma 2.3(1)(a)]{liu:minmodels}.
	\end{enumerate}
	Let $\cX'$ be the normalisation of $\cX$.
	Then $\cX'$ satisfies $(\ast)$ at all primes $\ell$.
	Now let $\cX''$ be the strong desingularisation of $\cX'$ as in \cite[Theorem 9.3.44]{liu2006algebraic}.
	We have $\cY(\bZ)\subseteq\cY''(\bZ)$ and $\cY_p\cong\cY''_p$ by (1).
	As $\cX''\to\cX'$ is an isomorphism away from the singular locus, we see that $\infty_{\pm}$ reduce to the same component of $\cX''_\ell$ at every prime $\ell$.
	So $T$ is empty in the notation of \Cref{rem:improved-bound} and the bound \eqref{eq:improved-bound-integral-points} applied to $\cY''$ yields the result.
\end{proof}

\begin{rem}
	The condition on~$p$ in \Cref{thm:even-deg-hyperell-bound} can be slightly weakened: instead of requiring the Weierstraß model $\cX$ defined by $y^2 = f(x)$ to have good reduction at~$p$ it is enough that $f$ is not a square modulo~$p$ and $\cX$ is regular at all $\bF_p$-points in~$\cY$. One then has the bound $\#\cY(\bZ) \leq \#\cY_p^{\sm}(\bF_p) + 2g$.
\end{rem}

The bound of \Cref{thm:even-deg-hyperell-bound} is optimal, as the following examples demonstrate.

\begin{example}\label{ex:g1sharp}
	($g=1$) Let $f(x) = x^4 + 5x^3 + 6x^2 + x$.
	It is square-free, has discriminant $49$, the coefficient $f_{2g+1}=5$ is odd, and $f\equiv x(x+4)^3 \pmod 7$ is not a square.
	The Jacobian of the hyperelliptic curve defined by $y^2 = f(x)$ is the elliptic curve $v^2 = 1 + 5t + 6t^2 + t^3$, which is isomorphic to the elliptic curve with LMFDB label \href{https://www.lmfdb.org/EllipticCurve/Q/392/a/1}{392.a1} and has rank $1$.
	Take $p=5$.
	Then $\cY\colon y^2 = f(x)$ satisfies $\#\cY(\bF_5)=3$, so $\cY$ has at most $5$ integral points by \Cref{thm:even-deg-hyperell-bound}.
	A quick search finds the $5$ points $(0,0)$, $(-1,\pm 1)$, $(4,\pm 26)$, thus $\cY(\bZ) = \{(0,0),(-1,\pm 1),(4,\pm 26)\}$.
\end{example}

\begin{example}\label{ex:g2sharp}
	($g=2$) Let $f(x) = x^6 -28x^2+4$.
	It is square-free, has discriminant $-2^{16}\cdot 43^2\cdot 127^2$ and $f$ is not a square modulo $43$ or $127$.
	One readily checks the condition at $2$ with $Q=x^3$ and $P=1-7x^2$.
	The hyperelliptic curve defined by $y^2 = f(x)$ is bielliptic, hence its Jacobian $J$ is isogenous to the product of the two elliptic curves $y^2 = x^3 - 28x + 4$ and $y^2 = x^3 - 28x^2 + 16$ \cite[p.216]{BP:bielliptic}.
	These are isomorphic to the elliptic curves with LMFDB labels \href{https://www.lmfdb.org/EllipticCurve/Q/43688/a/1}{43688.a1} and \href{https://www.lmfdb.org/EllipticCurve/Q/5461/a/1}{5461.a1} respectively, which both have rank $1$, so $J(\bQ)$ has rank $2$.
	Take $p=7$.
	Then $\cY\colon y^2 = f(x)$ satisfies $\#\cY(\bF_7)=2$, hence \Cref{thm:even-deg-hyperell-bound} gives the bound $\#\cY(\bZ) \leq 6$.
	A quick search finds the six points $(0,\pm 2)$, $(7,\pm 341)$, $(-7,\pm 341)$, thus $\cY(\bZ) = \{(0,\pm 2), (7,\pm 341), (-7,\pm 341)\}$.
\end{example}

\subsection{Hyperelliptic curves over imaginary quadratic fields}
\label{sec:hyperelliptic-over-imaginary-quadratic}

Let $K$ be an imaginary quadratic field with ring of integers~$\cO_K$, and let $Y/K$ be an affine hyperelliptic curve defined by 
\begin{equation}
	\label{eq:hyperelliptic-imaginary-quadratic}
	Y\colon\quad y^2 = f(x)
\end{equation}
where $f(x) \in \cO_K[x]$ is a squarefree polynomial of degree $2g+2 \geq 4$ whose leading coefficient is a square in $\cO_K$. Let $\cY \subseteq \bA^2_{\cO_K}$ be the model over~$\cO_K$ defined by the same equation. Write $Y = X \smallsetminus D$ and $\cY = \cX \smallsetminus \cD$ with~$X$ and~$\cX$ given by taking the closure in the weighted projective space $\bP(1,g+1,1)$. Assume that $r = \rank J(K)$ is equal to ~$g$. Since the degree of~$f$ is even and the leading coefficient is a square, there are two points at infinity, which are both $K$-rational. Thus \eqref{eq:chabauty-condition-over-number-field} is satisfied for empty $S$ and we may apply \Cref{thm:existence-of-chabauty-function-nf} to determine the $\cO_K$-points of~$\cY$. Note that the model~$\cX/\cO_K$ defined by the equation~\eqref{eq:hyperelliptic-imaginary-quadratic} might not be regular, in which case we have to pass to a desingularisation $\cX' \to \cX$ in order to specify the reduction types as explained in \Cref{sec:general-models}.
Here we only use a single prime~$\fp$; %
one could also use all primes lying over a fixed rational prime as in \Cref{thm:ros-chabauty} and \cite[§5]{GM:LinearQuadraticChabauty}. %

As a concrete illustration of our \Cref{thm:existence-of-chabauty-function-nf}, we determine the $\bZ[\zeta_3]$-points on the genus~2 hyperelliptic curve with LMFDB label \href{https://www.lmfdb.org/Genus2Curve/Q/1549/a/1549/1}{1549.a.1549.1} defined by
\begin{equation}
	\label{eq:ex-imag}
	Y \colon\quad y^2 = x^6 - 4x^5 + 2x^4 + 6x^3 + x^2 - 10x + 1.
\end{equation}

\begin{thm}
	\label{thm:imaginary-example}
	The set $\cY(\bZ[\zeta_3])$ is equal to
	\begin{equation}
		\left \{ (0, \pm 1), \; (2, \pm 1), \; (1, \pm \sqrt{-3}) \right\}.
	\end{equation}
\end{thm}

It is convenient to choose an auxiliary prime~$\fp$ with $K_{\fp} = \bQ_p$. %
We choose $\fp \coloneqq (2 - \sqrt{-3})$ which divides the rational prime~$7$. The Jacobian has rank~1 over~$\bQ$ but rank~$2$ over~$K = \bQ(\zeta_3)$. For the lower bound $\rank J(K) \geq 2$ it suffices to check that
\[ \det \begin{pmatrix}
	\int_{P_0}^{P_1} \omega_0 & \int_{P_0}^{P_1} \omega_1 \\[2mm]
	\int_{P_0}^{P_2} \omega_0 & \int_{P_0}^{P_2} \omega_1
\end{pmatrix} \neq 0, \]
where $P_0 = (0,1)$, $P_1 = (2,1)$, $P_2 = (1,\sqrt{-3})$, and $\omega_j = x^j \frac{\rd x}{y}$. The non-vanishing of the determinant implies that the divisors $[P_1] - [P_0]$ and $[P_2] - [P_0]$ are linearly independent in the Jacobian. The upper bound $\rank J(K) \leq 2$ can be obtained by an application of the \texttt{RankBound} function in Magma. The model~$\cX_0$ over~$\bZ$ given by \eqref{eq:ex-imag} has bad reduction at~2 and~1549. The model~$\cX'_0$ with the minimal equation
\[ y^2 + (x^3 + x + 1)y = -x^5 + x^3 - 3x \]
has good reduction at~2 and the map $\cX_0' \to \cX_0$, $(x,y) \mapsto (x, 2y + x^3 + x + 1)$ is an isomorphism over all odd primes, so the only bad prime for this model is $q = 1549$. An application of the \texttt{RegularModel} function in Magma shows that $\cX_0'$ is regular and that the fibre over $q$ has only one component. Let $\cX' \to \cX$ be the base change of~$\cX_0' \to \cX_0$ from~$\bZ$ to~$\cO_K$. Then the only bad fibres of~$\cX'$ are at the two primes of~$\cO_K$ dividing~$q$, and for $\fq\mid q$ we have $\cO_{K,\fq}\cong\bZ_q$ and thus $\cX \otimes \cO_{K,\fq} \cong \cX_0 \otimes \bZ_q$.
Hence $\cX'$ is regular and the fibres over~$\fq\mid q$ still have only one component. Therefore, all $\cO_K$-points have the same reduction type with respect to the model~$\cX'$.
\Cref{thm:existence-of-chabauty-function-nf} implies the existence of a nonzero log differential $\omega \in \rH^0(X_{K_{\fp}},\Omega^1(D))$ such that the function %
\[ \cY(\cO_{\fp}) \to K_{\fp},\quad P \mapsto \int_{P_0}^P \omega \]
vanishes on $\cY(\cO_K)$. There is no constant in the Chabauty function since the base point $P_0 = (0,1)$ is itself an $\cO_K$-integral point. The annihilating differential~$\omega$ is a nontrivial linear combination of the three basis differentials $\omega_0, \omega_1, \omega_2$ of~$\rH^0(X, \Omega^1(D))$, where $\omega_j = x^j \frac{\rd x}{y}$. Using the known integral points $P_0,P_1,P_2$, we find that the function $\rho\colon \cY(\cO_{\fp}) \to K_{\fp}$ defined by
\begin{equation}
	\label{eq:ex-imag-determinant-equation}
	\rho(P) \coloneqq \det \begin{pmatrix}
	\int_{P_0}^{P_1} \omega_0 & \int_{P_0}^{P_1} \omega_1 & \int_{P_0}^{P_1} \omega_2 \\[2mm]
	\int_{P_0}^{P_2} \omega_0 & \int_{P_0}^{P_2} \omega_1 & \int_{P_0}^{P_2} \omega_2 \\[2mm]
	\int_{P_0}^P \omega_0 & \int_{P_0}^P \omega_1 & \int_{P_0}^P \omega_2
\end{pmatrix}
\end{equation}
vanishes on $\cY(\cO_K)$. We remark that it does not necessarily vanish on $K$-rational points. For example for $P = (33/20, 4073/8000)$, which is not integral at~$2$ and~$5$, one has $\rho(P) \neq 0$ in~$\bQ_7$. Expanding the determinant in~\eqref{eq:ex-imag-determinant-equation} along the last row, we see that an annihilating log differential is given by $\omega = \alpha_0 \omega_0 + \alpha_1 \omega_1 + \alpha_2 \omega_2$ with coefficients
\begin{align*}
\alpha_0 &= 6\cdot 7^2 + 6\cdot 7^3 + 3\cdot 7^4 + 7^5 + 5\cdot 7^6 + 4\cdot 7^7 + O(7^8),\\
\alpha_1 &= 3\cdot 7^4 + 3\cdot 7^5 + 5\cdot 7^6 + 5\cdot 7^7 + O(7^8),\\
\alpha_2 &= 6\cdot 7^2 + 5\cdot 7^3 + 3\cdot 7^4 + 6\cdot 7^5 + 3\cdot 7^7 + O(7^8).
\end{align*}
We can analyse the zeros of~$\rho$ on each residue disc by computing the power series expansion in a uniformising parameter. For example, consider the point $P_0 = (0,1)$. Its residue disc is parametrised by $Q(t) = (x(t), y(t))$ with $x(t)=7t$, $t \in \bZ_7$, and
\begin{align*}
	y(t) = \sqrt{f(x(t))} = \sum_{n=0}^\infty \binom{1/2}{n}(f(7t)-1)^n = 1 + \frac12 (f(7t)-1) - \frac14 (f(7t)-1)^2 + \ldots
\end{align*}
In terms of the parameter~$t$, the function~$\rho(Q(t)) = \sum_{j=0}^2 \alpha_j \int_{P_0}^{Q(t)} x(t)^j \frac{\rd x(t)}{y(t)}$ is given by
\begin{align*}
	&(7^3 + 3\cdot 7^5 + 5\cdot 7^6 + O(7^7))t + (6\cdot 7^4 + 3\cdot 7^5 + 2\cdot 7^6 + O(7^7))t^2 + (7^5 + 2\cdot 7^6 + O(7^7)) t^3 + \ldots %
\end{align*}
The sequence of valuations of the coefficients starts $(\infty, 3, 4, 5, 6, 11, 8, \ldots)$. Using the fact that the basis differentials $\omega_i$ have power series expansions in $7\bZ_7[\![7t]\!] \rd t$, one can bound the subsequent valuations by~$8$ from below. Hence, by Strassmann's Theorem, the only zero of $\rho(Q(t))$ in~$\bZ_7$ is at $t = 0$. This shows that $P_0 = (0,1)$ is the only $\cO_K$-point of~$\cY$ in its residue disc. Each of the remaining residue discs contains exactly one of the known $\cO_K$-points, and the same argument shows that in each case, there are no other $\cO_K$-points in that disc, proving \Cref{thm:imaginary-example}.

The Sage code for this example can be found at \url{https://github.com/martinluedtke/AffChab1}.

\printbibliography
	
\end{document}